\documentclass[11pt]{amsart}
\usepackage{amssymb}
\usepackage{latexsym}
\usepackage{graphicx}
\usepackage{subfigure}
\usepackage{tikz}
\usepackage[linktocpage=true]{hyperref}
\usepackage{ulem}
\usepackage[left=3.8cm, right=3.8cm, top=3cm, bottom=3cm]{geometry}
\usepackage{color}
\usepackage{here}
\usepackage[backgroundcolor = lightgray,
            linecolor = lightgray,
            textsize = tiny]{todonotes}
\usepackage{stackengine}

\hypersetup{ colorlinks   = true, 
urlcolor  = blue, 
linkcolor    = blue, 
citecolor   = blue 
}

\newcommand{\resp}{resp.\ }

\def\beq{\begin{equation}}
\def\eeq{\end{equation}}

\def\det{\mathrm{det}\ }

\newcommand{\Z}{{\mathbb Z}}
\newcommand{\R}{{\mathbb R}}

\newcommand{\T}{{\mathbb T}}

\newtheorem{theorem}{Theorem}[section]
\newtheorem{remark}[theorem]{Remark}

\newtheorem{lemma}[theorem]{Lemma}

\newtheorem{prop}[theorem]{Proposition}
\newtheorem{corollary}[theorem]{Corollary}

\newtheorem{theoalph}{Theorem}

\newtheorem{coralph}[theoalph]{Corollary}

\newcommand{\dom}{\mathcal D}
\newcommand{\obs}{\mathcal O}

\begin{document}


\title[Marked Length Spectrum and the geometry of open dispersing billiards]
{
  Marked Length Spectrum, homoclinic orbits and the geometry of open dispersing billiards}

\author[P\'eter B\'alint]{P\'eter B\'alint$^*$}
\thanks{$^*$P.B. is supported in part by Hungarian National Foundation for Scientific Research (NKFIH OTKA) grants K104745 and K123782.}
\address{$^*$MTA-BME Stochastics Research Group, Budapest University of Technology and Economics, Egry J\'ozsef u.\,1, H-1111 Budapest, Hungary, and Department of Stochastics, Institute of Mathematics, Budapest University of Technology and Economics, Egry J\'ozsef u.\,1, H-1111 Budapest, Hungary}
\email{pet@math.bme.hu}

\author[Jacopo De Simoi]{Jacopo De Simoi$^\dagger$}
\thanks{$^\dagger$J.D.S. and M.L. are supported by the NSERC Discovery grant, reference number 502617-2017.}
\address{$^{\dagger,\S}$Department of Mathematics, University of Toronto, 40 St George St., Toronto, ON, Canada M5S 2E4}
\email{jacopods@math.utoronto.ca}
\email{martin.leguil@utoronto.ca}

\author[Vadim Kaloshin]{Vadim Kaloshin$^\ddagger$}
\thanks{$^\ddagger$ V.K. acknowledges partial support of the
NSF grant DMS-1402164.}
\address{$^\ddagger$University of Maryland, College Park, MD, USA}
\email{vadim.kaloshin@gmail.com}

\author[Martin Leguil]{Martin Leguil$^\S$}

\begin{abstract}
  We consider billiards obtained by removing three strictly convex
  obstacles satisfying the non-eclipse condition on the plane.  The
  restriction of the dynamics to the set of non-escaping orbits is
  conjugated to a subshift on three symbols that provides a natural
  labeling of all periodic orbits.  We study the following inverse
  problem: does the Marked Length Spectrum (i.e., the set of lengths of
  periodic orbits together with their labeling), determine the geometry
  of the billiard table?  We show that from the Marked Length Spectrum
  it is possible to recover the curvature at periodic points of period
  two, as well as the Lyapunov exponent of each periodic orbit.
\end{abstract}

\maketitle

\tableofcontents

\section*{Introduction}

In this paper, we study the spectral rigidity of a class of chaotic billiards obtained by removing $m \geq 3$ strictly convex obstacles from the plane $\R^2$.\footnote{For brevity, in most of the exposition we restrict ourselves to the case $m=3$, yet, all of the results apply to arbitrary $m\geq 3$.}  We assume that the  boundary of each obstacle is at least of class $C^3$, and that the obstacles satisfy some non-eclipse condition. Similar billiard tables were already considered for instance in  \cite{GR}, where the authors studied the classical scattering of a point particle from three hard circular discs in a plane. In \cite{S}, Stoyanov considered billiard trajectories in the exterior of two strictly convex domains in the plane, and obtained an asymptotic for the sequence of travelling
times  involving the distance between the two domains and the curvatures at the ends of the associated period two orbit. Yet, our perspective is dual to those works, in the sense that we focus on periodic orbits (in particular, such orbits do not escape to infinity), while in their case,
 they studied the deviation of escaping trajectories due to the presence of the obstacles.

By the strict convexity of the obstacles, the class of billiard tables we consider has
some hyperbolic properties. Typically, hyperbolic systems have a lot of periodic orbits,
in connection with the famous Anosov closing lemma, and it is thus natural to expect that
the periodic data give a lot of information on the system. In this paper, we focus
on the information given by the length of all periodic orbits. More precisely, by
the chaoticity of the billiard and the non-eclipse condition we require, there is
an embedded subshift of finite type which provides a natural labeling of the periodic orbits.
Then, the \textit{Marked Length Spectrum}  is defined as the set of all lengths of periodic
orbits together with their marking. The question we want to address is the following:
\begin{center}
{\it How much geometric information on the billiard table does the Marked Length Spectrum convey?
In particular, are two such tables with the same Marked Length Spectrum necessarily isometric?
}
\end{center}
\medskip

\textbf{The problem of spectral rigidity.}
The problem of spectral rigidity has been investigated in various settings. Let us recall
that it is connected with the famous problem of M. Kac \cite{K}:
\textit{``Can you hear the shape of a drum?''}, i.e., whether the shape of a planar domain
is determined by the Laplace Spectrum. For instance, Andersson-Melrose \cite{AM},
generalizing previous results, showed that for strictly convex $C^\infty$ domains, there exists
a remarkable relation between the singular support of the wave trace and the Length Spectrum.
In particular, the Laplace Spectrum determines the Length Spectrum in this setting.  Similarly,
there is a connection between Laplace Spectrum and Length Spectrum in hyperbolic situations.
Indeed, the Selberg trace formula shows that the Laplace Spectrum determines the Length
Spectrum on hyperbolic manifolds, and for generic Riemannian metrics, the Laplace Spectrum
determines the Length Spectrum. \medskip

\textbf{The Marked Length Spectrum determines a hyperbolic surface up to isometry.}
For hyperbolic surfaces, there is a natural marking of periodic trajectories by the homology,
obtained as follows: to each homology class, we associate the length of the closed geodesic
in this class. The question of spectral rigidity for hyperbolic surfaces was addressed by
Otal \cite{O} and independently by Croke \cite{Cr}: they showed that if $g_0$ and $g_1$
are negatively curved metrics on a closed surface $S$ with the same Marked Length Spectrum,
then $g_1$ is isometric to $g_0$ (see \cite{CS} for the multidimensional case). More recently, Guillarmou-Lefeuvre \cite{GL} proved that in all dimensions, the Marked Length Spectrum of a Riemannian manifold
$(M, g)$ with Anosov geodesic flow and non-positive curvature locally determines the metric.
\medskip

\textbf{Spectral rigidity for (convex) real analytic domains with symmetries.}
Let us now mention a few results related to the question of spectral ridigity for
convex billiards. It has been famously proven by Zelditch  (see \cite{Z1,Z2,Z3}) that
the inverse spectral problem has a positive answer in the case of a generic
class of real analytic $\Z_2$-symmetric plane convex domains (i.e., symmetric with
respect to some reflection); in other terms, the Spectrum determines the geometry of such domains. Hezari-Zelditch \cite{HZ2}  have obtained a higher dimensional analogue of this result:  bounded analytic domains ­in $\R^n$ with
reflection symmetries across all coordinate axes, and with one axis height fixed (satisfying some generic non-degeneracy conditions) are spectrally determined among other such domains. Results of this kind are, currently, far beyond reach in the smooth category, although,
in the last decade, interesting results have started to appear for the, weaker, spectral
rigidity properties. In  \cite{HZ1}, Hezari-Zelditch have  shown the following result in the direction of  the question of spectral ridigity: given a domain   bounded by an ellipse, then any one-parameter
isospectral $C^\infty$ deformation which additionally preserves the
$\Z_2\times \Z_2$ symmetry group of the ellipse is necessarily
flat (i.e., all derivatives
have to vanish at the initial parameter).

On the other hand, Colin de Verdi\`ere has studied the   dynamical version of the inverse spectral problem for billiards with symmetries: in \cite{CdV},  he has shown that in the case of convex analytic billiards  which have the symmetries of an ellipse, the Marked Length Spectrum determines completely the geometry. In \cite{DSKW}, the authors have obtained the following result about the question of dynamical spectral rigidity: any sufficiently  smooth $\Z_2$-symmetric strictly convex domain sufficiently close to a circle is dynamically spectrally rigid, i.e., all deformations among domains in the same class which preserve the length of all periodic orbits of the associated billiard flow must necessarily be isometries.

\medskip

\textbf{Lyapunov exponents of periodic orbits and the Marked Length Spectrum.}
In another direction, Huang-Kaloshin-Sorrentino  \cite{HKS}  have proved   that for a generic strictly convex domain, it is possible to recover the eigendata corresponding to Aubry-Mather periodic orbits of the induced billiard map, from the (maximal) Marked Length Spectrum of the domain. Here, the marking is defined by associating any rational number $r$ in $[0,\frac 12]$ with the maximum among all the perimeters  of periodic orbits with rotation number $r$.
\medskip

\textbf{Flat billiards.}
Let us also mention that rigidity questions have also been explored in the context of flat billiards. In the recent paper \cite{DELS}, the authors define a Bounce Spectrum
for polygons, recording the symbolic dynamics of the billiard flow, in the same way as the symbolic coding we introduce in the present paper.
Then, they have shown that two simply connected Euclidean polygons with the same Bounce Spectrum  are either similar or right-angled and affinely equivalent. \\

\textbf{Dispersing billiards.}
A dual formulation of the inverse spectral problem is the inverse
resonance problem, in which one attempts to reconstruct an unbounded domain
(e.g. the complement of a finite number of convex scatterers) by the resonances (i.e. the
poles) of the resolvent $(\Delta-z^2)^{-1}$ (see e.g. \cite{PS2,Zw,Z4}). From the dynamical point of
view, these systems are described by the theory of \textit{Dispersing Billiards}, which is today
a very active topic in dynamical systems.

In this paper, we are exploring the inverse dynamical
problem in this setting. We show that for the class of billiards introduced above, it is possible to reconstruct some geometric information from the Marked Length Spectrum, namely the radius of curvature at the bouncing points of period two orbits. Contrary to some of the results recalled above, we do not assume any additional symmetries. In fact, we use asymmetric orbits between the obstacles to recover separately the curvatures.  In the same direction as in \cite{HKS}, we also show that for general periodic orbits, it is possible to adapt part of the procedure we introduce in the case of period two orbits in order to recover the Lyapunov exponent of each periodic orbit (see the paragraph before \eqref{maked lepnov} for the  definition of this notion). 
Yet, it seems that for general periodic orbits, we don't have enough asymmetric information between the points in the orbit to distinguish between them on a differentiable level, and recover the curvature at each point separately. We continue the study of higher order periodic orbits in a separate work \cite{BDSKL}, where we discuss that by an argument of Livsic type, the information on Lyapunov exponents obtained here are sufficient to recover the differential of the billiard map, up to H\"older conjugacy, and we show that additional geometric quantities can be reconstructed from the Marked Length Spectrum. On the other hand, in \cite{DSKL}, we investigate  the case of
billiards of the same type as above, but with analytic boundary, and we show that   if the billiard table has some symmetries, the square of the billiard map itself can be entirely reconstructed from the Marked Length Spectrum, as well as the geometry.


\section{Preliminaries}

\subsection{Symbolic coding and Marked Length Spectrum}

In this paper, we consider billiard tables $\dom\subset \R^2$ given by
$\dom = \R^{2}\setminus \bigcup^{3}_{i = 1}\obs_{i}$ where each
$\obs_{i}$ is a convex domain with boundary
$\partial\obs_{i} = \Gamma_i$, which we assume to be sufficiently smooth
(of class at least $C^3$).  We refer to each of the $\obs_{i}$'s as
\emph{obstacle} or \emph{scatterer}.  We let $\ell_i:=|\Gamma_i|$ be the
respective lengths, and we parametrize each $\Gamma_i$ in arc-length, for some 
smooth map $\gamma_i \colon \T_i:=\R/\ell_i\Z\to \R^2$, $s\mapsto \gamma_i(s)$. We
assume that the non-eclipse condition is satisfied, i.e., the convex
hull of any two scatterers is disjoint from the third one.

We denote the collision space by
\begin{align*}
  \mathcal{M} &= \bigcup_i \mathcal{M}_i, &
  \mathcal{M}_i &=\{(q,v),\ q \in \Gamma_i,\ v\in \R^2,\ \|v\|=1,\ \langle v,n\rangle\geq 0\},
\end{align*}
where $n$ is the unit normal
vector to $\Gamma_i$ pointing inside $\mathcal{D}$. For each
$x=(q,v) \in \mathcal{M}$, $q$ is associated with the arclength
parameter $s \in [0,\ell_i]$ for some $i\in \{1,2,3\}$, i.e.,
$q=\gamma_i(s)$, and we let $\varphi\in [-\frac{\pi}{2},\frac{\pi}{2}]$
be the oriented angle between $n$ and $v$. In other terms, each $\mathcal{M}_i$ can be seen as
a cylinder $\T_i \times [-\frac{\pi}{2},\frac{\pi}{2}]$ endowed with
coordinates $(s,\varphi)$. In the following, given a point $x=(q,v) \in \mathcal{M}$ associated with the pair $(s,\varphi)$, we also denote by
$\gamma(s):=q$ the point of the table defined as the projection of $x$ onto   the $q$-coordinate.
Set $\Omega :=\{(q,v) \in \mathcal{D} \times \mathbb{S}^1\}$. We denote by $\Phi^t\colon \Omega\to \Omega$ the flow of the billiard
and let
$$
\mathcal{F} \colon \mathcal{M} \to \mathcal{M},\quad x \mapsto \Phi^{\tau(x)+0}(x)
$$
be the associated billiard map, where
$\tau\colon\mathcal{M} \to \R_+\cup \{+\infty\}$ is the first return time.
For each pair $(s,\varphi), (s',\varphi') \in \mathcal{M}$, we denote by
\begin{equation}\label{def hsspremie}
h(s,s'):=\|\gamma(s)-\gamma(s')\|
\end{equation}
the length of the segment connecting the associated points of the table.

By the convexity of the obstacles, for each $i \in\{1,2,3\}$, for $* \in \{j,k\}$, with $\{i,j,k\}=\{1,2,3\}$,
there exist $0\leq a_i^*\leq b_i^*\leq \ell_i$, and for each parameter $s\in [a_i^*,b_i^*]$, there exists a non-empty closed interval $I_i^*(s) 
\subset [-\frac{\pi}{2},\frac{\pi}{2}]$ 
such that $\tau(x)<+\infty$, if
$x=(s,\varphi) \in
\widetilde{\mathcal{M}}_i:=\widetilde{\mathcal{M}}_i^j\cup\widetilde{\mathcal{M}}_i^k$,
and $\tau(x)=+\infty$, if
$x \in \mathcal{M}_i \backslash
\widetilde{\mathcal{M}}_i$, 
where  
$$
\widetilde{\mathcal{M}}_i^*
:=\{(s,\varphi)\in \mathcal{M}_i: s\in [a_i^*,b_i^*],\ \varphi \in I_i^*(s)\}=\mathcal{M}_i \cap \mathcal{F}^{-1} (\mathcal{M}_*).
$$
In particular, the set of trajectories that do not escape to infinity is given by
$$
\bigcap_{j \in \Z} \mathcal{F}^{-j}(\widetilde{\mathcal{M}}),\quad \widetilde{\mathcal{M}}:=\widetilde{\mathcal{M}}_1\cup \widetilde{\mathcal{M}}_2\cup \widetilde{\mathcal{M}}_3,
$$
and is homeomorphic to a Cantor set.\footnote{For more details about this fact, we refer the reader to \cite{M} or to Section III in \cite{GR}. In \cite{GR}, the authors consider the case where the obstacles are round discs, but the same construction can be carried out for strictly convex obstacles.} In restriction to this set, the dynamics  is conjugated to a subshift of finite type associated with the transition matrix $$\begin{pmatrix}
0 & 1 & 1\\
1 & 0 & 1 \\
1 & 1 & 0
\end{pmatrix}.$$ In other terms, any word $(\rho_j)_{j}\in \{1,2,3\}^\Z$ such that $\rho_{j+1}\neq \rho_j$ for all $j \in \Z$ can be realized by an orbit, and by expansiveness of the dynamics, this orbit is unique. Such a word is called \textit{admissible}. Besides, this marking is unique provided that we fix a starting point in the orbit and an orientation.

In particular, any periodic orbit of period $p \geq 2$ can be labelled by a finite admissible word $\sigma=(\sigma_1\sigma_2\dots\sigma_p)\in \{1,2,3\}^p$, i.e., such that the infinite word $\sigma^\infty:=\dots\sigma\sigma\sigma\dots$ is admissible, or equivalently, such that $\sigma_1 \neq \sigma_p$ and $\sigma_{j}\neq \sigma_{j+1}$, for all $j\in \{1,\dots,p-1\}$.
We also let  $\overline{\sigma}$ be the transposed word
$$
\overline{\sigma}:=(\sigma_p \sigma_{p-1}\dots\sigma_1).
$$
It still encodes the same periodic trajectory as $\sigma$, but traversed backwards.

As explained above, for any $j \in \{1,\dots,p\}$, the $j^{\text{th}}$
symbol $\sigma_j$ of $\sigma$ corresponds to a point $x(j)$ in the
trajectory, where $x(j)=(s(j),\varphi(j))$ is represented by a position
and an angle coordinates. For all $k \in \Z$, we also extend the
previous notation by setting $\sigma_k:=\sigma_{\overline{k}}$, with
$\overline{k}:=k\text{ (mod } p)$, and similarly for $x(k)$, $s(k)$ and
$\varphi(k)$.

The \textit{Marked Length Spectrum} of $\mathcal{D}$ is defined as the function
\begin{equation}\label{marked spectrum}
 \mathcal{L}\colon \sigma \mapsto \mathcal{L}(\sigma),\quad \sigma\in \cup_{p \geq 2} \{1,2,3\}^p \text{ finite admissible word},
\end{equation}
where $\mathcal{L}(\sigma)$ is the length of the periodic orbit
identified by $\sigma$, obtained by summing the lengths of all the line
segments that compose it.

For any periodic orbit $(x_1,\dots, x_p)$ encoded by a word $\sigma$ of length $p \geq 2$, we have $D_{x_j}\mathcal{F}^p\in \mathrm{SL}(2,\R)$, for $j \in \{1,\dots,p\}$.\footnote{Recall that for $x=(s,\varphi)\in \mathcal{M}$ and $x'=(s',\varphi'):=\mathcal{F}(s,\varphi)$, we have $\det D_{x}\mathcal{F}=\frac{\cos \varphi}{\cos \varphi'}$. Thus, for any periodic orbit $(x_1,x_2,\dots ,x_p)$ of period $p \geq 2$, we have $\det D_{x_j}\mathcal{F}^p=1$, for $j \in \{1,\dots,p\}$.} Due to the strict convexity of the obstacles, the matrix of the differential $D_{x_j}\mathcal{F}^p$ is hyperbolic, and we denote by $\mu(\sigma)>1$ (\resp
$\lambda(\sigma):=\mu(\sigma)^{-1}<1$) its largest, (\resp smallest)
eigenvalue. The
\textit{Lyapunov exponent} of this orbit is defined as
$\mathrm{LE}(\sigma):=\frac 1p\log \mu(\sigma)=-\frac 1p \log
\lambda(\sigma) >0$.
Similarly, the \textit{Marked Lyapunov Spectrum} of the billiard table $\mathcal{D}$ is defined as the function
\begin{equation}\label{maked lepnov}
\mathrm{LE}\colon \sigma \mapsto \mathrm{LE}(\sigma),\quad \sigma\in\cup_{p \geq 2} \{1,2,3\}^p \text{ finite admissible word}.
\end{equation}

\subsection{Basic facts on chaotic billiards}

Let us also recall a few useful facts about chaotic billiards; for more details, we refer to the book of Chernov-Markarian \cite{CM}.

We consider the billiard flow $\Phi^t \colon \Omega\to \Omega$. For any point $X=(q,v) \in \Omega$, we denote  by $\omega \in [0,2\pi)$ the counterclockwise angle between the positive horizontal axis and the velocity vector $v$.  The \textit{Jacobi coordinates} $(d\eta,d\xi)$  in the subspace $T_{(x,y)}\dom\subset T_X \Omega$  
are defined as
$$
d\eta=\cos \omega\, dx+\sin \omega\, dy,\qquad d\xi=-\sin\omega\, dx+\cos \omega\, dy.
$$
In other terms, $d\eta$ is the component  along the velocity vector $v$, and $d\xi$ is its  orthogonal component. Let us introduce, furthermore, the subspace
$T_X^\perp\Omega:=\{dX: \cos \omega\, dx+\sin \omega\, dy=0\}\subset T_X \Omega$ which is invariant for the flow in the sense that
$D_X\Phi^t (T_X^\perp\Omega)=T_{\Phi^t X}^\perp\Omega$. The slope of a tangent line $L \subset T_X^\perp\Omega$ is defined as $\mathcal{B}:=\frac{d\omega}{d\xi}$. A smooth curve in $\mathcal{D}$ equipped with a continuous family
of unit normal vectors is called a \textit{wave front}; it is  \textit{dispersing} if it has positive slope, i.e., $\mathcal{B}>0$.

In the following, we restrict ourselves to the collision space $\mathcal{M}$,
and we denote by  $\|\cdot\|_p$ the  $p$-\textit{metric} on tangent vectors $(ds,d\varphi)$:
\begin{equation}\label{deefini petri}
\|(ds,d\varphi)\|_p:=\cos \varphi| ds|.
\end{equation}
Loosely speaking, a smooth curve in the collision space is called a dispersing wave front if it is the projection on $\mathcal{M}$ of  a dispersing wave front as above.  More precisely, let $i \in \{1,2,3\}$, let $0<a<b<\ell_i$, and
let $\mathcal{W}=\{(s,\phi(s)): a \leq s\leq b\}$, 
for some $C^1$ function
$\phi \colon [a,b] \to [-\frac{\pi}{2},\frac{\pi}{2}]$. We define $\mathcal{B}^+$ as
$$
\mathcal{B}^+(s):=\frac{\phi'(s)+\mathcal{K}(s)}{\cos \phi(s)},\quad \forall\,  a \leq s\leq b,
$$
where $\mathcal{K}(s)$ is the curvature at the point with parameter $s$.  Then, we  say that $\mathcal{W}$ is a dispersing wave front if $\mathcal{B}^+(s)>0$, for all $a \leq s\leq b$.
We denote its length by
$|\mathcal{W}|:=\int_{\mathcal{W}} \|dx\|_p$. Then,  for all
$k \geq 0$,
$\mathcal{F}^k(\mathcal{W})$ is a finite collection of dispersing wave fronts,  and the sequence
$(|\mathcal{F}^k(\mathcal{W}) |)_{k \geq 0}$, which denotes the total length of these wave fronts, grows exponentially fast.

This follows from the following fact.
Fix $x \in \mathcal{M}$, and let $x(k)=(s(k),\varphi(k))$, $k \geq 0$, be its forward iterates. Given  any  vector $dx=(ds,d\varphi)$ in the tangent line $L\subset T_x \mathcal{M}$, then for any $k \geq 0$, we have  (see e.g. Chernov-Markarian \cite{CM}, p. 58)
\begin{equation}\label{expansion first}
\frac{\|D_x \mathcal{F}^k\cdot dx\|_p}{\|dx\|_p}=\prod_{j=0}^{k-1}|1+\ell_j \mathcal{B}_j^+|,
\end{equation}
where $\mathcal{B}_j^+$ denotes the slope of the postcolisional line $L_j^+$ corresponding to the line $L_j:=D_x \mathcal{F}^j(L)$ obtained after $j$ iterations, and $\ell_j:=h(s(j),s(j+1))$ denotes the length of the $j^\text{th}$ segment of the forward orbit of $x$ under $\mathcal{F}$. Moreover,
$\mathcal{B}_{j+1}^+=\frac{\mathcal{B}_j^+}{1+\ell_j\mathcal{B}_j^+} +\frac{2\mathcal{K}(s_{j+1})}{\cos\varphi_{j+1}}$, which ensures that $\mathcal{B}_j^+>c$ for some $c>0$ that depends only on the billiard configuration.\\

To conclude, let us also recall some important symmetry of the billiard dynamics, which will be very useful in the following. Let us denote by  $\mathcal{I}$ the map $\mathcal{I}\colon (s,\varphi) \mapsto (s,-\varphi)$. It conjugates the billiard map $\mathcal{F}$ with its inverse $\mathcal{F}^{-1}$, according to the time-reversal property of the billiard dynamics:
$$
\mathcal{I} \circ \mathcal{F} \circ \mathcal{I}=\mathcal{F}^{-1}.
$$
In the following, a  periodic orbit of period $p=2q \geq 2$ is called \textit{palindromic} if it can be labelled by an admissible word $\sigma \in \{1,2,3\}^p$ such that  $\sigma=(\sigma_1\dots\sigma_{q-1}\sigma_q\sigma_{q-1}\dots\sigma_1\sigma_0)$ for certain symbols $(\sigma_0,\sigma_1,\dots,\sigma_q)\in \{1,2,3\}^{q+1}$. As we shall see later, there is a connection between the palindromic symmetry and the time-reversal property recalled above. In particular, by the palindromic symmetry  and by expansiveness of the dynamics, the associated trajectory hits the billiard table perpendicularly at the    points with symbols $\sigma_0$ and $\sigma_q$.

\section{Statement of the results}

\subsection{The case of period two orbits}

In this part, we consider billiard tables $\mathcal{D} \subset \R^2$
as above, and we focus on periodic orbits of period two.  There are
three such orbits, realized as the minimizing segments between each
pair of obstacles.

\begin{theoalph}\label{prop tperiod twwo}
  Consider a period two orbit encoded by a word
  $\sigma=(\sigma_1\sigma_0)\in \{1,2,3\}^2, 
  \sigma_0\ne \sigma_1$.
  Let $\tau_1$  be such that $\{\tau_1,\sigma_0,\sigma_1\}=\{1,2,3\}$, and set
  $\tau:=(\tau_1\sigma_0)$. We denote by $R_0,R_1>0$ the respective
  radii of curvature at the points with symbols $\sigma_0,\sigma_1$, and let $\lambda=\lambda(\sigma)<1$ be the smallest
  eigenvalue of $D\mathcal{F}^2$ at the points of
  $\sigma$.

  Then, for $n\gg 1$, the following estimates hold:\footnote{As we will
    see more in detail later, the right hand side of each of the following
    estimates is negative, because periodic orbits are minimizers of the
    length functional.}
\begin{enumerate}
	\item if $n$ is even, then
	\begin{equation*}
		\mathcal{L}(\tau \sigma^{n})-(n+1) \mathcal{L}(\sigma)-\mathcal{L}^\infty
		= - C\cdot  \mathcal{Q}\left(\frac{2R_0}{\mathcal{L}(\sigma)},\frac{2R_1}{\mathcal{L}(\sigma)}\right) \lambda^{n}+O(\lambda^{\frac{3n}{2}});
	\end{equation*}
	\item if $n$ is odd, then
	\begin{equation*}
		 \mathcal{L}(\tau \sigma^{n})-(n+1) \mathcal{L}(\sigma)-\mathcal{L}^\infty
		= -C \cdot \mathcal{Q}\left(\frac{2R_1}{\mathcal{L}(\sigma)},\frac{2R_0}{\mathcal{L}(\sigma)}\right) \lambda^{n}+O(\lambda^{\frac{3n}{2}}),
	\end{equation*}
\end{enumerate}
  for some real number $\mathcal{L}^\infty=\mathcal{L}^\infty(\sigma,\tau) \in \R$,
  some constant $C=C(\sigma,\tau)>0$, and
  a quadratic form $\mathcal{Q}\colon \R\times \R\to \R$ of the form
  \begin{equation*}
    \mathcal{Q}(X,Y):=(1+\lambda^2) (1+X)^2 - (1+\lambda)^2 XY + 2\lambda(1+ Y)^2.
  \end{equation*}
  Finally, the  eigenvalue $\lambda<1$ is solution to the equation
  \begin{align*}
    4 \lambda \left(\frac{\mathcal{L}(\sigma)}{2R_0}+1\right)\left(\frac{\mathcal{L}(\sigma)}{2R_1}+1\right)=(1+\lambda)^2.
  \end{align*}
\end{theoalph}
\begin{coralph}
  With the same notation as in Theorem \ref{prop tperiod twwo}, we have
  \begin{align*}
    \mathrm{LE}(\sigma) &= \frac 12 \log(\mu)=-\frac 12 \log(\lambda)=-\lim_{n \to+\infty} \frac{1}{2n} \log\left( -\mathcal{L}(\tau \sigma^n)+(n+1)\mathcal{L}(\sigma)+\mathcal{L}^\infty\right).
  \end{align*}
  In particular, the value of the Lyapunov exponent of $\sigma=(\sigma_1 \sigma_0)$ is determined by the Marked Length Spectrum.
\end{coralph}

\begin{remark}
In the case of two strictly convex obstacles,  Stoyanov   \cite{S} considered  expressions similar to those in Theorem  \ref{prop tperiod twwo} but for sequences of travelling times of billiard trajectories with $n\geq 1$ reflection points between the two scatterers instead of lengths of periodic orbits (there is a unique periodic orbit in this situation, of period two), and  showed an analogous asymptotic with an exponentially small error term. Based on that, he was able to  recover the distance between the two obstacles
and the Lyapunov exponent of the unique period two orbit. For further related results see also Section 10.3 in \cite{PS1} and Section 7 in \cite{PS2}.
\end{remark}


The proof of Theorem \ref{prop tperiod twwo} is given in
Section~\ref{section estimmmm}, and is based on a procedure explained in
Section~\ref{section estiitm}, where we introduce a sequence of periodic
orbits shadowing some orbit that is homoclinic to the period two orbit
$\sigma$.  Such orbits capture local geometric information on the orbit
$\sigma$.  First, we estimate the parameters of the points in these
orbits, and then we show that when we combine the lengths of these
orbits in an appropriate way, then the differences converge to zero
exponentially fast, at the rate given by the Lyapunov exponent of
$\sigma$. Besides, the error term is maximal near the orbit $\sigma$,
which allows us to extract some geometric information on $\sigma$.

Recall that the curvature $\mathcal{K}$ at a point of the table is
defined as the inverse of the radius of curvature $R$, i.e.,
$\mathcal{K}=R^{-1}$. Then, Theorem \ref{prop tperiod twwo} has the
following geometric consequence (see Corollary \ref{coro expan}  in
Subsection \ref{seubcdoeg} and also Corollary \ref{rema spectral inv} where we derive other spectral invariants).

\begin{coralph}\label{main corr}
The curvatures $\mathcal{K}_0$ and $\mathcal{K}_1$ at the bouncing points of periodic orbits of period two are determined by the Marked Length Spectrum (hence also the value of the constant $C$ in Theorem \ref{prop tperiod twwo}).
\end{coralph}

Actually, we have $\lim_{n \to+\infty}\mathrm{LE}(\tau \sigma^n)=\mathrm{LE}( \sigma )$, and by considering Taylor expansions of $\mathrm{LE}(\tau \sigma^n)$, we have also shown that it is possible to recover the respective derivatives of the curvature $\mathcal{K}_0'$ and $\mathcal{K}_1'$ at the two bouncing points of the orbit $\sigma$.

As an illustration of the above results, we show that the geometry of certain billiard tables is entirely determined by the Marked Length Spectrum.

\begin{remark}[See Corollary \ref{lem easy thr dis} and the proof that follows]
  The Marked Length Spectrum determines completely the geometry of
  billiard tables
  $\mathcal{D}=\R^2 \setminus\{\obs_1,\obs_2,\dots,\obs_m\}$ obtained by
  removing from the plane $m \geq 3$ round discs $\obs_1,\dots,\obs_m$
  which satisfy the non-eclipse condition, i.e., such that the convex
  hull of any two obstacles is disjoint from the other obstacles.
\end{remark}

\subsection{Marked Length Spectrum and Marked Lyapunov
  Spectrum}\label{section marked length}

We consider a periodic orbit of period $p \geq 2$, encoded by some
admissible word
$\sigma=(\sigma_1\sigma_2\dots\sigma_p) \in \{1,2,3\}^p$. We denote by
$\lambda=\lambda(\sigma)<1$ the smallest eigenvalue of $D\mathcal{F}^p$
at the points in the orbit.

Let $\tau=(\tau^-,\tau^+)\in \{1,2,3\}^2$ be such that
$\tau^-\neq \sigma_{p}$ and $\tau^+\neq \sigma_{1}$. For any integer $n \geq 0$, we define the two words
\begin{align*}
  h_{n}^+(\sigma,\tau)&:=(\tau^+\sigma^n\tau^-\overline{\sigma}^n):=(\tau^+\underbrace{\sigma\sigma\dots \sigma}_\text{n\ \text{times}}\tau^-\underbrace{\overline{\sigma\sigma\dots \sigma}}_\text{n\ \text{times}}),\\
  h_{n}^-(\sigma,\tau)&:=(\tau^-\overline{\sigma}^n\tau^+\sigma^n):=(\tau^-\underbrace{\overline{\sigma\sigma\dots \sigma}}_\text{n\ \text{times}}\tau^+\underbrace{\sigma\sigma\dots \sigma}_\text{n\ \text{times}}).
\end{align*}


Note that the symbolic representations are just shifted by $np+1$ symbols, hence $h_{n}^+(\sigma,\tau)$ and $h_{n}^-(\sigma,\tau)$ correspond to the same periodic orbit of period $2np+2$, denoted by $h_n'$ in the following. Also, this orbit $h_n'$ is palindromic, which implies that the bounces at the symbols $\tau^+$ and $\tau^-$ are perpendicular.

\begin{theoalph}\label{main prop per}
There exist a real number $\mathcal{L}^\infty(\sigma,\tau) \in \R$ and two constants $C_{\mathrm{even}}(\sigma,\tau),C_{\mathrm{odd}}(\sigma,\tau)>0$ such that for $n \gg 1$,
\begin{enumerate}
\item if $n$ is even, then
\begin{equation*}
\mathcal{L}(h_n'(\sigma,\tau) )-2n \mathcal{L}(\sigma)-\mathcal{L}^\infty(\sigma,\tau)= -C_{\mathrm{even}}(\sigma,\tau) \cdot \lambda^{n}+O(\lambda^{\frac{3n}{2}});
\end{equation*}
\item if $n$ is odd, then
\begin{equation*}
\mathcal{L}(h_n'(\sigma,\tau) )-2n \mathcal{L}(\sigma)-\mathcal{L}^\infty(\sigma,\tau)= -C_{\mathrm{odd}}(\sigma,\tau) \cdot \lambda^{n}+O(\lambda^{\frac{3n}{2}}).
\end{equation*}
\end{enumerate}
\end{theoalph}

The proof of the above theorem is given in Section \ref{section lyaapu} and  is along the same lines as in the case of the period two.
As a direct consequence of Theorem \ref{main prop per}, we obtain:
\begin{coralph}\label{corrror firs}
The Marked Lyapunov Spectrum  is entirely determined by the Marked Length Spectrum (see \eqref{marked spectrum} and \eqref{maked lepnov} for the definitions).
\end{coralph}

Theorem \ref{main prop per} for arbitrary period $p\ge 3$ is weaker than Theorem
\ref{prop tperiod twwo} for period two, because of two reasons.
First, observe that  period two orbits are in a sense degenerate, namely, their local dynamics
has only {\it two} free parameters. Indeed, the length of each side is given by the Marked
Length Spectrum, the angles of reflections are $\frac \pi 2$, and the only free parameters
are the two curvatures at the two collision points. Meanwhile, period $p\ge 3$ orbits
have $3p-2$ parameters:  $p-1$ independent orbit segments' lengths, $p-1$ independent angular data and $p$ curvatures at bouncing points\footnote{Even for $p=3$ we have
$7$ vs $2$ free parameters!}. More precisely, to prove an analog of Theorem
\ref{prop tperiod twwo} we face the following problem.

\begin{remark}
  The proofs of Theorem \ref{prop tperiod twwo} in the period two case
  and Theorem \ref{main prop per} for a general periodic orbit $\sigma$
  are based on precise estimates on the parameters of certain sequences
  of orbits which approximate $\sigma$. Such estimates are obtained by
  linearizing the dynamics; indeed, each point in the orbit $\sigma$ is
  a saddle fixed point of the $p^{\text{th}}$ iterate of the billiard
  map $\mathcal{F}$, where $p \geq 2$ is the period of $\sigma$. This
  yields two independent relations according to the parity of the number
  $n\geq 0$ of repetitions of $\sigma$ in the coding of the
  approximating sequences. Besides, those relations carry some
  asymmetric information between the points of $\sigma$.  This is enough
  to reconstruct the curvature for period two orbits
  (Corollary~\ref{main corr}), but it is insufficient when the period
  $p$ is at least $3$, due to the large number of parameters encoded
  implicitly by $\mathcal{F}^p$. One (unsuccessful) strategy that we
  have tried to produce more relations was to shift the reference point
  of $\sigma$ and vary the word $\tau$ used to define approximating
  sequences. This does modify the geometry of the associated orbits; yet,
  those changes happen ``far'' from the orbit $\sigma$ that we want to
  describe, while the estimates in Theorem \ref{main prop per} are
  obtained by considering points in a neighborhood of $\sigma$. Such
  changes result in a scaling of the asymptotic constants in the
  estimates of Subsection \ref{subsection esit parammm}, but do not give
  any additional asymmetric information to distinguish between the
  points of $\sigma$.
\end{remark}

\subsection*{Organization of the paper}

The rest of the paper is organized as follows. We fix a period two orbit
$AB$ (see Fig.\,\ref{fig:Hom-1}). In Section \ref{section estiitm} and Section \ref{sec:improved-estimates}
we describe a class of homoclinic orbits to $AB$ and estimate the speed
these homoclinic orbits are approximated by periodic orbits.  In Section \ref{sec:consec-MLS}
 we utilize these estimates, prove Theorem
\ref{prop tperiod twwo}, and derive the associated constants.  In
Section \ref{section lyaapu}, we prove Theorem \ref{main prop per} and
its Corollaries.

\section{Estimates on periodic orbits shadowing some homoclinic orbit}\label{sec sec per deux}\label{section estiitm}

\subsection{Idea of the construction}

In the following, we explain a geometric construction which allows us to
extract from the Marked Length Spectrum some geometric information near
a periodic orbit.  In most of the paper, we explain in detail this
procedure in the case where the periodic orbit has period two.  As we
shall see later, the general construction as well as most estimates
remain true for more general periodic orbits. Yet, period two orbits
have special features which make them nicer to deal with, namely:
\begin{itemize}
\item they have some palindromic symmetry; loosely speaking, this will
  be especially useful in the following in order to connect the future
  and the past of certain orbits that we introduce in the construction;
\item we know the angles at the two bouncing points (the orbit hits the
  table perpendicularly at these two points) as well as the length of
  the two arcs of the orbit (they are equal to half of the total length
  of the orbit). In particular, if we want to describe the local
  geometry near the points of the orbit, then at the leading order, the
  only quantity to reconstruct is the radius of curvature at the two
  bouncing points.
\end{itemize}

The construction we are about to describe is based on an orbit
homoclinic to the chosen period two orbit identified by $\sigma$, which
conveys a lot of geometric information on $\sigma$ (the same homoclinic orbit was considered in the aforementioned work of Stoyanov \cite{S}).  However since this
orbit is not periodic, this information is not readily available in
the Marked Length Spectrum. Therefore, we define a sequence
$(h_n)_{n \geq 0}$ of periodic orbits which shadow more and more closely
this homoclinic orbit. Exploiting the symbolic coding of the dynamics,
such orbits are obtained by truncating the infinite word associated to
the homoclinic orbit in order to produce a finite word, which then
corresponds to a periodic orbit (see Figure~\ref{fig:Hom-1}).  By
expansivity of the dynamics, the the sequence $(h_n)_{n \geq 0}$
converges to the limit homoclinic orbit exponentially fast, whose rate
is related to the Lyapunov exponent of the period two orbit (which is
also the Lyapunov exponent of the homoclinic orbit). By comparing the
length of $h_n$ with the length of the original period two orbit, we see
that the estimates involve two different constants depending on the
parity of $n\geq 0$. This is due to some asymmetry in the orbits $h_n$:
indeed, as we will see, among all the points in the orbits $h_n$, there
is exactly one bounce which is closest to the period two orbit that we
approximate, and this bounce is on one of the two obstacles associated
to the period two orbit, which depends on the parity of $n$.
Combining these estimates with the properties of period two orbits that
we described above, we see that this is actually sufficient to recover the
radius of curvature at each point of the period two orbit under
construction.

\subsection{The case of period two orbits}

In this part, we focus on the first two obstacles
$\mathcal{O}_1,\mathcal{O}_2$, and on the period two orbit between
them. It is labelled by the word $\sigma=(12)$ according to the marking
introduced
earlier.
We denote by $\mathcal{L}(12)$ the length of the periodic orbit $(12)$,
and we denote by $A$ (\resp $B$), the point in the orbit which lies on
the boundary $\Gamma_1$ of the first obstacle (\resp on the boundary
$\Gamma_2$ of the second obstacle).  We let
$\ell:=\frac 12 \mathcal{L}(12)$ be the distance between the points $A$
and $B$ in the plane, and we assume that $A=(-\frac 12 \ell,0)$ and
$B=(\frac 12 \ell,0)$.  We also choose the parametrizations of
$\mathcal{M}_i$, $i \in \{1,2,3\}$, in such a way that $(0,0)=(0_1,0)$
(\resp $(0,0)=(0_2,0)$), are the $(s,\varphi)$-coordinates of the
(perpendicular) bounce of the orbit $\sigma$ at the point $A$ (\resp
$B$). 
In particular, we have $A=\gamma_1(0)$ and $B=\gamma_2(0)$.


For $\{i,j\}=\{1,2\}$, the matrix of the differential of the billiard
map $\mathcal{F}$ in $(s,\varphi)$-coordinates at the point with
coordinates $(0_i,0)$ is
\begin{align*}
  D_{(0_i,0)}\mathcal{F}=-\begin{pmatrix}
    \frac{\ell}{R_i}+1 & \ell\\
    \frac{\ell}{R_iR_j}+\frac{1}{R_i}+\frac{1}{R_j} & \frac{\ell}{R_j}+1
  \end{pmatrix},
\end{align*}
where $R_1=\mathcal{K}_1^{-1}$ (\resp $R_2=\mathcal{K}_2^{-1}$) denotes
the radius of curvature at the point $A$ (\resp $B$).

The periodic orbit $(12)$ is hyperbolic.  We denote by $\lambda=\lambda(12)<1$, $\mu=\mu(12):=\lambda^{-1}>1$ the common eigenvalues of the differentials  $D_{(0_1,0)}\mathcal{F}^2=D_{(0_2,0)}\mathcal{F}\circ D_{(0_1,0)}\mathcal{F}$ and $D_{(0_2,0)}\mathcal{F}^2=D_{(0_1,0)}\mathcal{F}\circ D_{(0_2,0)}\mathcal{F}$.
In particular, $\pm \frac 12 \log \mu=\mp \frac 12 \log \lambda$ are the two Lyapunov exponents of the hyperbolic fixed points $(0_1,0)$ and $(0_2,0)$ of $\mathcal{F}^2$.

\begin{lemma}\label{bounded domm}
  For any $x=(s,\varphi) \in \widetilde{\mathcal{M}}$, the point
  $\gamma(s) \in \partial \mathcal{D}\subset \R^2$ is in the region of
  the plane bounded by the three segments and the three arcs connecting
  the points in period two orbits.
\end{lemma}

\begin{proof}
Let us assume, by contradiction, that there exists some $x_0=(s_0,\varphi_0) \in \widetilde{\mathcal{M}}$  such that the associated point in $\partial\mathcal{D}$ lies on $\Gamma_1$ beyond the point $A$, i.e., such that $\gamma_1(s_0)=(\bar x,\bar y)$, with $\bar y > 0$. After possibly replacing $x_0$ with $\mathcal{I}( x_0)=(s_0,-\varphi_0)$, we may assume that $\varphi_0\geq 0$, in such a way that the associated vector does not point towards the segment AB. Indeed, by the time-reversal property, the orbit of $\mathcal{I}( x_0)$  under $\mathcal{F}$ is the orbit of $x_0$ traversed backwards,  hence this change does not affect the fact that the orbit escapes to infinity or not. Let us denote by $(x(n)=(s(n),\varphi(n)))_{n \geq 0}$, the forward iterates of $x(0):=x_0$ under $\mathcal{F}$. By strict convexity of the obstacles, the sequences $(\varphi(2k))_{k\geq 0}$ and $(-\varphi(2k+1))_{k\geq 0}$ are strictly increasing, and then, the associated points in $\partial\mathcal{D}$ never come back to the region below the segment $AB$. Therefore, the symbolic coding of the forward orbit $(x(n))_{n \geq 0}$ is $121212\dots$, which is also the coding of the forward iterates of the point $(0_1,0)$ in the period two orbit $(12)$.
By the exponential growth of wave fronts (see \eqref{expansion first}), this implies that $x_0$ is in the stable manifold of the point $(0_1,0)$.
Therefore, $\lim_{k\to +\infty}x(2k)=(0_1,0)$, and the sequence $(\varphi(2k))_{k \geq 0}$ converges  to $0$, a contradiction, since it is also non-negative and strictly increasing.

Let us give a few more
details. 
Take 
a continuous monotonic map $\phi\colon[0,s_0]\to \R_+$ such that
$\phi(0)=0$  and $\phi(s_0)=\varphi_0$, and set
$$
\mathcal{W}^{(2k)}:=\{x^{(2k)}(s):= \mathcal{F}^{2k}(s,\phi(s)):0\leq s \leq s_0\},\quad \forall\, k \geq 0.$$
By the strict convexity of the obstacles, and since the symbolic coding of  $(x(n))_{n \geq 0}$ is $121212\dots$, then for each $k \geq 0$,  $\mathcal{W}^{(2k)}$ is a dispersing wave front 
whose projection 
on the table is the arc of $\Gamma_1$ bounded by the points $A$ and $\gamma_1(s(2k))$.
We obtain
\begin{align*}
&|\mathcal{W}^{(2k)}|=\left|\int_{0}^{s_0} \|D_{x^{(0)}(s)} \mathcal{F}^{2k} \cdot (x^{(0)})'(s)\|_p\, ds\right|\\
&\geq \Lambda^{(2k)} \left|\int_{0}^{s_0} \|(x^{(0)})'(s)\|_p\, ds\right|\geq \Lambda^{(2k)} \cdot \|x^{(0)}(s_0)\|_p=\Lambda^{(2k)} \cdot \|x_0\|_p. 
\end{align*}
By \eqref{expansion first}, the sequence $( \Lambda^{(2k)})_{k\geq 0}$ goes to $+\infty$, while
the left hand side is uniformly bounded.\footnote{Note that the p-length of a dispersing wavefront is uniformly bounded by the length of its trace on the scatterer.\label{foonote cinq}} 
Letting $k\to +\infty$, we deduce that  $x_0=(s_0,\varphi_0)=(0_1,0)$, and thus $\gamma_1(s_0)=\gamma_1(0_1)=A$, a contradiction.
\end{proof}

\subsection{Shadowing by palindromic periodic orbits}

For any integer $n \geq 0$, we introduce the periodic orbit $h_n$ associated to the symbolic coding
\begin{equation*}
h_n:=(32\vert(12)^n)=(32\underbrace{121212\dots 1212}_\text{2n}).
\end{equation*}

\begin{figure}[H]
\begin{center}
    \includegraphics [width=15cm]{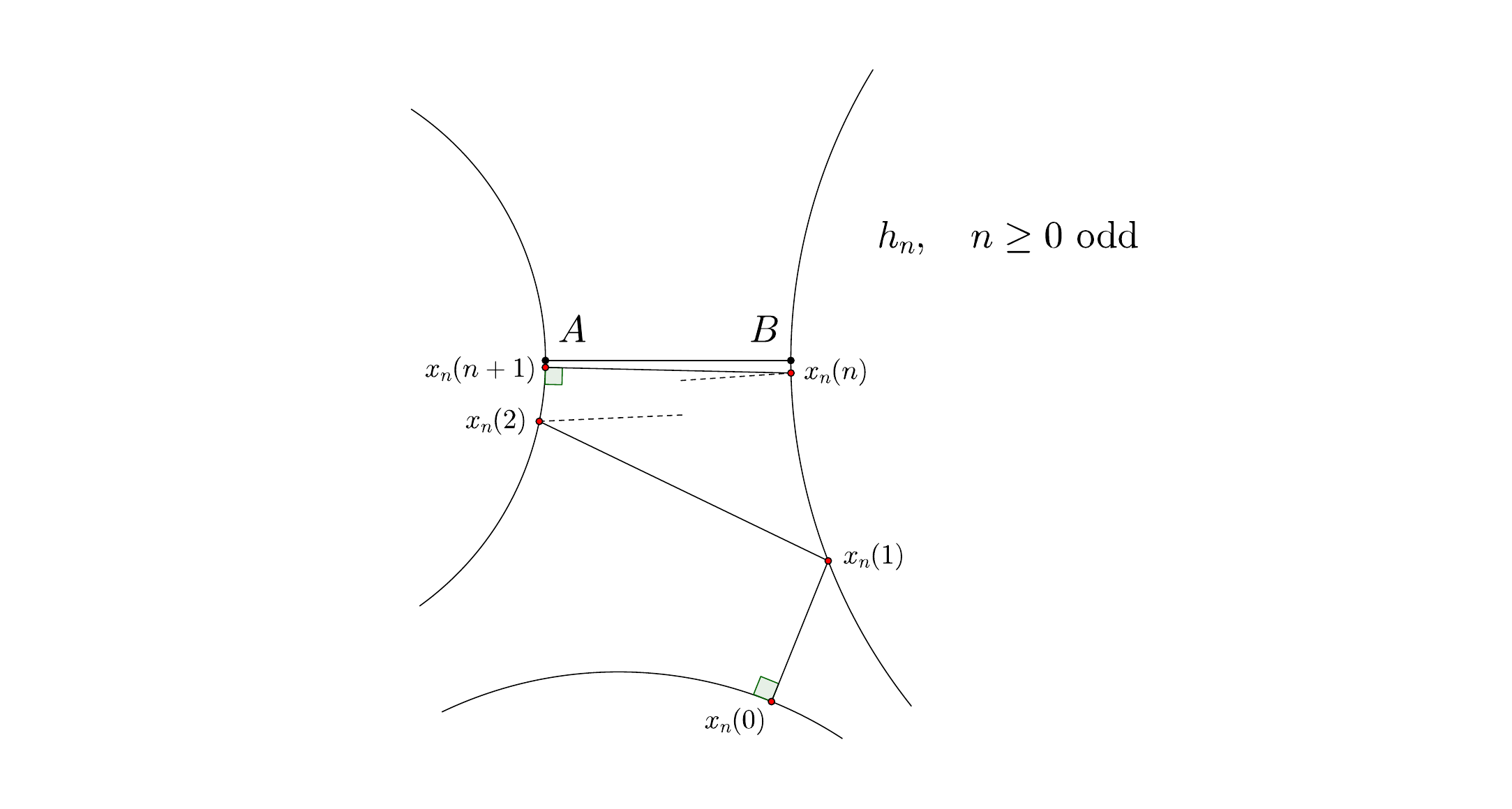}
\end{center}
\caption{Approximation of the homoclinic orbit to the period two orbit
  between the first two obstacles (here we assume that $n$ is odd).}
\label{fig:Hom-1}
\end{figure}

The period of the orbit $h_n$ is equal to $2n+2$. We denote by $x_n(k)=(s_n(k),\varphi_n(k))_{k=0,\dots,2n+1}$ the $(s,\varphi)$-coordinates of the points in this orbit, and we extend them to all integer indices by setting $x_n(k):=x_n(k \text{ mod } (2n+2))$, for all $k \in \Z$. We label the points of $h_n$ in such a way that one period of this orbit matches the following geometric description:

\begin{itemize}
\item there is a unique perpendicular bounce on the third obstacle, with coordinates $(s_n(0),0)$;
\item the following $n$ bounces alternate between the first and the
  second obstacle, getting closer to the periodic orbit $(12)$ each
  time; for each integer $k \in \{1,\dots,n\}$, the point with
  coordinates $(s_n(k),\varphi_n(k))$ is associated to a point on the
  first obstacle, resp. the second obstacle, whenever the index $k$ is
  even, resp. odd; in particular, the second bounce
  $(s_n(1),\varphi_n(1))$ is always on the second obstacle;
\item the next point has coordinates $(s_n(n+1),0)$ and is associated to a perpendicular bounce on the first obstacle, resp. second obstacle, whenever $n$ is odd, resp. even;  it is the point of $h_n$ which is closest to the orbit $(12)$;
\item the next $n$ bounces correspond to the same points  as for $k \in \{1,\dots,n\}$, the new orbit segment being the previous one traversed backwards. \\
\end{itemize}


We see that the orbit $h_n$ is asymmetric between the obstacles $1$ and
$2$, in the sense that for each odd integer $n\geq 0$, the point of
$h_n$ which is closest to the period two orbit $(12)$ is on the first
obstacle, while for each even integer $n\geq 0$, the point of $h_n$
which is closest to the period two orbit $(12)$ is on the second
obstacle.

The last item of the above description follows from the following result.
\begin{lemma}\label{lemma palind}
For any $k\in \{0,\dots,n+1\}$, we have
$$
s_n(2n+2-k)=s_n(k),\qquad \varphi_n(2n+2-k)=-\varphi_n(k).
$$
\end{lemma}
This explains why each of the orbit segments associated to indices in
$\{n+1,\dots,2n+2\}$ is one of the orbit segments associated to indices
in $\{0,\dots,n+1\}$, traversed backwards, according to the above
description.  In particular, this also implies that
$\varphi_n(0)=\varphi_n(n+1)=0$.
\begin{proof}
  The proof follows from some palindromic property of the orbits $h_n$,
  $n \geq 0$. Indeed, for any even integer $n \geq 0$, we have
  \begin{equation*}
    \dots \vert h_n\vert h_n\vert\hdots
    =\dots\vert\underset{\substack{\uparrow}}{3}21212\dots1\underset{\substack{\uparrow}}{2}1\dots 21212\vert\underset{\substack{\uparrow}}{3}21212\dots1\underset{\substack{\uparrow}}{2}1\dots 21212\vert\dots
  \end{equation*}
  and for any odd integer $n \geq 0$, we have
  \begin{equation*}
    \dots \vert h_n\vert h_n\vert\hdots
    =\dots\vert\underset{\substack{\uparrow}}{3}21212\dots2\underset{\substack{\uparrow}}{1}2\dots 21212\vert\underset{\substack{\uparrow}}{3}21212\dots2\underset{\substack{\uparrow}}{1}2\dots 21212\vert\dots
  \end{equation*}
  We see that the above symbolic expansions are palindromic at the
  points marked with arrows, i.e., the future and the past of such
  points have the same symbolic
  coding.

  Recall that the map
  $\mathcal{I}\colon (s,\varphi) \mapsto (s,-\varphi)$ switches future
  and past, according to the time-reversal property
  $\mathcal{I} \circ \mathcal{F} \circ \mathcal{I}=\mathcal{F}^{-1}$ of
  the billiard dynamics. Let us consider the point $x_n(n+1)$.  Due to
  the palindromic symmetry, 
  the orbits of $x_n(n+1)$ and $\mathcal{I} (x_n(n+1))$ under
  $\mathcal{F}$ have the same symbolic coding.  By expansivity of the
  dynamics of $\mathcal{F}$, we conclude that
  $\mathcal{I} (x_n(n+1))=x_n(n+1)$, and thus, $\varphi_n(n+1)=0$. For the
  same reason, we have $\varphi_n(0)=0$.  Then, for any
  $k \in \{0,\dots,n+1\}$, we obtain
  \begin{align*}
    (s_n(2n+2-k),-\varphi(2n+2-k))&=\mathcal{I} \circ \mathcal{F}^{n+1-k} (x_n(n+1))\\
                                  &= \mathcal{F}^{k-n-1} \circ \mathcal{I}(s_n(n+1),0)\\
                                  &= \mathcal{F}^{k-n-1} (x_n(n+1))\\
                                  &=(s_n(k),\varphi_n(k)).\qedhere
  \end{align*}
\end{proof}

\subsection{Preliminary estimates on the parameters at bounces}

\begin{lemma}\label{premier lemmmme}
There exists a constant $\Lambda>1$ such that for any integers $n,p \geq 0$, and for any integer $k\in \{0,\dots,2n+1\}$, we have
\begin{equation}\label{eq 1}
\|x_{n}(k)-x_{n+p}(k)\|= O(\Lambda^{-(2n+1)+k}).
\end{equation}
\end{lemma}

\begin{proof}
Let $n,p \geq 0$. Recall that $x_{n}(0)=(s_n(0),0)$ and that $x_{n+p}(0)=(s_{n+p}(0),0)$. In particular, the piece of $\Gamma_{3}=\partial \mathcal{O}_3$ lying between the above points endowed with $0$ angle in all points is a dispersing wave front, that we denote by $\mathcal{W}_{n,p}(0)\subset \mathcal{M}$. For all $0 \leq k \leq 2n+1$, we set $\mathcal{W}_{n,p}(k):=\mathcal{F}^k(\mathcal{W}_{n,p}(0))$; it is also a dispersing wave front. The forward orbits of $x_n(0)$ and $x_{n+p}(0)$ have the respective symbolic codings
\begin{equation*}
3\underbrace{21212\dots 21212}_{2n+1}\vert 3\dots,\qquad
3\underbrace{21212\dots 21212}_{2n+1}\vert 1\dots,
\end{equation*}
hence  the projection of $\mathcal{W}_{n,p}(k)$ on the billard table is an arc of $\Gamma_1$, resp. $\Gamma_2$, when  $0 \leq k\leq 2n+1$ is even, resp. odd.
Recall that $\|\cdot\|_p$ is the $p$-metric on tangent vectors in $T\mathcal{M}$ as in \eqref{deefini petri}, and that for all $0 \leq k \leq 2n+1$, we denote by  $|\mathcal{W}_{n,p}(k)|:=\int_{\mathcal{W}_{n,p}(k)} \|dx\|_p$  the length of the associated wave front.

Take $x=(t,0) \in \mathcal{W}_{n,p}(0)$, and denote by $x(t,k)=(s(t,k),\varphi(t,k))\in \mathcal{W}_{n,p}(k)$, $0 \leq k \leq 2n+1$, its forward iterates. Given  any  vector $dx=(ds,d\varphi)$ in the tangent line $L\subset T_x \mathcal{M}$, then as in \eqref{expansion first}, for all $0 \leq k \leq 2n+1$, we have
\begin{equation*}
\frac{\|D_x \mathcal{F}^k\cdot dx\|_p}{\|dx\|_p}=\prod_{j=0}^{k-1}|1+\ell_j \mathcal{B}_j^+|,
\end{equation*}
with $\ell_j:=h(s(t,j),s(t,j+1))$.
In our case, Lemma \ref{bounded domm} ensures that $\ell_j \geq \ell_{\text{min}}$ and $\mathcal{B}_j^+ \geq \mathcal{B}_{\text{min}}$  for some constants $\ell_{\text{min}}, \mathcal{B}_{\text{min}}>0$ which are uniform in $n\geq 0$, and thus,
$$
\frac{\|D_x \mathcal{F}^k\cdot dx\|_p}{\|dx\|_p}\geq \Lambda^k,
$$
with $\Lambda:=1+\ell_{\text{min}} \mathcal{B}_{\text{min}}>1$. Therefore, for any $0 \leq k \leq 2n+1$, the size $|\mathcal{W}_{n,p}(k)|$ of the image by $\mathcal{F}^{k}$ of the initial wave front $|\mathcal{W}_{n,p}(0)|$ satisfies
\begin{align*}
|\mathcal{W}_{n,p}(k)|&= \int_{\mathcal{W}_{n,p}(0)} \|D_{x} \mathcal{F}^k \cdot dx\|_p
= \left|\int_{s_{n}(0)}^{s_{n+p}(0)} \big\|D_{(t,0)} \mathcal{F}^k \cdot\frac{d}{dt} (t,0)\big\|_p\, dt\right|\\
&\geq \Lambda^k \left|\int_{s_{n}(0)}^{s_{n+p}(0)} \|(1,0)\|_p\, ds\right|=\Lambda^k |s_{n+p}(0)-s_n(0)|.
\end{align*}
As we have seen, for all $0 \leq k \leq 2n+1$, the projection of $\mathcal{W}_{n,p}(k)$ is an arc of $\Gamma_1$ or $\Gamma_2$, hence $|\mathcal{W}_{n,p}(k)|\leq C$, for some constant $C>0$ depending only on the geometry of the table.\footnote{See Footnote \ref{foonote cinq}. 
} We thus obtain
$$
|s_{n+p}(0)-s_n(0)| \leq C \Lambda^{-(2n+1)}.
$$
By Lemma \ref{bounded domm}   and the non-eclipse condition,  for all $x \in \mathcal{W}_{n,p}(0)$, and $0\leq k \leq 2n+1$,  the cosine of the angle of reflection of $\mathcal{F}^k(x)$ is lower bounded by some uniform constant $c_\text{min}^{-1}>0$. This lower bound on the cosine of the angle extends from the endpoints to the entire curve $\mathcal{W}_{n,p}(k)$ as it is an increasing curve in the $(s,\varphi)$ coordinates. Hence
\begin{align*}
&\|x_{n+p}(k)-x_n(k)\|=\| x(s_{n+p}(0),k)-x(s_n(0),k)\|=\left\|\int_{s_{n}(k)}^{s_{n+p}(k)}\partial_s x(s,k)\, ds\right\|\\
&\leq c_\text{min}\left|\int_{s_{n}(k)}^{s_{n+p}(k)}\|\partial_s x(s,k)\|_p\, ds\right|=c_\text{min}\int_{\mathcal{W}_{n,p}(k)}\|dx\|_p=c_\text{min}|\mathcal{W}_{n,p}(k)|.
\end{align*}
Moreover, we have $|\mathcal{W}_{n,p}(2n+1)|=\int_{\mathcal{W}_{n,p}(k)} \|D_x \mathcal{F}^{2n+1-k} \cdot dx\|_p\geq \Lambda^{2n+1-k} |\mathcal{W}_{n,p}(k)|$, and then,
$$
\|x_{n+p}(k)-x_n(k)\| \leq c_\text{min} |\mathcal{W}_{n,p}(k)|\leq  c_\text{min}C \Lambda^{-(2n+1-k)},
$$
which concludes.
\end{proof}

\begin{corollary}\label{coroooooooo}
For any $n \geq 0$, and for any $k\in \{0,\dots,n+1\}$, we have
$$
x_n(k)=x_\infty(k)+O(\Lambda^{-n}),\qquad x_n(2n+2-k)=\mathcal{I}(x_\infty(k))+O(\Lambda^{-n}),
$$
where $(x_\infty(k)=(s_\infty(k),\varphi_\infty(k)))_{k \in \Z}$ are the $(s,\varphi)$-coordinates of the points in the palindromic orbit $h_\infty$ 
encoded by the infinite word
$$
h_\infty=((12)^\infty 3 (21)^\infty)=\dots121232121\dots
$$
Here, $x_\infty(0)=(s_\infty(0),0)$ are the coordinates of the unique
point of $h_\infty$ associated to the symbol $3$.\footnote{Due to the
  palindromic symmetry, as in Lemma~\ref{lemma palind}, the angle at this point has to vanish.}
\end{corollary}

\begin{proof}
For each $n \geq 0$, and each $0 \leq k \leq n+1$, the above estimate on $x_n(k)$ follows immediately from \eqref{eq 1}, by letting $p$ go to infinity. Indeed, for any $n_0 \geq 0$, and for any $0 \leq k \leq n_0+1$,  the sequence $(x_n(k))_{n \geq n_0}$ is Cauchy, hence converges to some limit $x_\infty(k)=(s_\infty(k),\varphi_\infty(k)) \in \R^2$. Since $x_n(k+1)=\mathcal{F}(x_n(k))$, for all $0 \leq k \leq n+1$, and by continuity of $\mathcal{F}$, we deduce that the sequence of points $(x_\infty(k))_{k \geq 0}$ is the forward orbit of $x_\infty(0)$. By the definition of $x_\infty(0)$ as a limit of points whose second coordinate vanishes, we have $x_\infty(0)=(s_\infty(0),0)$, and then the future and the past of this point coincide. Therefore, the orbit of $x_\infty(0)$ is encoded by the word as in the statement of the lemma.

The estimate on the point $x_n(2n+2-k)$ follows from Lemma~\ref{lemma
  palind}, since we have
$$
\mathcal{I}(x_\infty(k))=(s_\infty(k),-\varphi_\infty(k))=(s_n(k),-\varphi_n(k))+O(\Lambda^{-n}),
$$
and  $(s_n(k),-\varphi_n(k))=\mathcal{I}(x_n(k))=(x_n(2n+2-k))$.
\end{proof}

\begin{figure}[H]
\begin{center}
    \includegraphics [width=15cm]{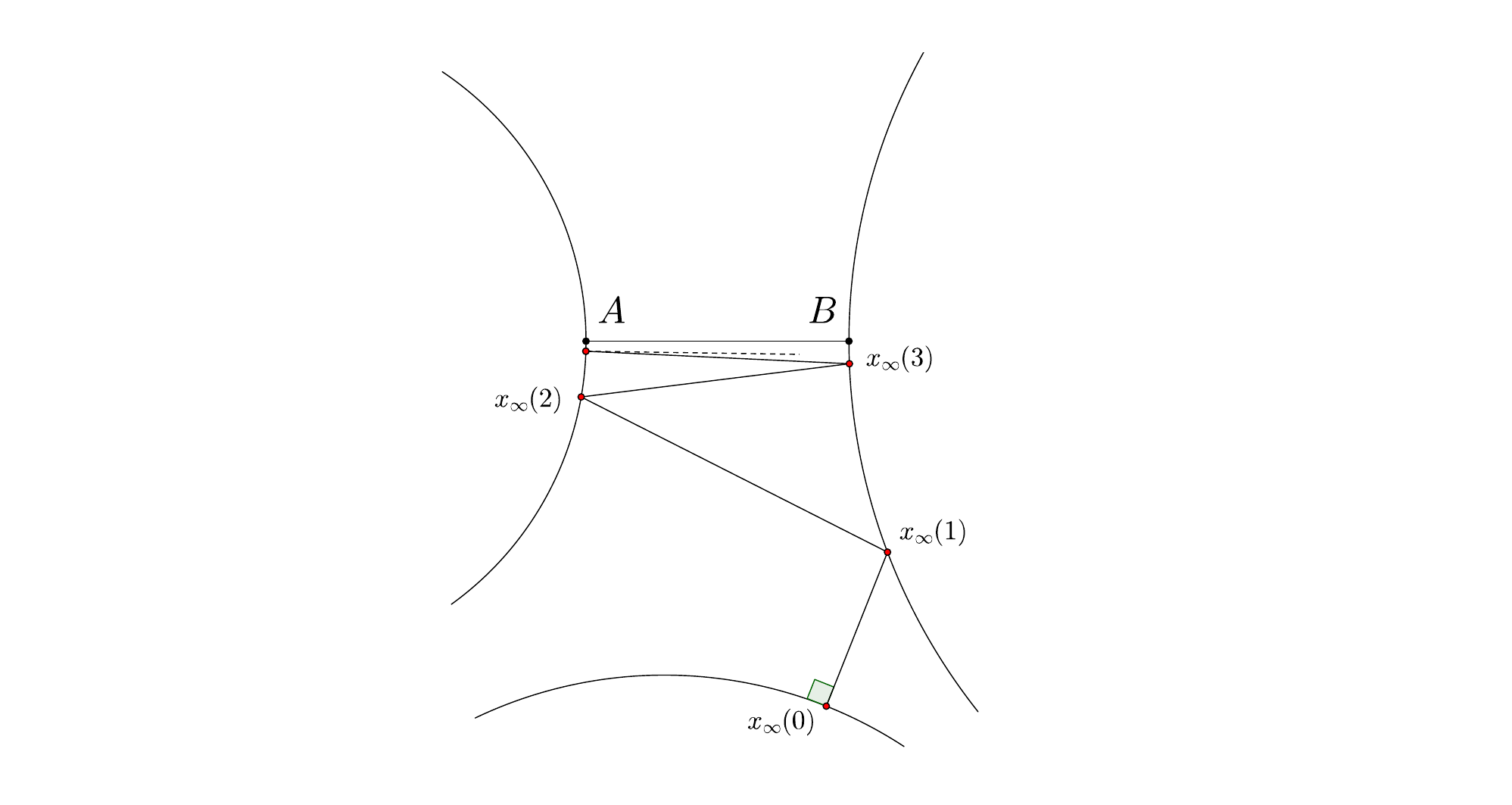}
\end{center}
\caption{Period two orbit and an approximating homoclinic}
\label{fig:Hom-2}
\end{figure}

For $j \in \{1,2\}$, we denote by  $\mathrm{E}_{\mathcal{F}^2}^s(0_j,0):=\{v\in T_{(0_j,0)} \mathcal{M}: D_{(0_j,0)} \mathcal{F}^2\cdot v=\lambda v\}$
 the stable space of $\mathcal{F}^2$ at $(0_j,0)$, 
associated to the smallest eigenvalue $\lambda<1$ of $(12)$. In the following, given two sequences $(u_k)_{k \geq 0}$ and $(v_k)_{k \geq 0}$ such that $v_k \neq 0$ for $k \geq k_0$, for some integer $k_0 \geq 0$, we write $u_k \sim v_k$ if $\lim_{k \to+\infty} \frac{u_k}{v_k}=1$.
\begin{lemma}\label{homcolincc}
The orbit $h_\infty$ is homoclinic  to the period two orbit $(12)$. More precisely, there exist two vectors $\bar{v}^s_1 \in \mathrm{E}_{\mathcal{F}^2}^s(0_1,0)$, $\bar{v}^s_2 \in \mathrm{E}_{\mathcal{F}^2}^s(0_2,0)$, with $\|\bar{v}_1^s\|=1$, and $\bar{v}_2^s= D_{(0_1,0)} \mathcal{F} \cdot \bar{v}_1^s$, such that  for $k \gg 1$, the following estimates hold: 
\begin{align*}
x_\infty(2k)=(s_\infty(2k),\varphi_\infty(2k))&\sim \|x_\infty(2k)\| \bar{v}_1^s,\\ 
x_\infty(2k+1)=(s_\infty(2k+1),\varphi_\infty(2k+1))&\sim \|x_\infty(2k)\| \bar{v}_2^s, 
\end{align*}
with
$$
\|x_\infty(2k)\|=O(\Lambda^{-2k}).
$$
\end{lemma}

\begin{proof}
As in Lemma \ref{bounded domm}, this is a consequences of  the expansivity of the dynamics of $\mathcal{F}$ and of the fact that the forward orbits of  $h_\infty$ and $(12)$ have the same symbolic coding. More precisely, for each integer $k \geq 1$,
we let $[0,s_\infty(2k)]\ni s\mapsto \varphi^{(2k)}(s)$ be a continuous monotonic map such that $\varphi^{(2k)}(0)=0$  and $\varphi^{(2k)}(s_\infty(2k))=-\varphi_\infty(2k)$,
and  we set $\mathcal{W}^{(2k)}:=\{x^{(2k)}(s):=(s,\varphi^{(2k)}(s)):0\leq s \leq s_\infty(2k)\}$.
%
%
By the strict convexity of the obstacles,   $\mathcal{W}^{(2k)}$ is a dispersing wave front,  and for each $0 \leq j \leq k-1$, the projection of $\mathcal{F}^{2j}(\mathcal{W}^{(2k)})$ on the table is the arc of $\Gamma_1$ bounded by the points $A$ and $\gamma_1(s_\infty(k-j))$. 
Indeed, $\mathcal{F}^{2j}(s_\infty(2k),-\varphi_\infty(2k))=\mathcal{F}^{2j} \circ \mathcal{I} (x_\infty(2k))=\mathcal{I}\circ \mathcal{F}^{2(k-j)}(x_\infty(0))=(s_\infty(2(k-j)),-\varphi_\infty(2(k-j)))$.
As before, we obtain
\begin{align*}
C &\geq |\mathcal{F}^{2j}(\mathcal{W}^{(2k)})|=\left|\int_{0}^{s_{\infty}(2k)} \|D_{x^{(2k)}(s)} \mathcal{F}^{2j} \cdot (x^{(2k)})'(s)\|_p\, ds\right|\\
&\geq \Lambda^{2j} \left|\int_{0}^{s_{\infty}(2k)} \|(x^{(2k)})'(s)\|_p\, ds\right|\geq \Lambda^{2j}  c_{\text{min}}^{-1}\cdot \|x^{(2k)}(s_\infty(2k))-x^{(2k)}(0)\|,
\end{align*}
and then,
\begin{equation}\label{premie est}
\|x_\infty(2k)\| \leq c_{\text{min}} C \Lambda^{-2k}.
\end{equation}
In particular, $\lim_{k \to +\infty} \|x_\infty(2k)\|=0$, and then,
$$
\lim_{k \to +\infty} \frac{x_\infty(2k)}{\|x_\infty(2k)\|}=:\bar{v}^s_1,
$$
for some unit vector $\bar{v}^s_1 \in \mathrm{E}_{\mathcal{F}^2}^s (0_1,0)$.


For any integer $k \geq 0$, we also have
\begin{align*}
x_\infty(2k+1)&=\mathcal{F}(x_\infty(2k))-\mathcal{F}(0_1,0)\\
&\sim D_{(0_1,0)} \mathcal{F} \cdot x_\infty(2k)\\ 
&\sim D_{(0_1,0)} \mathcal{F} \cdot \|x_\infty(2k)\| \bar{v}_s^1\\ 
&=\|x_\infty(2k)\|  \bar{v}_2^s,
\end{align*}
with $\bar{v}_2^s:=D_{(0_1,0)} \mathcal{F}\cdot \bar{v}_s^1\in \mathrm{E}_{\mathcal{F}^2}^s(0_2,0)$.
\end{proof}

%

\section{Improved estimates on the parameters}
\label{sec:improved-estimates}

In this part, we keep the same notations as previously, and we get improved estimates using a change of coordinates.

For $j \in \{1,2\}$, the point $(0_j,0)$ is a saddle fixed point of $\mathcal{F}^2$, with eigenvalues $\lambda=\lambda(12)<1$ and $\mu=\mu(12):=\lambda^{-1}>1$.  We denote by
\begin{align*}
\mathrm{E}_{\mathcal{F}^2}^s(0_j,0)&:=\{v\in T_{(0_j,0)} \mathcal{M}: D_{(0_j,0)} \mathcal{F}^2\cdot v=\lambda v\},\\
\mathrm{E}_{\mathcal{F}^2}^u(0_j,0)&:=\{v\in T_{(0_j,0)} \mathcal{M}: D_{(0_j,0)} \mathcal{F}^2\cdot v=\lambda^{-1} v\},
\end{align*}
the stable, resp. unstable space of $\mathcal{F}^2$ at $(0_j,0)$, 
associated to the smallest eigenvalue $\lambda<1$, resp. largest eigenvalue $\mu>1$ of $(12)$.

The rest of this section is dedicated to the proof of the following result.

\begin{prop}\label{prop sympt hoc}
There exist a real number  $\xi_\infty \in \R$ and  two vectors $v_1^s\in \mathrm{E}_{\mathcal{F}^2}^s(0_1,0)$, $v_2^s\in \mathrm{E}_{\mathcal{F}^2}^s(0_2,0)$, with $\|v_1^s\|=1$ and $v_2^s:=D_{(0_1,0)}\mathcal{F}\cdot v_1^s$, such that for each integer $k \geq 0$, it holds
\begin{align*}
x_\infty(2k)=\lambda^k \xi_\infty\cdot v_1^s+O(\lambda^{\frac{3k}{2}}),\\ 
x_\infty(2k+1)=\lambda^k \xi_\infty\cdot v_2^s+O(\lambda^{\frac{3k}{2}}). 
\end{align*}
Furthermore, there exist an integer $n_0 \geq 0$ and two vectors $v_1^u\in \mathrm{E}^u_{\mathcal{F}^2}(0_1,0)$  and $v_2^u:=D_{(0_1,0)} \mathcal{F} \cdot v_1^u \in \mathrm{E}^u_{\mathcal{F}^2}(0_2,0)$ such that for  each integer
$n \geq n_0$, and for each integer $k\in \{0,\dots,\lceil\frac{n+1}{2}\rceil\}$, it holds
\begin{align*}
x_n(2k)-x_\infty(2k)&=\lambda^{n+1-k}\xi_\infty\cdot v_1^u+O(\lambda^{n-\frac{k}{2}}),\\
x_n(2k+1)-x_\infty(2k+1)&=\lambda^{n+1-k}\xi_\infty\cdot v_2^u+O(\lambda^{n-\frac{k}{2}}).
\end{align*}
\end{prop}
The proof of this proposition is given in Corollary \ref{coro sympt hoc} and Corollary \ref{main crororor} below.

\subsection{The linearization near a saddle fixed point}\label{section linerasition}

The points $(0_1,0)$ and $(0_2,0)$ are saddle fixed points of the square $\mathcal{F}^2$ of the billiard map. For $n\geq 0$ large enough, and for  each sufficiently large integer $k\in \{0,\dots,n+1\}$, the point $x_n(k)$ is in a neighborhood of one of those two points. To obtain the estimates of Proposition \ref{prop sympt hoc}, we use a change of coordinates as follows to linearize the dynamics.

By Lemma 23  in \cite{HKS}, for any $\varepsilon>0$,
there exist a neighborhood $\mathcal{U}$ of $(0_1,0)$ in the $(s,\varphi)$-plane, a neighborhood $\mathcal{V} \subset \R^2$ of $(0,0)$, and a $C^{1,\frac 12}$-diffeomorphism
$$
\Phi\colon \left\{
\begin{array}{rcl}
\mathcal{U} & \to & \mathcal{V},\\
(s,\varphi) & \mapsto & (\xi,\eta), 
\end{array}
\right.
$$
such that
\begin{equation*}
\Phi \circ \mathcal{F}^2 \circ \Phi^{-1} = D_{(0_1,0)} \mathcal{F}^2,\qquad \|\Phi-\mathrm{id}\|_{C^1} \leq \varepsilon, \qquad \|\Phi^{-1}-\mathrm{id}\|_{C^1} \leq \varepsilon,
\end{equation*}
and
\begin{equation*}
\Phi(z)-\Phi(z')=z-z'+O(\max(|z|^{\frac 12},|z'|^{\frac 12})|z-z'|).
\end{equation*}


Let $L_1=L_1(\mathcal{F}) \in \mathrm{SL}(2,\R)$ be a linear isomorphism such that $v_1^s:=L_1^{-1}(1,0)\in \mathrm{E}_{\mathcal{F}^2}^s(0_1,0)$, with $\|v_1^s\|=1$, and $L_1^{-1}(0,1)\in \mathrm{E}_{\mathcal{F}^2}^u(0_1,0)$.
By considering $\Psi:=L_1\circ \Phi$, with $\Phi$ as above, we deduce that for any $\varepsilon>0$,
there exist a neighborhood $\mathcal{U}$ of $(0_1,0)$, a neighborhood $\mathcal{V}$ of $(0,0)$, and a $C^{1,\frac 12}$-diffeomorphism $\Psi\colon \mathcal{U} \to \mathcal{V}$, such that
\begin{equation}\label{fseccc conjuugg}
\Psi \circ \mathcal{F}^2 \circ \Psi^{-1} = D_\lambda, \qquad \|\Psi-L_1\|_{C^1} \leq \varepsilon, \qquad \|\Psi^{-1}-L_1^{-1}\|_{C^1} \leq \varepsilon,
\end{equation}
and
\begin{align}\label{contole non}
\Psi(z)-\Psi(z')&=L_1(z-z')+O(\max(|z|^{\frac 12},|z'|^{\frac 12})|z-z'|),\\
\Psi^{-1}(z)-\Psi^{-1}(z')&=L_1^{-1}(z-z')+O(\max(|z|^{\frac 12},|z'|^{\frac 12})|z-z'|),\nonumber
\end{align}
where $D_\lambda:=\mathrm{diag}(\lambda,\lambda^{-1})\colon (\xi,\eta)\mapsto (\lambda\xi,\lambda^{-1} \eta)$. Recall that $\mathcal{I} \colon (s,\varphi)\mapsto (s,-\varphi)$. Set $\mathcal{I}^*:=\Psi \circ \mathcal{I} \circ \Psi^{-1}$, and let $L_2=L_2(\mathcal{F}) \in \mathrm{SL}(2,\R)$ be the linear isomorphism (in fact, a reflection)
\begin{equation}\label{definition llll2}
L_2:=L_1 \circ \mathcal{I} \circ L_1^{-1}.
\end{equation}
Note that $\mathcal{I}^*(0,0)=(0,0)$. Then, we also have
\begin{equation}\label{contole non bisbi}
\mathcal{I}^*(z)-\mathcal{I}^*(z')=L_2(z-z')+O(\max(|z|^{\frac 12},|z'|^{\frac 12})|z-z'|).
\end{equation}
In particular, 
$L_2$ is the linear part of $\mathcal{I}^*$ at the point $(0,0)$, i.e., $D_{(0,0)} \mathcal{I}^*=L_2$, and $\mathcal{I}^*(z)=L_2(z)+O(|z|^{\frac{3}{2}})$.\\

By Corollary \ref{homcolincc} and Lemma \ref{coroooooooo}, there exist $k_0,n_0 \geq 0$ such that for $n \geq n_0$, and for all $k \in \{k_0,\dots,n+1-k_0\}$, the point $x_n(2k)$ is in the neighborhood $\mathcal{U}$ of $(0_1,0)$. We denote by $(\xi_{n}(2k),\eta_{n}(2k))$ the coordinates of the point $\Psi(x_{n}(2k))$.
It is possible to extend the system of coordinates given by $\Psi$ to a neighborhood of the separatrices as follows: for $k\in \{0,\dots,k_0-1\}$, we let
$$
\Psi(x_n(2k))=(\xi_{n}(2k),\eta_{n}(2k)):=(\lambda^{k-k_0}\xi_{n}(2k_0),\lambda^{k_0-k}\eta_{n}(2k_0)),
$$
and analogously for $k\in \{n+1-k_0+1,\dots,n\}$.
Let us abbreviate
$$
(\xi_{n},\eta_{n}):=(\xi_{n}(0),\eta_{n}(0))=\Psi(x_n(0)),
$$
so that
$$
(\xi_{n}(2k),\eta_{n}(2k))=(\lambda^{k} \xi_n, \lambda^{-k} \eta_n),\quad \forall\, k \in \{0,\dots,n\}.
$$

By Lemma \ref{homcolincc}, the points $(x_\infty(k))_{k \geq 0}$ are on the stable manifold of $(0_1,0)$. Therefore, their images by $\Psi$ are on the  stable manifold  of the origin, which is here the horizontal axis $\{\eta=0\}$.
In our extended
system of coordinates, we then have
$$
\Psi(x_\infty(2k))=(\xi_\infty(2k),0)=(\lambda^k \xi_\infty,0), \quad \forall\, k \geq 0,
$$
for some $\xi_\infty \in \R$.

\subsection{Proof of Proposition \ref{prop sympt hoc}}\label{sec:prove-prop}

In this part, we keep the same notations as before.
Thanks to the above conjugacy, we can improve the estimates shown in Lemma \ref{homcolincc}.
\begin{corollary}\label{coro sympt hoc}
Let $v_1^s=L_1^{-1} (1,0)\in \mathrm{E}_{\mathcal{F}^2}^s(0_1,0)$ be as above, and set $v_2^s:=D_{(0_1,0)}\mathcal{F}\cdot v_1^s\in \mathrm{E}_{\mathcal{F}^2}^s(0_2,0)$. Then, we have the estimates
\begin{align*}
x_\infty(2k)=\lambda^k \xi_\infty\cdot v_1^s+O(\lambda^{\frac{3k}{2}}),\\ 
x_\infty(2k+1)=\lambda^k \xi_\infty\cdot v_2^s+O(\lambda^{\frac{3k}{2}}). 
\end{align*}
For $j =1,2$, we also have $v_j^s =\bar{v}_j^s$,  with $\bar{v}_j^s$ as in Lemma \ref{homcolincc}.
\end{corollary}

\begin{proof}
By \eqref{contole non}, for each $k \geq 0$, we have
\begin{align*}
x_\infty(2k)&=x_\infty(2k)-(0_1,0)=\Psi^{-1} (\xi(2k),0)-\Psi^{-1}(0,0)\\
&=\Psi^{-1}(\lambda^k \xi_\infty,0)-\Psi^{-1}(0,0)\\
&=\lambda^k  \xi_\infty \cdot L_1^{-1}  (1,0)+O(\lambda^{\frac{3k}{2}}).
\end{align*}
By comparing with the estimates of  Lemma \ref{homcolincc}, and since $v_1^s=L_1^{-1}(1,0)$ is taken to be of norm one, we see that
$v_1^s=\bar{v}_1^s$.  

Then, for odd integers,  in the same way as before, we obtain
\begin{align*}
x_\infty(2k+1)&=\mathcal{F}(x_\infty(2k))-\mathcal{F}(0_1,0)
= D_{(0_1,0)} \mathcal{F} \cdot x_\infty(2k)+O(\lambda^{2k})\\
&= \lambda^k \xi_\infty\, D_{(0_1,0)} \mathcal{F} \cdot v_1^s+O(\lambda^{\frac{3k}{2}}),
\end{align*}
which concludes.
\end{proof}

Recall that $\mathcal{I}\colon (s,\varphi)\mapsto (s,-\varphi)$, and that $\mathcal{I} \circ \mathcal{F} \circ \mathcal{I} = \mathcal{F}^{-1}$. Besides, we let $\mathcal{I}^*:=\Psi \circ \mathcal{I} \circ \Psi^{-1}$, and we let $L_2=L_2(\mathcal{F}) :=L_1 \circ \mathcal{I} \circ L_1^{-1}\in \mathrm{SL}(2,\R)$ be as in \eqref{definition llll2}.

\begin{lemma}\label{carac ldeux}
For some $\alpha \in \R^*$, we have
$$
L_2 (\xi,\eta)=(\alpha^{-1}\eta,\alpha\xi),\qquad \forall\, (\xi,\eta)\in \R^2.
$$
\end{lemma}

\begin{proof}
We have $\mathcal{I} \circ \mathcal{F}^2 \circ \mathcal{I}=\mathcal{F}^{-2}$. By \eqref{fseccc conjuugg} and the definition of $\mathcal{I}^*$, we thus obtain $\mathcal{I}^* \circ D_\lambda \circ \mathcal{I}^*=D_\lambda^{-1}$. Then, by \eqref{contole non bisbi}, and since $\mathcal{I}^*(0,0)=(0,0)$,  we deduce that $L_2 \circ  D_\lambda \circ L_2(z)=D_\lambda^{-1}(z)+O(|z|^{\frac{3}{2}})$, for all $z \in \R^2$. But the maps $L_2=L_1 \circ \mathcal{I} \circ L_1^{-1}\in \mathrm{SL}(2,\R)$ and $D_\lambda$ are linear. Therefore, by identifying the linear parts, we get $L_2 \circ  D_\lambda \circ L_2=D_\lambda^{-1}$. Now, we also have $\mathcal{I}_0 \circ  D_\lambda\circ \mathcal{I}_0=D_\lambda^{-1}$, with $\mathcal{I}_0\colon(\xi,\eta)\mapsto(\eta,\xi)$, hence $\mathcal{I}_0 \circ L_2 \in \mathrm{SL}(2,\R)$ commutes with the diagonal matrix $D_\lambda$. Since the $\mathrm{SL}(2,\R)$-centralizer of $D_\lambda$ is reduced to the subset of diagonal matrices in $\mathrm{SL}(2,\R)$, we conclude that $L_2 = \mathcal{I}_0 \circ D_\alpha\colon (\xi,\eta)\mapsto (\alpha^{-1}\eta,\alpha \xi)$, for some $\alpha \in \R^*$.
\end{proof}

\begin{lemma}\label{lemma eaxxu}
Let $n \geq n_0$. 
It holds
\begin{equation}\label{etan xin}
\eta_n=\alpha \lambda^{n+1} \xi_n+O(\lambda^{\frac{5n}{4}}),\qquad \xi_n=\xi_\infty(1+\lambda^{n+1})+O(\lambda^{\frac{5n}{4}}).
\end{equation}
Then, for any $k \in\{0,\dots,n+1\}$, we have
\begin{equation}\label{rel bizar}
\Psi(x_n(2k))=(\lambda^k\xi_n,\lambda^{-k}\eta_n)=\xi_\infty(\lambda^{k}+\lambda^{n+1+k},\alpha\lambda^{n+1-k})+ O(\lambda^{\frac{5n}{4}-k}).
\end{equation}
\end{lemma}

\begin{proof}
Let $n \geq n_0$, and let $k \in\{0,\dots,n+1\}$. We obtain successively
\begin{align*}
(\lambda^{n+1-k} \xi_n,\lambda^{k-n-1} \eta_n)&=(\xi_n(2n+2-2k),\eta_n(2n+2-2k))\\
&=\Psi\circ \mathcal{F}^{2n+2-2k}(x_n(0))\\
&=\Psi\circ \mathcal{F}^{2n+2-2k}\circ \mathcal{I}(x_n(0))\\
&=\Psi \circ \mathcal{I} \circ \mathcal{F}^{2k-2n-2}(x_n(0))\\
&=(\Psi\circ \mathcal{I} \circ \Psi^{-1}) \circ \Psi( \mathcal{F}^{2k}(x_n(0)))\\
&=\mathcal{I}^* (\xi_n(2k),\eta_n(2k))-\mathcal{I}^*(0,0)\\
&=L_2 (\lambda^{k} \xi_n,\lambda^{-k} \eta_n)
+O(\|(\lambda^k \xi_n,\lambda^{-k}\eta_n)\|^{\frac 32})\\
&=(\alpha^{-1}\lambda^{-k} \eta_n,\alpha\lambda^{k} \xi_n)
+\lambda^{\frac{3k}{2}} \cdot O(\|(\xi_n,\lambda^{-2k} \eta_n)\|^{\frac 32}).
\end{align*}
Indeed, we have $\mathcal{I}(x_n(0))=\mathcal{I}(s_n(0),0)=(s_n(0),0)=x_n(0)$, and $\mathcal{F}^{2k-2n-2}(x_n(0))=\mathcal{F}^{2k}(x_n(0))$, by the $(2n+2)$-periodicity of  $h_n$. Here, we have also used that $\mathcal{I}^*(0,0)=(0,0)$. Then, the last two lines above follow from \eqref{contole non bisbi} and Lemma \ref{carac ldeux}. 

Multiplying by $\alpha \lambda^k$, and projecting on the first component, we deduce that $\eta_n=\alpha \lambda^{n+1} \xi_n+\lambda^{\frac{5k}{2}} \cdot O(\|(\xi_n,\lambda^{-2k} \eta_n)\|^{\frac 32})$. Letting $k:=\lceil \frac{n}{2}\rceil$, we thus obtain $\eta_n=\alpha \lambda^{n+1} \xi_n+\lambda^{\frac{5n}{4}} \cdot O(\|(\xi_n,\lambda^{-n} \eta_n)\|^{\frac 32})=\alpha \lambda^{n+1} \xi_n+\lambda^{\frac{5n}{4}} \cdot O(|\xi_n|^{\frac 32})$, i.e.,
$$
\eta_n=\alpha \lambda^{n+1} \xi_n+O(\lambda^{\frac{5n}{4}}),
$$
which concludes the proof of the first estimate in \eqref{etan xin}.

For  any $k \in\{0,\dots,n+1\}$, we thus have
\begin{equation}\label{premier estimatesktk}
(\lambda^k \xi_n,\lambda^{-k} \eta_n)=\xi_n(\lambda^k,\alpha \lambda^{n+1-k})+O(\lambda^{\frac{5n}{4}-k}).
\end{equation}

Then, by linearity of $L_2$, and by \eqref{contole non bisbi}, we obtain
\begin{align*}
(\alpha^{-1} \eta_n,\alpha  (\xi_n-\xi_\infty))&=
L_2( \xi_n-\xi_\infty, \eta_n)\\
&=L_2(  \xi_n, \eta_n)-L_2( \xi_\infty,0)\\
&=\mathcal{I}^*( \xi_n,  \eta_n)-\mathcal{I}^*(  \xi_\infty,0)+O(\|(\xi_n-\xi_\infty, \eta_n)\|^{\frac 32}) \\
&=\Psi \circ \mathcal{I} (s_n(0),0)-\Psi \circ \mathcal{I} (s_\infty(0),0)+O(\|(\xi_n-\xi_\infty, \eta_n)\|^{\frac 32})\\
&=\Psi(s_n(0),0)-\Psi(s_\infty(0),0)+O(\|(\xi_n-\xi_\infty, \eta_n)\|^{\frac 32})\\
&=( \xi_n,\eta_n)-( \xi_\infty,0)+O(\|(\xi_n-\xi_\infty, \eta_n)\|^{\frac 32})\\
&=(\xi_n-\xi_\infty,\eta_n)+O(\|(\xi_n-\xi_\infty, \eta_n)\|^{\frac 32}).
\end{align*}
The estimate  of the error term above follows from the fact that $\mathcal{I}^*(z)=L_2(z)+O(|z|^{\frac{3}{2}})$, and from the definition of the extension of our system of coordinates to a neighborhood of the separatrix (the differential of $\mathcal{I}^*$ at points which are on the separatrix is equal to $L_2$).
Projecting on the first component,  
we thus obtain $\lambda^{n+1}\xi_n=\alpha^{-1} \eta_n+O(\lambda^{\frac{5n}{4}})=\xi_n-\xi_\infty+O(\lambda^{\frac{5n}{4}})$, hence
$$
\xi_n=\frac{\xi_\infty}{1-\lambda^{n+1}}+O(\lambda^{\frac{5n}{4}})=\xi_\infty(1+\lambda^{n+1})+O(\lambda^{\frac{5n}{4}}),
$$
which gives the second estimate in \eqref{etan xin}.
Combining this with \eqref{premier estimatesktk}, this concludes the proof of \eqref{rel bizar}.
\end{proof}

\begin{corollary}\label{main crororor}
Let $v_1^u:=\alpha L_1^{-1}(0,1)\in \mathrm{E}^u_{\mathcal{F}^2}(0_1,0)$, and set $v_2^u:=D_{(0_1,0)} \mathcal{F} \cdot v_1^u \in \mathrm{E}^u_{\mathcal{F}^2}(0_2,0)$. Then, for  each
$n \geq n_0$, and for each $k\in \{0,\dots,\lceil\frac{n+1}{2}\rceil\}$, we have
\begin{align*}
x_n(2k)-x_\infty(2k)&=\lambda^{n+1-k}\xi_\infty\cdot v_1^u+O(\lambda^{n-\frac{k}{2}}),\\
x_n(2k+1)-x_\infty(2k+1)&=\lambda^{n+1-k}\xi_\infty\cdot v_2^u+O(\lambda^{n-\frac{k}{2}}).
\end{align*}
\end{corollary}

\begin{proof}
%
Let $k\in \{0,\dots,\lceil\frac{n+1}{2}\rceil\}$. Recall that $\Psi(x_\infty(2k))=(\lambda^k \xi_\infty,0)$. Therefore, we deduce from \eqref{rel bizar} that
\begin{align*}
\Psi(x_n(2k))-\Psi(x_\infty(2k))&=\xi_\infty(\lambda^{n+1+k},\alpha\lambda^{n+1-k})+O(\lambda^{\frac{5n}{4}-k})\\
&=\alpha\lambda^{n+1-k}\xi_\infty(\alpha^{-1}\lambda^{2k},1)+O(\lambda^{\frac{5n}{4}-k}).
\end{align*}
By \eqref{contole non}, and since $0 \leq k \leq\lceil \frac{n+1}{2}\rceil$, we conclude that
$$
x_n(2k)-x_\infty(2k)=\lambda^{n+1-k}\xi_\infty\cdot \alpha L_1^{-1}(0,1)+O(\lambda^{n-\frac{k}{2}}).
$$
We have $(0,1) \in \mathrm{E}_{D_\lambda}^{u}(0,0)$, and $L_1^{-1} (0,1) \in \mathrm{E}_{\mathcal{F}^2}^u(0_1,0)$. 
Set $v_1^u:=\alpha L_1^{-1} (0,1)$. We have $v_1^u\in \mathrm{E}_{\mathcal{F}^2}^u(0_1,0)$, and for each $k\in \{0,\dots,\lceil\frac{n+1}{2}\rceil\}$, it holds
$$
x_n(2k)-x_\infty(2k)=\lambda^{n+1-k}\xi_\infty\cdot v_1^u+O(\lambda^{n-\frac{k}{2}}).
$$

Applying $\mathcal{F}$, we get
\begin{align*}
x_n(2k+1)-x_\infty(2k+1)&=\mathcal{F}(x_n(2k))-\mathcal{F}(x_\infty(2k))\\
&=D_{(0_1,0)}\mathcal{F}\cdot (x_n(2k)-x_\infty(2k))+O(\lambda^{2(n+1-k)})\\
&=\lambda^{n+1-k}\xi_\infty\, D_{(0_1,0)}\mathcal{F}\cdot v_1^u+O(\lambda^{n-\frac{k}{2}}),
\end{align*}
with $D_{(0_1,0)}\mathcal{F}\cdot v_1^u=:v_2^u$, and $v_2^u \in D_{(0_1,0)}\mathcal{F}\cdot \mathrm{E}_{\mathcal{F}^2}^u(0_1,0)=\mathrm{E}_{\mathcal{F}^2}^u(0_2,0)$.
\end{proof}

\begin{remark}\label{remarque geomee}
It is also possible to  derive the previous result (Lemma \ref{lemma eaxxu}) by more geometric arguments. In the rest of this section, we explain how similar estimates can be obtained by considering the growth of certain dispersing waves fronts.
\end{remark}

Given $n \geq 0$, we consider the dispersing wave front
$\mathcal{W}_{n,\infty}(0):=\{x(s):=(s,0):s_n\leq s \leq S_n\}$, where
$s_n:=\min(s_n(0),s_\infty(0))$, and $S_n:=\max(s_n(0),s_\infty(0))$.
In other terms,  $\mathcal{W}_{n,\infty}(0)$ is the arc of $\Gamma_3$
connecting the points with respective parameters $s_n(0)$ and $s_\infty(0)$,
each point being endowed with $0$ angle. Given any
$x=(s,0) \in \mathcal{W}_{n,\infty}(0)$, we denote by $x(s,k)$,
$0 \leq k \leq 2n+1$, its forward iterates.
\begin{figure}[H]
\begin{center}
    \includegraphics [width=15cm]{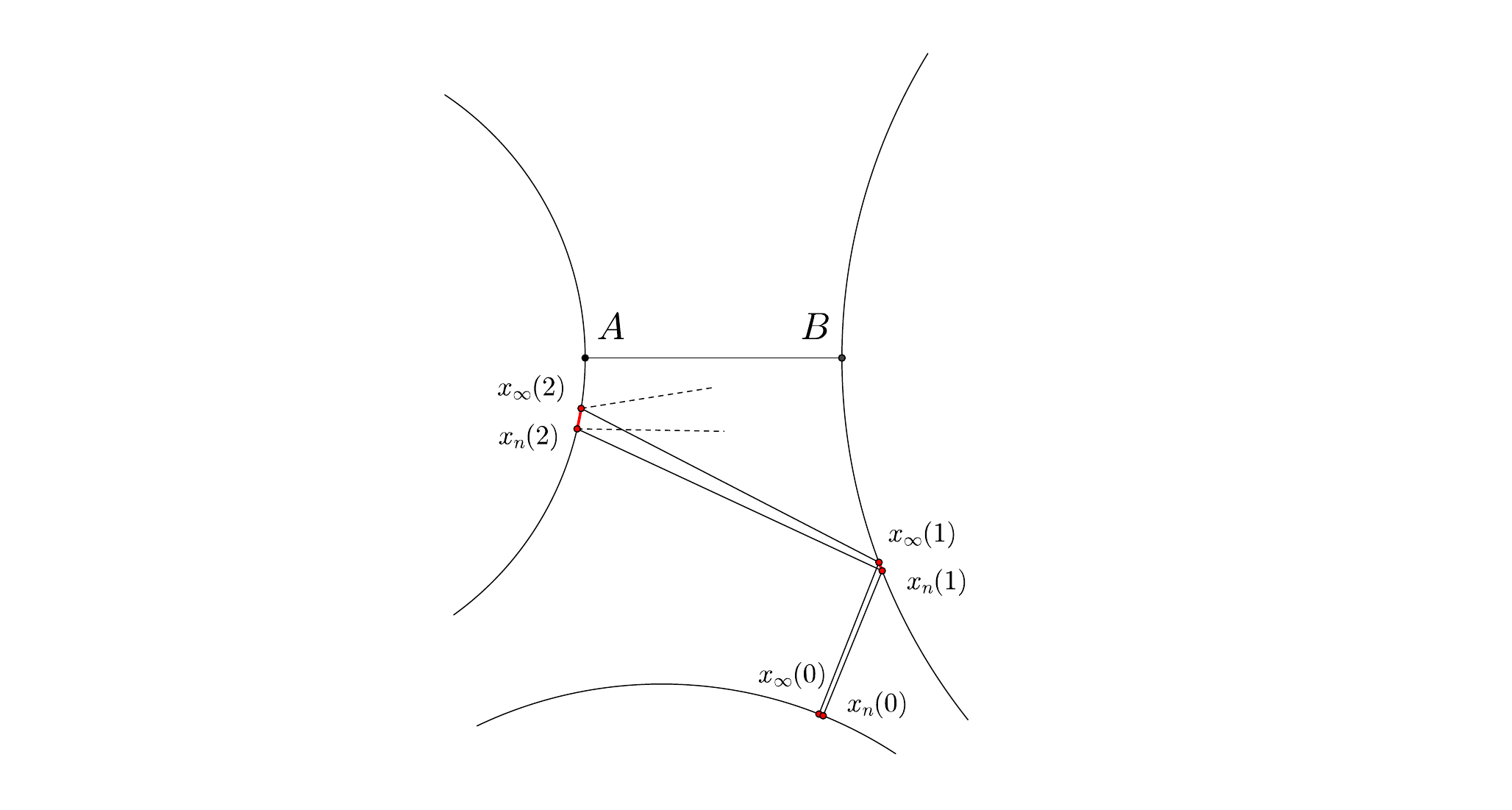}
\end{center}
\caption{Approximating periodic and homoclinic orbits}
\label{fig:hom-period-appr}
\end{figure}

The respective symbolic codings of the forward orbits of $x_n(0)$ and $x_\infty(0)$ are
\begin{equation*}
3\underbrace{21212\dots 21212}_{2n+1} 321\dots,\qquad
3\underbrace{21212\dots 21212}_{2n+1} 121\dots,
\end{equation*}
By strict convexity, it follows that for any $1 \leq k \leq n$, $\mathcal{W}_{n,\infty}(2k):=\mathcal{F}^{2k}(\mathcal{W}_{n,\infty})$ is a dispersing wave front whose projection on the configuration space is the arc of $\Gamma_1$ between the points with respective parameters $s_n(2k)$ and $s_\infty(2k)$. In particular, we have $|\mathcal{W}_{n,\infty}(2k)|\leq C_0$, for some uniform constant $C_0>0$. Besides, by Lemma \ref{homcolincc}, we know that $\lim_{n \to +\infty} x_\infty(2n)=(0_1,0)$,  and by Corollary \ref{coroooooooo}, we also have $\lim_{n \to +\infty} x_n(2n)=\mathcal{I} (x_\infty(2))=(s_\infty(2),-\varphi_\infty(2))$. Therefore, the sequence of wave fronts $(\mathcal{W}_{n,\infty}(2n))_{n \geq 0}$ converges pointwise, and there exists $W>0$ such that $\lim_{n \to+\infty}|\mathcal{W}_{n,\infty}(2n)|=W$.

Besides, for any $1 \leq k \leq n$, it holds
\begin{align*}
|\mathcal{W}_{n,\infty}(2k)|&=\int_{\mathcal{W}_{n,\infty}(0)} \|D_{x} \mathcal{F}^{2k} \cdot dx\|_p
= \left|\int_{s_{n}(0)}^{s_{\infty}(0)} \|D_{(s,0)} \mathcal{F}^{2k} \cdot (1,0)\|_p\, ds\right|\\
&= \int_{s_{n}}^{S_n} \left\|\prod_{j=0}^{k-1} D_{x(s,2j)}\mathcal{F}^2\cdot (1,0)\right\|_p\, ds.
\end{align*}

On the one hand,  the orbit $h_\infty=(x_\infty(k))_{k \geq 0}$ is homoclinic to the period two orbit $(12)$ whose Lyapunov exponent is equal to $\frac 12 \log \mu$, where $\mu=\lambda^{-1}>1$. Since the $s$-coordinate corresponds to the position on the obstacles, whose curvature is nonzero, the vector $(1,0)$ has some non-vanishing component along the unstable direction. Thus, for $k=n$, we obtain
$$
\left\|\prod_{j=0}^{n-1} D_{x(s_\infty(0),2j)}\mathcal{F}^2\cdot (1,0)\right\|_p=\left\|\prod_{j=0}^{n-1} D_{x_\infty(0)}\mathcal{F}^2\cdot (1,0)\right\|_p\sim K(s_\infty(0)) \cdot \mu^n,
$$
for some constant $K(s_\infty(0))>0$.

On the other hand, by Corollary \ref{coroooooooo}, we get
\begin{align*}
\left\|\prod_{j=0}^{n-1} D_{x_n(2j)}\mathcal{F}^2\cdot (1,0)\right\|_p&\sim \left\|\prod_{j=\lceil \frac n2\rceil}^{2} D_{\mathcal{I}(x_\infty(2j))}\mathcal{F}^2 \circ \prod_{j=0}^{\lceil \frac n2\rceil-1} D_{x_\infty(2j)}\mathcal{F}^2\cdot (1,0)\right\|_p\\
&\sim K(s_n(0)) \cdot \mu^n,
\end{align*}
for some constant $K(s_n(0))>0$.

Then, by strict convexity of the table, there exists a monotonic map $K_n\colon[s_n,S_n]\to \R_+^*$, $t\mapsto K_n(t)$, such that $K_n(s_n)=K(s_n)$, $K_n(S_n)=K(S_n)$, and
$$
\left\|\prod_{j=0}^{n-1} D_{x(s,2j)}\mathcal{F}^2\cdot (1,0)\right\|_p \sim K_n(s) \cdot \mu^n.
$$
By Corollary \ref{coroooooooo},  we have $\lim_{n \to+\infty}|\mathcal{W}_{n,\infty}(0)|=\lim_{n \to+\infty} (S_n-s_n)=0$, and then, 
\begin{equation*}
W\sim |\mathcal{W}_{n,\infty}(2n)|\sim \left( \int_{s_n}^{S_n} K_n(s)\, ds\right) \cdot \mu^n\sim  \left|s_n(0)-s_\infty(0)\right|\cdot K(s_\infty(0))\mu^n.
\end{equation*}
By the fact that $x_n(0)=(s_n(0),0)$ and $x_\infty(0)=(s_\infty(0),0)$, we deduce that
$$
x_n(0)-x_\infty(0)\sim  \lambda^n\cdot v,
$$
for some vector $v \in \R^2$ with $\|v\|=K(s_\infty(0))^{-1} W>0$.
In particular, we see that we recover the same kind of estimate as in Corollary \ref{main crororor}, for $k=0$.
To obtain estimates on the forward iterates $x_\infty(2k)$ and $x_n(2k)$, with $0 \leq k \leq n$, we just have to apply the dynamics. For instance, we have
\begin{align*}
x_n(1)-x_\infty(1) = \mathcal{F}( x_n(0))-\mathcal{F} (x_\infty(0))\sim D_{x_\infty(0)} \mathcal{F} \cdot (x_n(0)- x_\infty(0)).
\end{align*}
In the same way, by Corollary \ref{coroooooooo}, for any $k \in\{0,\dots, n+1\}$, we get
\begin{align*}
&x_n(k)-x_\infty(k) = \mathcal{F}^k( x_n(0))-\mathcal{F}^k (x_\infty(0))\\
&\sim \prod_{j=0}^{k-1} D_{x_\infty(j)} \mathcal{F} \cdot (x_n(0)- x_\infty(0))=(s_n(0)-s_\infty(0)) \cdot \prod_{j=0}^{k-1} D_{x_\infty(j)} \mathcal{F} \cdot (1,0).
\end{align*}
With the notations introduced above, we get $\left\| \prod_{j=0}^{k-1} D_{x_\infty(2j)} \mathcal{F}^2 \cdot (1,0)\right\| \sim K(s_\infty(0)) \cdot \mu^k$. In analogy with Corollary \ref{main crororor}, for $k \in \{0,\dots,n+1\}$, we  thus have 
$$
\|x_n(2k)-x_\infty(2k)\| \sim W\cdot \lambda^{n-k}.
$$

\section{Consequences on the Marked Length Spectrum}
\label{sec:consec-MLS}

\subsection{Further remarks on the asymptotic constants}

In this part, we want to relate the asymptotic constants associated to the vectors $v_1^s,v_2^s,v_1^u,v_2^u$ defined above.
Here, we let $v_1^s=L_1^{-1} (1,0)\in \mathrm{E}_{\mathcal{F}^2}^s(0_1,0)$,   $v_2^s=D_{(0_1,0)}\mathcal{F}\cdot v_1^s\in \mathrm{E}_{\mathcal{F}^2}^s(0_2,0)$ be as in Corollary \ref{coro sympt hoc}, and we let $v_1^u:=\alpha L_1^{-1} (0,1) \in \mathrm{E}^u_{\mathcal{F}^2}(0_1,0)$, $v_2^u=D_{(0_1,0)} \mathcal{F} \cdot v_1^u \in \mathrm{E}^u_{\mathcal{F}^2}(0_2,0)$ be as in Corollary \ref{main crororor}.

Let us denote
\begin{align*}
v_1^s=(C_{1,s}^s,C_{1,\varphi}^s), &\qquad  v_2^s=(C_{2,s}^s,C_{2,\varphi}^s),\\
v_1^u=(C_{1,s}^u,C_{1,\varphi}^u), &\qquad  v_2^u=(C_{2,s}^u,C_{2,\varphi}^u),
\end{align*}
and recall that  for $\{i,j\}=\{1,2\}$, we have
\begin{equation}\label{matrice df}
D_{(0_i,0)}\mathcal{F}=-\begin{pmatrix}
\frac{\ell}{R_i}+1 & \ell\\
\frac{\ell}{R_iR_j}+\frac{1}{R_i}+\frac{1}{R_j} & \frac{\ell}{R_j}+1
\end{pmatrix}=-\begin{pmatrix}
\alpha_i & \ell\\
\gamma & \alpha_j
\end{pmatrix},
\end{equation}
with $\ell:=\frac{\mathcal{L}(12)}{2}=h(0_1,0_2)$, $\alpha_1:= \frac{\ell}{R_1}+1$,  $\alpha_2:= \frac{\ell}{R_2}+1$, and $\gamma:=\ell^{-1}(\alpha_1 \alpha_2 - 1)$.

In particular, we get
\begin{equation}\label{matrice df carre}
D_{(0_i,0)}\mathcal{F}^2=\begin{pmatrix}
2\alpha_1 \alpha_2- 1 & 2 \alpha_j\ell\\
2 \alpha_i \gamma & 2\alpha_1 \alpha_2- 1
\end{pmatrix}=\frac 12\begin{pmatrix}
\lambda+\lambda^{-1}& 4 \alpha_j\ell\\
4 \alpha_i \gamma & \lambda+\lambda^{-1}
\end{pmatrix}.
\end{equation}
Indeed, we know that $\mathrm{tr}(D_{(0_1,0)}\mathcal{F}^2)=\mathrm{tr}(D_{(0_2,0)}\mathcal{F}^2)=\lambda+\lambda^{-1}$, since $\lambda,\lambda^{-1}$ are the common eigenvalues of $D_{(0_1,0)}\mathcal{F}^2$ and $D_{(0_2,0)}\mathcal{F}^2$.

\begin{corollary}
We have
\begin{equation}\label{estim constant c1s}
\frac{C_{1,s}^s}{C_{1,\varphi}^s}=-\frac{4 \alpha_2\lambda}{1-\lambda^2}\ell<0,\qquad
\frac{C_{1,s}^u}{C_{1,\varphi}^u}=\frac{4 \alpha_2 \lambda}{1-\lambda^2}\ell>0,
\end{equation}
and
\begin{equation}\label{estim constant c2s}
\frac{C_{2,s}^s}{C_{2,\varphi}^s}=-\frac{4 \alpha_1\lambda}{1-\lambda^2}\ell<0,\qquad
\frac{C_{2,s}^u}{C_{2,\varphi}^u}=\frac{4 \alpha_1\lambda}{1-\lambda^2}\ell>0.
\end{equation}
Moreover, we have
\begin{equation}\label{etsim socnt cs}
\frac{C_{2,s}^s}{C_{1,\varphi}^s}=\frac{2 \lambda}{1-\lambda}\ell>0,\qquad \frac{C_{2,s}^u}{C_{1,\varphi}^u}=-\frac{2}{1-\lambda}\ell<0.
\end{equation}
\end{corollary}

\begin{proof}
Since $v_1^s=(C_{1,s}^s,C_{1,\varphi}^s)\in \mathrm{E}_{\mathcal{F}^2}^s(0_1,0)$, then by definition, and by \eqref{matrice df carre}, we have
$$
\frac{\lambda+\lambda^{-1}}{2} C_{1,s}^s + 2 \alpha_2 \ell C_{1,\varphi}^s=\lambda C_{1,s}^s,
$$
and the first equality of \eqref{estim constant c1s} follows. The second one is proved analogously, since $v_1^u=(C_{1,s}^u,C_{1,\varphi}^u)\in \mathrm{E}_{\mathcal{F}^2}^u(0_1,0)$.
We argue similarly for \eqref{estim constant c2s}, since $v_2^s\in \mathrm{E}_{\mathcal{F}^2}^s(0_2,0)$, and $v_2^u\in \mathrm{E}_{\mathcal{F}^2}^u(0_2,0)$.

Now, by definition,  we have $v_2^s= D_{(0_1,0)} \mathcal{F} \cdot v_1^s$, and then,  \eqref{matrice df} and \eqref{estim constant c1s} yield
$$
C_{2,s}^s=-\alpha_1 C_{1,s}^s -\ell C_{1,\varphi}^s=\left(\frac{4 \alpha_1\alpha_2}{\lambda^{-1}-\lambda}-1\right)\ell C_{1,\varphi}^s.
$$
Moreover, by the definition of $D_{(0_1,0)} \mathcal{F}^2$, we see that
\begin{equation}\label{equation recatrace}
\mathrm{tr}(D_{(0_1,0)} \mathcal{F}^2)=4 \alpha_1\alpha_2-2=\lambda+\lambda^{-1},
\end{equation}
hence
$$
\frac{C_{2,s}^s}{C_{1,\varphi}^s}=\frac{2 (1+\lambda)}{\lambda^{-1}-\lambda}\ell=\frac{2 (1+\lambda)}{\lambda^{-1}(1-\lambda^2)}\ell=\frac{2 \lambda}{1-\lambda}\ell.
$$
Similarly, we have $v_2^u= D_{(0_1,0)} \mathcal{F} \cdot v_1^u$, and
$$
C_{2,s}^u=-\alpha_1 C_{1,s}^u -\ell C_{1,\varphi}^u=-\left(\frac{4 \alpha_1\alpha_2}{\lambda^{-1}-\lambda}+1\right)\ell C_{1,\varphi}^u=-\frac{2}{1-\lambda}\ell C_{1,\varphi}^u,
$$
which gives \eqref{etsim socnt cs}.
\end{proof}

\begin{corollary}\label{neffffff}
The following relations hold between the asymptotic constants $C_{1,\varphi}^s,C_{2,\varphi}^s,C_{1,\varphi}^u$ and $C_{2,\varphi}^u$:
\begin{equation}\label{lien cohi}
\frac{C_{2,\varphi}^s}{C_{1,\varphi}^s}=-\frac{1+\lambda}{2 \alpha_1},\qquad\quad \frac{C_{1,\varphi}^u}{C_{1,\varphi}^s}=-1,\qquad\quad
\frac{C_{2,\varphi}^u}{C_{1,\varphi}^s}=\frac{1+\lambda^{-1}}{2 \alpha_1}.
\end{equation}
Since we will need it in the following, we also compute:
\begin{equation}\label{lien bis cohi}
\left(\frac{C_{2,\varphi}^s}{C_{1,\varphi}^s}\right)^2=\frac{\lambda\alpha_2}{\alpha_1},\qquad\quad \left(\frac{C_{2,\varphi}^u}{C_{1,\varphi}^s}\right)^2=\frac{\alpha_2}{\lambda\alpha_1}.
\end{equation}
\end{corollary}

\begin{proof}
The first equality follows immediately from \eqref{estim constant c2s} and \eqref{etsim socnt cs}, by writing
$$
\frac{C_{2,\varphi}^s}{C_{1,\varphi}^s}=\frac{C_{2,\varphi}^s}{C_{2,s}^s}\cdot \frac{C_{2,s}^s}{C_{1,\varphi}^s}=-\frac{1-\lambda^2}{4\alpha_1\lambda}\cdot\frac{2 \lambda}{1-\lambda}=-\frac{1+\lambda}{2\alpha_1}.
$$
To prove the other identities, we argue as follows. Let us assume that $n=2m-1$, for some integer $m \geq 1$. By the estimates of Corollary \ref{main crororor}, for $k=m$, we get
$$
x_n(2m)-x_\infty(2m)\sim \lambda^{n+1-m} \xi_\infty \cdot v_1^u=\lambda^m \xi_\infty \cdot v_1^u.
$$
By Lemma \ref{lemma palind}, we also have $x_n(n+1)=x_n(2m)=(s_n(2m),0)$, i.e., $\varphi_n(2m)=0$. Thus, by projecting the previous relation on the second coordinate, we get
\begin{equation}\label{first ififnrifnig}
\varphi_\infty(2m)\sim -C_{1,\varphi}^u \cdot \lambda^m \xi_\infty.
\end{equation}
On the other hand, by Corollary \ref{coro sympt hoc}, we also have
$$
x_\infty(2m)=(s_\infty(2m),\varphi_\infty(2m))\sim \lambda^m \xi_\infty \cdot v_1^s,
$$
and then, the projection of this relation on the second coordinate yields
$$
\varphi_\infty(2m) \sim C_{1,\varphi}^s \cdot \lambda^m \xi_\infty.
$$
Combining this with \eqref{first ififnrifnig}, we obtain
\begin{equation}\label{first c1vzrhi}
\frac{C_{1,\varphi}^u }{C_{1,\varphi}^s}=-1,
\end{equation}
which gives the second equality in \eqref{lien cohi}.

Let us  show the last one. By \eqref{estim constant c2s}, \eqref{etsim socnt cs} and \eqref{first c1vzrhi}, we get
$$
\frac{C_{2,\varphi}^u}{C_{1,\varphi}^s}=\frac{C_{2,\varphi}^u}{C_{2,s}^u}\cdot \frac{C_{2,s}^u}{C_{1,\varphi}^u}\cdot \frac{C_{1,\varphi}^u}{C_{1,\varphi}^s}=\frac{1-\lambda^2}{4 \alpha_1 \lambda}\cdot \left(-\frac{2}{1-\lambda}\right) \cdot (-1)=\frac{1+\lambda}{2 \alpha_1 \lambda},
$$
which concludes the proof of \eqref{lien cohi}.

To show \eqref{lien bis cohi}, we use the fact that $\lambda+\lambda^{-1}+2=4 \alpha_1\alpha_2$.  Also, by \eqref{lien cohi}, we have
$\frac{C_{2,\varphi}^s}{C_{2,\varphi}^u}=-\lambda$.  We get
$$
\left(\frac{C_{2,\varphi}^s}{C_{1,\varphi}^s}\right)^2=\lambda^2 \left(\frac{C_{2,\varphi}^u}{C_{1,\varphi}^s}\right)^2= \frac{1+2\lambda+\lambda^{2}}{4 \alpha_1^2}=\lambda \cdot \frac{\lambda+2+\lambda^{-1}}{4 \alpha_1^2}=\lambda \cdot  \frac{4 \alpha_1 \alpha_2}{4 \alpha_1^2}=\frac{\lambda \alpha_2}{\alpha_1}.
$$
\end{proof}

\subsection{Computing the $2$-jet of the generating function}

Recall that the function $h\colon(s,s')\mapsto \|\gamma(s)-\gamma(s')\|$ is generating for the billiard map $\mathcal{F}\colon(s,\varphi)\mapsto(s',\varphi')$. In other terms, for the $1$-form $\omega:=-\sin(\varphi)ds$, we have
$$
\mathcal{F}^* \omega-\omega=dh(s,s').
$$
In the following, given a point $\gamma(s) \in \partial \mathcal{D}$ in the boundary of the table, we denote by $R(s)>0$ the radius of curvature at $\gamma(s)$, defined as the radius of the osculating circle at this point. We also have $R(s)=\mathcal{K}(s)^{-1}$, where $\mathcal{K}(s)$ is the curvature. In the following lemma, we compute the second order Taylor expansion of the function $h$. Given $x=(s,\varphi)\in \mathcal{M}$ with $\mathcal{F}(s,\varphi)=(s',\varphi')$, we set
\begin{equation}\label{def zetafff}
\zeta^-(x):=\frac{R(s)\cos(\varphi)}{h(s,s')},\qquad \zeta^+(x):=\frac{R(s')\cos(\varphi')}{h(s,s')}.
\end{equation}

\begin{figure}[H]
\begin{center}
    \includegraphics [width=14cm]{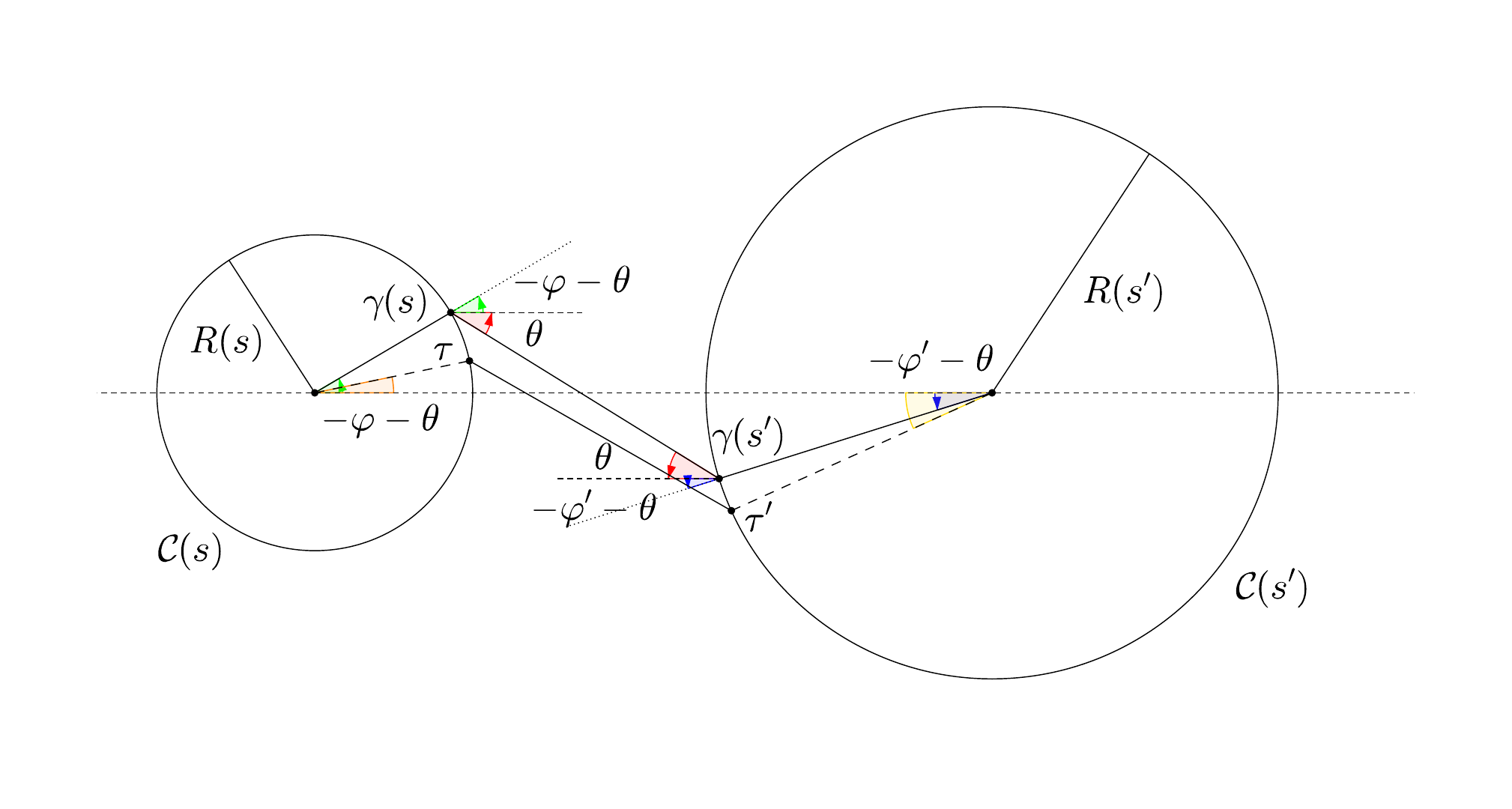}
\end{center}
\end{figure}

\begin{lemma}\label{lemma utile}
Let $x=(s,\varphi)$, $\bar{x}=(\bar{s},\bar{\varphi})\in \mathcal{M}$ be close to each other. We set $x'=(s',\varphi'):=\mathcal{F}(s,\varphi)$, and $\bar{x}'=(\bar{s}',\bar{\varphi}'):=\mathcal{F}(\bar{s},\bar{\varphi})$.  Then, we get
\begin{align*}
&\quad h(\bar{s},\bar{s}')-h(s,s')=
R(s) \sin(\varphi) (\varphi-\bar{\varphi})+R(s') \sin(\varphi') (\varphi'-\bar{\varphi}')\\
&\quad +\frac{h(s,s')}{2}\left(\zeta^-(x)\left(1+\zeta^-(x)\right)\cdot(\varphi-\bar{\varphi})^2+\zeta^+(x) \left(1+\zeta^+(x)\right)\cdot(\varphi'-\bar{\varphi}')^2\right.\\
&\quad \left.+ 2 \zeta^-(x)\zeta^+(x)\cdot (\varphi-\bar{\varphi})(\varphi'-\bar{\varphi}')+O(\|\bar{x}-x\|^3)\right).
\end{align*}
\end{lemma}

\begin{proof}
Assume that $x=(s,\varphi)\in \mathcal{M}$ and $\bar{x}=(\bar{s},\bar{\varphi})\in \mathcal{M}$ are close to each other, and let $(s',\varphi'):=\mathcal{F}(s,\varphi)$, and $(\bar{s}',\bar{\varphi}'):=\mathcal{F}(\bar{s},\bar{\varphi})$.

We denote by  $\mathcal{C}(s)$, $\mathcal{C}(s')$ the respective osculating circles at the  points $\gamma(s)$ and $\gamma(s')$, with   radii $R(s)$, $R(s')$, and we set $v(s,s'):=\gamma(s)-\gamma(s')$, so that $h(s,s')=\|v(s,s')\|$. Then $\varphi\in[-\frac{\pi}{2},\frac{\pi}{2}]$, resp. $\varphi'\in[-\frac{\pi}{2},\frac{\pi}{2}]$, is the oriented angle between the normal to $\mathcal{C}(s)$ at $\gamma(s)$ and the vector $v(s,s')$, resp. between the normal to $\mathcal{C}(s')$ at $\gamma(s')$ and the vector $-v(s,s')$. We may also assume that the line segment connecting the centers of $\mathcal{C}(s)$ and $\mathcal{C}(s')$ is horizontal, and we  denote by $\theta=\theta(s,s')$ the oriented angle between the horizontal and the vector $v(s,s')$. We also denote by $\tau$ and $\tau'$ the points obtained by projecting radially the  points $\gamma(\bar{s})$ and $\gamma(\bar{s}')$ on the respective osculating circles $\mathcal{C}(s)$ and $\mathcal{C}(s')$, and we define accordingly their angular coordinates $\psi$,  $\psi'$  on the circles $\mathcal{C}(s)$ and $\mathcal{C}(s')$. We have
$$
\psi=\bar{\varphi}+O((\bar{s}-s)^3),\qquad\psi'=\bar{\varphi}'+O((\bar{s}'-s')^3).
$$

Using complex notations in the plane, we obtain successively
\begin{align*}
&\quad \|\tau-\tau'\|\\
&=\|(\gamma(s)-\gamma(s'))+(\tau-\gamma(s))-(\tau'-\gamma(s'))\|\\
&=|v(s,s')-R(s) e^{-\mathrm{i} (\varphi+\theta)}(1-e^{-\mathrm{i} (\psi-\varphi)})-R(s') e^{-\mathrm{i} (\varphi'+\theta)}(1-e^{-\mathrm{i} (\psi'-\varphi')})|\\
&=h(s,s') \cdot \left|1-\frac{R(s)\mathrm{i} e^{-\mathrm{i} (\varphi+\theta)}}{v(s,s')}(\psi-\varphi)-\frac{R(s) e^{-\mathrm{i} (\varphi+\theta)}}{2 v(s,s')}(\psi-\varphi)^2\right.\\
&\quad -\left. \frac{R(s') \mathrm{i} e^{-\mathrm{i} (\varphi'+\theta)}}{v(s,s')}(\psi'-\varphi')-\frac{R(s')e^{-\mathrm{i} (\varphi'+\theta)}}{2v(s,s')}(\psi'-\varphi')^2+O((\psi-\varphi)^3)\right|\\
&=h(s,s') -\Re R(s) \mathrm{i} e^{-\mathrm{i} (\varphi+\theta)}\frac{\overline{v(s,s')}}{|v(s,s')|}(\psi-\varphi)-\Re R(s') \mathrm{i} e^{-\mathrm{i}(\varphi'+\theta)}\frac{\overline{v(s,s')}}{|v(s,s')|}(\psi'-\varphi')\\
&\quad -\Re \frac{R(s)}{2}  e^{-\mathrm{i} (\varphi+\theta)}\frac{\overline{v(s,s')}}{|v(s,s')|}(\psi-\varphi)^2-\Re \frac{R(s')}{2}e^{-\mathrm{i} (\varphi'+\theta)}\frac{\overline{v(s,s')}}{|v(s,s')|}(\psi'-\varphi')^2\\
&\quad + \frac{1}{2|v(s,s')|} \left(\Re R(s)  e^{-\mathrm{i} (\varphi+\theta)}\frac{\overline{v(s,s')}}{|v(s,s')|}(\psi-\varphi)+\Re R(s') e^{-\mathrm{i}(\varphi'+\theta)}\frac{\overline{v(s,s')}}{|v(s,s')|}(\psi'-\varphi')\right)^2\\
&\quad +O((\psi-\varphi)^3)\\
&=h(s,s')- R(s) \sin(\varphi) (\psi-\varphi)-R(s') \sin(\varphi') (\psi'-\varphi')\\
&\quad +\frac{R(s)}{2}\cos(\varphi) (\psi-\varphi)^2+\frac{R(s')}{2}\cos(\varphi') (\psi'-\varphi')^2\\
&\quad + \frac{1}{2 h(s,s')}(R(s)\cos(\varphi) (\psi-\varphi)+R(s')\cos(\varphi') (\psi'-\varphi'))^2+O((\psi-\varphi)^3),
\end{align*}
where we have used that $\frac{\overline{v(s,s')}}{|v(s,s')|}=-e^{\mathrm{i}\theta}$.

Now, by the property of osculating circles, we know that $\mathcal{C}(s)$, resp. $\mathcal{C}(s')$ is tangent to $\partial \mathcal{D}$ at $\gamma(s)$, resp. $\gamma(s')$ up to order three, i.e., $\|\tau-\gamma(\bar s)\|=O(|s-\bar{s}|^3)$, and $\|\tau'-\gamma(\bar{s}')\|=O(|s'-\bar{s}'|^3)=O(|s-\bar{s}|^3)$. Then, we also have $|\psi-\bar{\varphi}|=O(|s-\bar{s}|^3)$ and $|\psi'-\bar{\varphi}'|=O(|s'-\bar{s}'|^3)$. Therefore, we get
\begin{align*}
&\quad h(\bar{s},\bar{s}')-h(s,s')=
- R(s) \sin(\varphi) (\bar{\varphi}-\varphi)-R(s') \sin(\varphi') (\bar{\varphi}'-\varphi')\\
& +\frac{R(s)}{2}\cos(\varphi) (\bar{\varphi}-\varphi)^2+\frac{R(s')}{2}\cos(\varphi') (\bar{\varphi}'-\varphi')^2\\
& + \frac{1}{2 h(s,s')}(R(s)\cos(\varphi) (\bar{\varphi}-\varphi)+R(s')\cos(\varphi') (\bar{\varphi}'-\varphi'))^2+O(\|\bar{x}-x\|^3)=\\
&- R(s) \sin(\varphi) (\bar{\varphi}-\varphi)-R(s') \sin(\varphi') (\bar{\varphi}'-\varphi')\\
& +\frac{R(s)\cos(\varphi)}{2}\left(1+\frac{R(s)\cos(\varphi)}{h(s,s')}\right)(\bar{\varphi}-\varphi)^2\\
&+\frac{R(s')\cos(\varphi')}{2}\left(1+\frac{R(s')\cos(\varphi')}{h(s,s')}\right)
\cdot (\bar{\varphi}'-\varphi')^2 \\
&+ \frac{R(s)\cos(\varphi)R(s')\cos(\varphi')}{ h(s,s')}\cdot (\bar{\varphi}-\varphi)(\bar{\varphi}'-\varphi')+O(\|\bar{x}-x\|^3),
\end{align*}
as claimed.
\end{proof}

\subsection{Estimates on the length of periodic orbits}\label{section estimmmm}

In this part, we apply the  results of the previous sections to show that the lengths of the periodic orbits $(h_n)_{n \geq 0}$ can be combined in order to recover some geometric information on the billiard table $\mathcal{D}$.

Recall that $\sigma=(12)$, and that for any integer $n\geq 0$, the periodic orbit $h_n$ is defined as $h_n:=(32(12)^n)$. We let $(x_n(k)=(s_n(k),\varphi_n(k)))_{0 \leq k \leq 2n+1}$ be the coordinates of the points in $h_n$.   We also recall that given a periodic orbit $(x(k)=(s(k),\varphi(k))_{0 \leq k \leq p}$ of period $p \geq 2$ labelled by a word $\rho\in \{1,2,3\}^p$, we denote by $\mathcal{L}(\rho)$ the length of this orbit. By the definition of the generating function $h$, we have
$$
\mathcal{L}(\rho)=\sum_{k=0}^{p-1} h(s(k),s(k+1)).
$$

With our previous notations, the $(s,\varphi)$-coordinates of the points in the period two orbit $(12)$ are $x(k)=(s(k),\varphi(k))=(0_1,0)$, for any even integer $k$, and $x(k)=(s(k),\varphi(k))=(0_2,0)$, for any odd integer $k$.
As we have seen in Lemma \ref{lemma palind}, for any $n \geq 0$, the orbit $h_n$ is palindromic, thus
$$
\mathcal{L}(h_n)-(n+1)\mathcal{L}(\sigma)=2 \sum_{k=0}^{n} [h(s_n(k),s_n(k+1))-h(s(k),s(k+1))].
$$

Recall that $(x_\infty(k)=(s_\infty(k),\varphi_\infty(k)))_{k}$ are the $(s,\varphi)$-coordinates of the points in the homoclinic orbit $h_\infty$.  By Proposition \ref{prop sympt hoc}, 
we know that 
\begin{equation}\label{estimedjjt}
\|x_n(k)-x_\infty(k)\|=O(\lambda^{n-\frac k2}),\qquad  \|x_\infty(k)-x(k)\|=O(\lambda^{\frac k2}).
\end{equation}

 Therefore, by exponential decay of the terms in the following sum, we may define a quantity $\mathcal{L}^\infty=\mathcal{L}^\infty(\sigma)\in \R$ as 
$$
\mathcal{L}^\infty:=\lim_{n\to \infty}(\mathcal{L}(h_n)-(n+1)\mathcal{L}(\sigma))=2 \sum_{k=0}^{+\infty}[h(s_\infty(k),s_\infty(k+1))-h(s(k),s(k+1))].
$$
Then, 
we have
\begin{equation}\label{estim lhnd}
\mathcal{L}(h_n)-(n+1)\mathcal{L}(\sigma)-\mathcal{L}^\infty= -\Sigma_n,
\end{equation}
with $\Sigma_n:=\Sigma_n^1+\Sigma_n^2$, and
\begin{align}
\Sigma_n^1&:= 2\sum_{k=0}^{n}\left[h(s_\infty(k),s_\infty(k+1))-h(s_n(k),s_n(k+1))\right]\label{ligne 1},\\
\Sigma_n^2&:=2\sum_{k=n+1}^{+\infty}\left[h(s_\infty(k),s_\infty(k+1))-h(s(k),s(k+1))\right].\label{ligne 2}
\end{align}

Let $R_1=\mathcal{K}_1^{-1}>0$ and $R_2=\mathcal{K}_2^{-1}>0$ be the respective radii of curvature at the points $A=\gamma(0_1)$ and $B=\gamma(0_2)$ in the orbit $(12)$, and recall that $\ell=\frac{\mathcal{L}(12)}{2}$.
\begin{lemma}\label{estim asymptttot}
For any integer $n \geq 0$, and for $k \in \{0,\dots,n+1\}$, we have
\begin{align*}
\zeta^-(x_n(2k))&=R_1/\ell+O(\max(\lambda^k,\lambda^{n-k})),\\ \zeta^+(x_n(2k))&=R_2/\ell+O(\max(\lambda^k,\lambda^{n-k})),\\
\zeta^-(x_n(2k+1))&=R_2/\ell+O(\max(\lambda^k,\lambda^{n-k})),\\
\zeta^+(x_n(2k+1))&=R_1/\ell+O(\max(\lambda^k,\lambda^{n-k})),
\end{align*}
where $\zeta^-,\zeta^+$ are defined as in \eqref{def zetafff}. 
Besides,
\begin{equation*}
h(s_n(2k),s_n(2k+1))=\ell+O(\max(\lambda^k,\lambda^{n-k})),
\end{equation*}
and
\begin{align*}
\varphi_\infty(2k)&=C_{1,\varphi}^s \xi_\infty \cdot \lambda^k + O(\lambda^{\frac{3k}{2}}),\\
\varphi_\infty(2k+1)&=C_{2,\varphi}^s \xi_\infty \cdot \lambda^k + O(\lambda^{\frac{3k}{2}}),\\
\varphi_n(2k)-\varphi_\infty(2k)&=  C_{1,\varphi}^u\xi_\infty \cdot \lambda^{n+1-k}+O(\lambda^{n-\frac k2}),\\
\varphi_n(2k+1)-\varphi_\infty(2k+1)&=C_{2,\varphi}^u\xi_\infty \cdot \lambda^{n+1-k}+O(\lambda^{n-\frac k2}).
\end{align*}
\end{lemma}

\begin{proof}
The first two sets of estimates  are a direct  consequence of \eqref{estimedjjt}. Recall that by definition, we have  $v_1^*=(C_{1,s}^*,C_{1,\varphi}^*)$ and $v_2^*=(C_{2,s}^*,C_{2,\varphi}^*)$, for $*\in \{s,u\}$. Then, the last four estimates are obtained by projecting the estimates in Proposition \ref{prop sympt hoc} 
on the second  component.
\end{proof}

Recall that $\alpha_1=\frac{\ell}{R_1}+1$ and $\alpha_2=\frac{\ell}{R_2}+1$. We define a symmetric constant $C_{\varphi}^s\in \R_+^*$ as (the equality follows from \eqref{lien bis cohi})
\begin{equation}\label{def constant sym}
C_{\varphi}^s:=\frac{2(C_{1,\varphi}^s)^2}{ (1-\lambda^2)\alpha_1}=\frac{2\lambda^{-1}(C_{2,\varphi}^s)^2}{ (1-\lambda^2)\alpha_2}.
\end{equation}

In the following, we abbreviate
$$
A_1:=\frac{R_1}{\ell}\left(1+\frac{R_1}{\ell}\right),\qquad B:=\frac{R_1}{\ell}\cdot \frac{R_2}{\ell},\qquad A_2:=\frac{R_2}{\ell}\left(1+\frac{R_2}{\ell}\right).
$$

\begin{prop}\label{main proprorop}
We get the following expansions according to the parity of $n \gg 1$:
\begin{align*}
\Sigma_n&=\ell  \xi_\infty^2 C_\varphi^s ((1+\lambda^2)A_1 \alpha_1-(1+\lambda)^2 B+2 \lambda A_2 \alpha_2)\lambda^{n+1}+O(\lambda^{\frac{3n}{2}}),\qquad n \text{ odd},\\
\Sigma_n&=\ell  \xi_\infty^2 C_\varphi^s(2 \lambda A_1 \alpha_1-(1+\lambda)^2 B+(1+ \lambda^2) A_2 \alpha_2)\lambda^{n+1}+O(\lambda^{\frac{3n}{2}}),\qquad n \text{ even}.
\end{align*}
\end{prop}

\begin{proof}
Recall that $\Sigma_n=\Sigma_n^1+\Sigma_n^2$. To estimate $\Sigma_n^1$ and $\Sigma_n^2$, we sum the Taylor expansions given by Lemma \ref{lemma utile} over all the points appearing in the sums. For  $\Sigma_n^1$, we take the reference points in the Taylor expansion to be those of the periodic orbit $h_n$, while for $\Sigma_n^2$, we take the reference points to be those of the orbit $(12)$. We see that in both sums, first order terms vanish:  for  $\Sigma_n^1$, we get
\begin{align*}
&\sum_{k=0}^n R(s_n(k)) \sin(\varphi_n(k)) (\varphi_n(k)-\varphi_\infty(k))\\
&+R(s_n(k+1)) \sin(\varphi_n(k+1)) (\varphi_n(k+1)-\varphi_\infty(k+1))=0.
\end{align*}
Indeed, the above sum is telescopic, and the angles at the first and last points vanish. Alternatively, this follows from the following  facts:
\begin{itemize}
\item for each $n \geq 0$, the  orbit $h_n$ is periodic, hence is a local minimizer of the length functional;
\item for each $n \geq 0$,  $h_n$ has period $2n+2$ and is palindromic, so that we can restrict ourselves to the first $n+1$ iterates;
\item for each $n \geq 0$,  $h_n$ shadows the homoclinic orbit $h_\infty$ for the first $n+1$ steps.
\end{itemize}
We argue similarly for $\Sigma_n^2$, since $(12)$ is also periodic.

Therefore, we just have to consider second order terms in the two sums. In the sum $\Sigma_n^1$, the index $k$ satisfies $k \in \{0,\dots,n\}$, and thus, $\max(\lambda^{\frac{k}{2}},\lambda^{n-\frac{k}{2}})=\lambda^{\frac{k}{2}}$ in the estimates of Lemma \ref{estim asymptttot}.  We obtain
\begin{align*}
\Sigma_n^1&=\sum_{k=0}^{n}
h(s_n(k),s_n(k+1))\left(\zeta^-(x_n(k))\left(1+\zeta^-(x_n(k))\right)\cdot(\varphi_n(k)-\varphi_\infty(k))^2\right.\\
&\left.+\zeta^+(x_n(k)) \left(1+\zeta^+(x_n(k))\right)\cdot(\varphi_n(k+1)-\varphi_\infty(k+1))^2+ 2 \zeta^-(x_n(k))\zeta^+(x_n(k))\cdot\right.\\
&\left.(\varphi_n(k)-\varphi_\infty(k))(\varphi_n(k+1)-\varphi_\infty(k+1))\right)+O(|s_n(k)-s_\infty(k)|^3),
\end{align*}
and, noting that $h(s(k),s(k+1))=\ell$ and $x(k)=(s(k),\varphi(k))=(0,0)$,  for all $k \in \Z$, we also have
\begin{align*}
&\Sigma_n^2=\sum_{k=n+1}^{+\infty}
\ell \left(\zeta^-(x(k))\left(1+\zeta^-(x(k))\right)\cdot \varphi_\infty(k)^2+\zeta^+(x(k)) \left(1+\zeta^+(x(k))\right)\right.\\
&\left.\cdot \varphi_\infty(k+1)^2+ 2 \zeta^-(x(k))\zeta^+(x(k))\cdot \varphi_\infty(k) \varphi_\infty(k+1)\right)+O(|s_\infty(k)|^3).
\end{align*}

$\bullet\ \mathbf{Odd\ case:}$ Let us assume that $n=2m-1$, for some integer $m \geq 1$. In the above expression of $\Sigma_n^1$, we see that we can group terms two by two, except for the indices $k=0$ and $k=n$, that we need to consider separately. By Lemma \ref{estim asymptttot} and the expression of $\Sigma_n^1$, we obtain
\begin{align*}
&\Sigma_n^1=\sum_{k=0}^{m-1}
2\ell\left(A_1 (C_{1,\varphi}^u \xi_\infty)^2 \cdot \lambda^{2(n+1-k)}+A_2 (C_{2,\varphi}^u \xi_\infty)^2 \cdot \lambda^{2(n+1-k)}\right.\\
&\left.+B C_{1,\varphi}^uC_{2,\varphi}^u \xi_\infty^2 (1+\lambda^{-1})\cdot \lambda^{2(n+1-k)} +O(\lambda^{2n-k})\right)\\
&+\ell A_1 (C_{1,\varphi}^u \xi_\infty)^2 \cdot\lambda^{2(n+1-m)}+O(\lambda^{2n-m})\\
&=2\ell (C_{1,\varphi}^s \xi_\infty )^2 \left(
A_1 \left(\frac{C_{1,\varphi}^u}{C_{1,\varphi}^s}\right)^2 +B(1+\lambda^{-1}) \frac{C_{1,\varphi}^u}{C_{1,\varphi}^s}\cdot \frac{C_{2,\varphi}^u}{C_{1,\varphi}^s}+A_2 \left(\frac{C_{2,\varphi}^u}{C_{1,\varphi}^s}\right)^2\right)\sum_{k=0}^{m-1}\lambda^{2(n+1-k)}\\
&+\ell    (C_{1,\varphi}^s\xi_\infty)^2 A_1 \left(\frac{C_{1,\varphi}^u}{C_{1,\varphi}^s}\right)^2 \cdot\lambda^{2m}+O(\lambda^{3m}).
\end{align*}
The additional term above is obtained for the index $k=n$, according to our previous remark; the index $k=0$ does not need to be considered, since it is in the error term, in $O(\lambda^{3m})$.
By the identities obtained in Corollary \ref{neffffff}, we thus get
\begin{align*}
\Sigma_n^1&=\ell (C_{1,\varphi}^s \xi_\infty)^2 \left(
2A_1  -B \frac{(1+\lambda^{-1})^2}{\alpha_1}+2A_2 \frac{\lambda^{-1}\alpha_2}{\alpha_1}\right)\frac{\lambda^2}{1-\lambda^2}\cdot\lambda^{n+1}\\
&+\ell   (C_{1,\varphi}^s \xi_\infty)^2 A_1  \cdot\lambda^{n+1}+O(\lambda^{\frac{3n}{2}})\\
&=\frac{\ell (C_{1,\varphi}^s \xi_\infty)^2}{(1-\lambda^2)\alpha_1}((1+\lambda^2)A_1 \alpha_1-(1+\lambda)^2B+2\lambda A_2 \alpha_2)\lambda^{n+1}+O(\lambda^{\frac{3n}{2}}).
\end{align*}
Let us now estimate the other sum. By Lemma \ref{estim asymptttot} and the  expression of $\Sigma_n^2$, we get
\begin{align*}
&\Sigma_n^2=\sum_{k=m}^{+\infty}
2\ell\left(A_1 (C_{1,\varphi}^s \xi_\infty)^2 \cdot \lambda^{2k}+A_2 (C_{2,\varphi}^s \xi_\infty)^2 \cdot \lambda^{2k}\right.\\
&\left.+B C_{1,\varphi}^s C_{2,\varphi}^s \xi_\infty^2 (1+\lambda)\cdot \lambda^{2k} +O(\lambda^{3k})\right)-\ell A_1 (C_{1,\varphi}^s \xi_\infty)^2 \cdot\lambda^{2m}+O(\lambda^{3m})\\
&=\ell  (C_{1,\varphi}^s \xi_\infty)^2 \left(
2A_1 +2 B(1+\lambda) \frac{C_{2,\varphi}^s}{C_{1,\varphi}^s}+ 2 A_2 \left(\frac{C_{2,\varphi}^s}{C_{1,\varphi}^s}\right)^2\right)\sum_{k=m}^{+\infty}\lambda^{2k}\\
&-\ell    (C_{1,\varphi}^s \xi_\infty)^2 A_1  \cdot\lambda^{2m}+O(\lambda^{3m})\\
&=\frac{\ell (C_{1,\varphi}^s \xi_\infty)^2}{(1-\lambda^2)\alpha_1}((1+\lambda^2)A_1 \alpha_1-(1+\lambda)^2 B+2 \lambda A_2 \alpha_2)\lambda^{n+1}+O(\lambda^{\frac{3n}{2}}).
\end{align*}
Here, again, the additional term comes from the index $k=n+1$ in the sum $\Sigma_n^2$, and we use the identities obtained in Corollary \ref{neffffff} to simplify the above expression. \\

\noindent $\bullet\ \mathbf{Even\ case:}$ Let us now assume that $n=2m$, for some integer $m \geq 0$. In this case, we have an additional bounce, which is also the one  closest to the  orbit $(12)$, and it is on the second obstacle, instead of the first one as previously. We thus get
\begin{align*}
&\Sigma_n^1=
2\ell\left(A_1 (C_{1,\varphi}^u \xi_\infty)^2 +A_2 (C_{2,\varphi}^u \xi_\infty)^2 +B C_{1,\varphi}^uC_{2,\varphi}^u \xi_\infty^2 (1+\lambda^{-1}) \right)\sum_{k=0}^{m} \lambda^{2(n+1-k)}\\
&-\ell  (C_{2,\varphi}^u \xi_\infty)^2A_2 \cdot\lambda^{2(n+1-m)}- 2 \ell B C_{1,\varphi}^uC_{2,\varphi}^u \xi_\infty^2 \lambda^{-1}\cdot \lambda^{2(n+1-m)}+O(\lambda^{2n-m})\\
&=\ell (C_{1,\varphi}^s \xi_\infty )^2 \left(
2A_1  -B\frac{(1+\lambda^{-1})^2}{\alpha_1}+2 A_2 \frac{\lambda^{-1}\alpha_2}{\alpha_1}\right)\frac{\lambda}{1-\lambda^2}\cdot \lambda^{n+1}\\
&-\ell  (C_{1,\varphi}^s \xi_\infty)^2 A_2\frac{\alpha_2}{\lambda \alpha_1} \lambda \cdot\lambda^{n+1}+\ell  (C_{1,\varphi}^s \xi_\infty)^2 B \frac{(1+\lambda^{-1})}{\alpha_1}\lambda^{-1}\cdot \lambda^{n+1}+O(\lambda^{\frac{3n}{2}})\\
&=\frac{\ell (C_{1,\varphi}^s \xi_\infty)^2}{(1-\lambda^2)\alpha_1}(2 \lambda A_1 \alpha_1-(1+\lambda)^2B+(1+\lambda^2) A_2 \alpha_2)\lambda^{n+1}+O(\lambda^{\frac{3n}{2}}).
\end{align*}
The two additional terms that appear in the previous estimates are due to the additional bounce, following the above remark.
Similarly, for $\Sigma_n^2$, we obtain
\begin{align*}
&\Sigma_n^2=\sum_{k=m+1}^{+\infty}
2\ell\left(\left(A_1 (C_{1,\varphi}^s \xi_\infty)^2 +A_2 (C_{2,\varphi}^s \xi_\infty)^2 +B C_{1,\varphi}^s C_{2,\varphi}^s \xi_\infty^2 (1+\lambda)\right) \lambda^{2k} +O(\lambda^{3k})\right)\\
&+\ell A_2 (C_{2,\varphi}^s \xi_\infty)^2 \cdot\lambda^{2m}+2\ell B C_{1,\varphi}^s C_{2,\varphi}^s \xi_\infty^2 \lambda \cdot \lambda^{2m}+O(\lambda^{3m})\\
&=\ell  (C_{1,\varphi}^s \xi_\infty)^2 \left(
2A_1 +2 B(1+\lambda) \frac{C_{2,\varphi}^s}{C_{1,\varphi}^s}+ 2 A_2 \left(\frac{C_{2,\varphi}^s}{C_{1,\varphi}^s}\right)^2\right)\frac{\lambda}{1-\lambda^2} \cdot \lambda^{n+1}\\
&+\ell    (C_{1,\varphi}^s \xi_\infty)^2 A_2\left(\frac{C_{2,\varphi}^s}{C_{1,\varphi}^s}\right)^2\lambda^{-1}  \cdot\lambda^{n+1}+ \ell (C_{1,\varphi}^s \xi_\infty)^2 2B  \frac{C_{2,\varphi}^s}{C_{1,\varphi}^s}  \cdot \lambda^{n+1} +O(\lambda^{\frac{3n}{2}})\\
&=\frac{\ell  (C_{1,\varphi}^s \xi_\infty)^2}{1-\lambda^2} \left(
2\lambda A_1 -2B(1+\lambda) \frac{1+\lambda}{2\alpha_1}+ (\lambda+\lambda^{-1}) A_2 \frac{\lambda \alpha_2}{\alpha_1}\right) \cdot \lambda^{n+1}+O(\lambda^{\frac{3n}{2}})\\
&=\frac{\ell (C_{1,\varphi}^s \xi_\infty)^2}{(1-\lambda^2)\alpha_1}(2 \lambda A_1 \alpha_1-(1+\lambda)^2 B+(1+\lambda^2) A_2 \alpha_2)\lambda^{n+1}+O(\lambda^{\frac{3n}{2}}).
\end{align*}
We conclude the proof by adding up the expansions of $\Sigma_n^1$ and $\Sigma_n^2$ we obtained  for $n$ even and odd.
\end{proof}

\subsection{Marked Length Spectrum and geometry}\label{seubcdoeg}

In this part, we keep the same notations as previously, and we derive some dynamical and geometric consequences from the expansions obtained in Proposition \ref{main proprorop}. Recall that for $j \in \{1,2\}$, we have
$$
\alpha_j=\frac{\ell}{R_j}+1,\qquad  A_j=\frac{R_j}{\ell}\left(1+\frac{R_j}{\ell}\right),\qquad\text{and}\qquad B=\frac{R_1}{\ell}\cdot \frac{R_2}{\ell}.
$$
In particular, we see that
$$
A_j \alpha_j=\left(1+\frac{R_j}{\ell}\right)^2.
$$
By \eqref{estim lhnd} and Proposition \ref{main proprorop}, we deduce that for $n \gg 1$,
\begin{equation}\label{expression difference e}
\mathcal{L}(h_n)-(n+1)\mathcal{L}(\sigma)-\mathcal{L}^\infty= -\ell \xi_\infty^2 C_\varphi^s \cdot  \mathcal{Q}\left(\frac{R_i}{\ell},\frac{R_j}{\ell}\right)\cdot \lambda^{n+1}+O(\lambda^{\frac{3n}{2}})<0,
\end{equation}
where $(i,j)=(1,2)$, when $n$ is odd, and $(i,j)=(2,1)$, when $n$ is even, and where $\mathcal{Q}$ is the quadratic form
defined as follows:
\begin{equation*}
\mathcal{Q}(X,Y):=(1+\lambda^2) (1+X)^2 - (1+\lambda)^2 XY + 2\lambda(1+ Y)^2.
\end{equation*}

\begin{lemma}\label{lemma lyduepno}
We have
$$
\mathrm{LE}(\sigma)=\frac{1}{2}\log(\mu)=-\frac{1}{2}\log(\lambda)=-\lim_{n \to+\infty} \frac{1}{2n} \log\left( -\mathcal{L}(h_n)+(n+1)\mathcal{L}(\sigma)+\mathcal{L}^\infty\right).
$$
In particular, the value of the Lyapunov exponent of $\sigma$ is determined by the Marked Length Spectrum.
\end{lemma}

\begin{proof}
This is a direct consequence of \eqref{expression difference e}. Note that  for each integer $n \geq 0$, the terms which appear in right hand side of the above equality can be recovered from the Marked Length Spectrum of the billiard table $\mathcal{D}$.
\end{proof}

\begin{corollary}\label{lem easy thr dis}
The Marked Length Spectrum determines completely the geometry in the case where the three obstacles are round discs. More generally, the Marked Length Spectrum determines completely the geometry of billiard tables as follows:
\begin{itemize}
\item the table is obtained by removing from the plane  $m\geq 3$ obstacles $\mathcal{O}_1,\dots,\mathcal{O}_m$ whose boundaries are circles;
\item  the non-eclipse condition is satisfied, i.e., the convex hull of any two obstacles is disjoint from the other obstacles.
\end{itemize}
\end{corollary}

\begin{proof}
Let us first consider the case of three obstacles $\mathcal{O}_1,\mathcal{O}_2,\mathcal{O}_3$ with circular boundaries. For each pair $\{i,j\} \in \{1,2,3\}$, $i\neq j$, we can determine the Lyapunov exponent $\mathrm{LE}(ij)=-\frac{1}{2}\log (\lambda(ij))$ according to Lemma \ref{lemma lyduepno}.   As in \eqref{equation recatrace}, it holds
\begin{equation}\label{eq gen ij}
4\alpha_i(ij) \alpha_j(ij)=4 \left(\frac{\mathcal{L}(ij)}{2R_i}+1\right)\left(\frac{\mathcal{L}(ij)}{2R_j}+1\right)=\lambda(ij)+2+\lambda(ij)^{-1},
\end{equation}
where $R_i,R_j$ are the respective radii of curvature at the two points of $(ij)$, and $\mathcal{L}(ij)$ is the length of the periodic orbit $(ij)$.  By the above relation, we can compute the product $\alpha_i(ij)\alpha_j(ij)$ for each pair $i\neq j$, and thus, the value of $R_1,R_2,R_3$: 
indeed, we obtain three equations between the three unknown quantities $R_1,R_2,R_3$ (the lengths $\mathcal{L}(12),\mathcal{L}(13),\mathcal{L}(23)$ are already known).
As a result,   the geometry of the billiard table is entirely determined, up to isometries (compositions of translations, rotations, and reflections). 

In the case of billiard tables with $m\geq 4$ obstacles satisfying the above conditions, we argue as follows. The geometry of the triple $\{\mathcal{O}_1, \mathcal{O}_{2},\mathcal{O}_{3}\}$ is  determined by the above procedure. 
For the obstacle $\obs_4$, if we consider the triple $\{\obs_2,\obs_3,\obs_4\}$, then we have two possible positions $P_1(23),P_2(23)$ for the center of $\obs_4$, which are symmetric of each other by the reflection along the line $D_{23}$ through the centers of $\obs_2$ and $\obs_3$. In particular, $D_{23}$ is the line segment bisector of $[P_1(23),P_2(23)]$. If we consider  the triple $\{\obs_1,\obs_2,\obs_4\}$ instead, then we have two possible positions $P_1(12),P_2(12)$ for the center of $\obs_4$. Now, if $\{P_1(12),P_2(12)\}=\{P_1(23),P_2(23)\}$, it means that the line $D_{12}$ through the centers of $\obs_1$ and $\obs_2$  is also the line segment bisector of $[P_1(23),P_2(23)]$, and thus, it coincides with $D_{23}$. But this is impossible, since by the non-eclipse condition, the centers of $\obs_1,\obs_2,\obs_3$ cannot be aligned.
Therefore, there is exactly one possible position for the center of $\obs_4$ which is coherent with the previous choices.
By repeating this procedure, we  can inductively recover the geometry of the billiard table, up to isometries.
\end{proof}

In the proof of the previous result, we see that in order to recover the
radius of curvature at each bouncing point of period two orbits, it is
sufficient to combine symmetric information between the different
obstacles. Yet, this is possible only because the obstacles are round
discs, and then, for each obstacle, the radius of curvature is the same
at the two points of the obstacle which are periodic and of period
two. For more general obstacles, we have to use some additional
information to recover the radius of curvature. To achieve this, we use
the asymmetries of periodic orbits $(h_n)_{n \geq 0}$, in connection
with the estimates of Proposition \ref{main proprorop}.

\begin{corollary}\label{coro expan}
Let $\mathcal{D}$ be a billiard table formed by three obstacles as above and whose Marked Length Spectrum is known. Then, for each $i\neq j \in \{1,2,3\}$, we can determine the radii  of curvature $R_i=\mathcal{K}_i^{-1}$ and $R_j=\mathcal{K}_j^{-1}$ at the two bouncing points of the orbit labelled by $(ij)$.
\end{corollary}

\begin{proof}
Let us assume that the Marked Length Spectrum of the billiard table is known. We focus on the   orbit $(12)$, and we set  $X_1:=\frac{R_1}{\ell}$,  $X_2:=\frac{R_2}{\ell}$.  \textit{A priori}, the value of $\xi_\infty^2 C_{\varphi}^s$ in \eqref{expression difference e} is not known, hence we consider the ratio
$
\rho=\rho(12):=\frac{\mathcal{Q}(X_1,X_2)}{\mathcal{Q}(X_2,X_1)}
$
to eliminate it. Taking the ratio of the quantities in \eqref{expression difference e} associated to two consecutive integers $n$ and $n+1$, with $n \gg 1$, we deduce that the value of $\rho$ is determined by the Marked Length Spectrum,
which yields the equation
\begin{equation}\label{first eq quad}
\mathcal{Q}(X_1,X_2)-\rho\mathcal{Q}(X_2,X_1)=0.
\end{equation}
We see $\mathcal{Q}(X_1,X_2)-\rho\mathcal{Q}(X_2,X_1)\in \R[X_1][X_2]$ as a quadratic polynomial with coefficients in $\R[X_1]$ ($X_1$ is unknown). Therefore, the knowledge of $\rho$ is enough to deduce the expression of $X_2=X_2(X_1)$ in terms of $X_1$. Besides,  taking $(ij)=(12)$ in \eqref{eq gen ij}, we deduce  that $X_1$ and $X_2$  satisfy the following equation:
\begin{equation}\label{second feko}
4\lambda \left(1+X_1\right)\left(1+X_2\right)-(1+\lambda)^2 X_1 X_2=0.
\end{equation}
We claim that from the last equation,  it is possible to determine the value of $X_1$ and $X_2$, since we already know the expression of $X_2$ in terms of $X_1$.
Indeed, \eqref{first eq quad} can be rewritten as
\begin{equation}\label{deuxieme eq quad}
(1-2 \rho\lambda+\lambda^2)(1+X_1)^2 + (1-\rho)(1+\lambda)^2X_1X_2-\rho(1-2\rho^{-1}\lambda+\lambda^2)(1+X_2)^2=0.
\end{equation}
We have two cases: either $R_1=R_2$, and then, the common value of $R_1$ and $R_2$ is determined by the equation $4\lambda (1+X_1)^2-(1+\lambda)^2 X_1^2=0$ that comes from \eqref{second feko}, or $R_1 \neq R_2$. In the latter case, the two equations \eqref{second feko} and \eqref{deuxieme eq quad} are  independent.
Indeed, \eqref{second feko} is unchanged if we permute $X_1$ and $X_2$, while \eqref{deuxieme eq quad} is if and only if $(1+\rho)(1-\lambda)^2[(1+X_1)^2-(1+X_2)^2]=0$, which cannot happen, since $\rho\geq 0$,  $\lambda<1$ (by hyperbolicity), and $X_1\neq X_2$. Therefore, in the case where $R_1\neq R_2$, we can recover the value of $X_1$ and $X_2$, hence of $R_1, R_2$, by plugging the expression of $X_2=X_2(X_1)$ we obtained    into \eqref{deuxieme eq quad}, and then, solving in $X_1$.
\end{proof}

\textit{A posteriori}, we deduce that the quantities $\xi_\infty,C_{\varphi}^s$ are also  spectrally determined:
\begin{corollary}\label{rema spectral inv}
The value of the parameter $\xi_\infty\in \R^*$ and of the constant $C_{\varphi}^s\in \R_+^*$ introduced respectively in Proposition \ref{prop sympt hoc}  and  in \eqref{def constant sym} are  determined by the Marked Length Spectrum.
\end{corollary}

\begin{proof}
	The quantities $\ell$, $\lambda$, $\alpha_i=\frac{\ell}{R_i}+1$, $\alpha_j=\frac{\ell}{R_j}+1$ and $\mathcal{Q}\big(\frac{R_{i \vphantom{j}}}{\ell},\frac{R_j}{\ell}\big)$ are determined by the Marked Length Spectrum, according to Lemma \ref{lemma lyduepno} and Corollary \ref{coro expan}. By \eqref{matrice df carre}, the matrix of $D_{(0_1,0)}\mathcal{F}^2$ is also determined by the Marked Length Spectrum, as well as the stable space $\mathrm{E}_{\mathcal{F}^2}^s(0_1,0)$ and the value of the constant $C_{1,\varphi}^s$, which by definition is equal to the second component of the unitary vector $v_1^s=L_1^{-1}(1,0) \in \mathrm{E}_{\mathcal{F}^2}^s(0_1,0)$. By \eqref{def constant sym}, we can thus recover the value of the constant $C_{\varphi}^s$. On the other hand, it follows from  \eqref{expression difference e} that the value of $\xi_\infty^2 C_{\varphi}^s$ is determined by  the Marked Length Spectrum. We conclude that the value of $\xi_\infty$ can also be recovered.
\end{proof}

\section{General periodic orbits and Marked Lyapunov Spectrum}\label{section lyaapu}

In the following, we show that part of the previous arguments can be adapted, and that we can reconstruct the Marked Lyapunov Spectrum from the Marked Length Spectrum (see \eqref{marked spectrum} and \eqref{maked lepnov}  for the definitions).

Recall that given a word $\rho=(\rho_1\dots\rho_{r-1}\rho_r)$, we let  $\overline{\rho}$ be the transposed word
$$
\overline{\rho}:=(\rho_r \rho_{r-1}\dots\rho_1).
$$

We consider a periodic orbit of period $p \geq 2$, encoded by an admissible word\footnote{i.e., such that $\sigma_{j}\neq \sigma_{j+1}$ for $j\in \{1,\dots,p-1\}$ and such that $\sigma_1 \neq \sigma_p$.} $\sigma=(\sigma_1\sigma_2\dots\sigma_p) \in \{1,2,3\}^p$.
Let us denote by $(x^{+,\sigma}(k))_{1 \leq k \leq p}$ the points in this orbit, where $x^+(k)=x^{+,\sigma}(k)=(s^+(k),\varphi^+(k))$
is represented by a position and an angle coordinates. As previously, for all $k \in \Z$, we let $x^+(k):=x^+(\overline{k})$, where $\overline{k}:=k\text{ mod } p$, and similarly for $s^+(k),\varphi^+(k)$. We label the points in such a way that $\gamma(s^+(k)) \in \Gamma_{\sigma_{\overline{k}}}$. Moreover, we let $(x^{-,\sigma}(k))_{1 \leq k \leq p}$ be the coordinates of the points in the same orbit, but traversed backwards, so that the new orbit is encoded by the word $\overline{\sigma}$.  Analogously, we let $x^-(k)=x^{-,\sigma}(k)=(s^-(k),\varphi^-(k))$, and we choose the labelling  in such a way that $\gamma(s^-(k)) \in \Gamma_{\sigma_{\overline{1-k}}}$. Equivalently, we have $x^-(k)=\mathcal{I}(x^+(p+1-k))$, for all $k \in \{1,\dots,p\}$.

Let $\tau=(\tau^-,\tau^+)\in \{1,2,3\}^2$ be such that $\tau^+\neq \sigma_{1}$ and $\tau^-\neq \sigma_{p}$.
In the same way as previously, we define a sequence of  periodic orbits $(h_n')_{n \geq 0}$, where $h_n'=h_n'(\sigma,\tau)$ has period $2np+2$ and is labelled by the words
\begin{align*}
h_{n}^+&=h_{n}^+(\sigma,\tau):=(\tau^+\sigma^n\tau^-\overline{\sigma}^n):=(\tau^+\underbrace{\sigma\sigma\dots \sigma}_\text{n\ \text{times}}\tau^-\underbrace{\overline{\sigma\sigma\dots \sigma}}_\text{n\ \text{times}}),\\
h_{n}^-&=h_{n}^-(\sigma,\tau):=(\tau^-\overline{\sigma}^n\tau^+\sigma^n):=(\tau^-\underbrace{\overline{\sigma\sigma\dots \sigma}}_\text{n\ \text{times}}\tau^+\underbrace{\sigma\sigma\dots \sigma}_\text{n\ \text{times}}).
\end{align*}

For each $n \geq 0$, we have two sets of coordinates  $(x_{n}^{\pm,\sigma,\tau}(k))_{-np \leq k \leq np+1}$ for the points
in the periodic orbit $h_{n}'$, with
$x_{n}^{\pm,\sigma,\tau}(k)=(s_{n}^{\pm,\sigma,\tau}(k),\varphi_{n}^{\pm,\sigma,\tau}(k))$.  For ease of notation, in what follows, we will drop
$\sigma,\tau$, which will be considered fixed, and write
$x_{n}^\pm(k)=(s_{n}^\pm(k),\varphi_{n}^\pm(k))$. We order the points in such
a way that
\begin{itemize}
\item $x_{n}^\pm(0)$ is the point associated with the symbol $\tau^\pm$;
\item for $k=1,\dots,np$, the point $x_{n}^+(\pm k)$ is associated with the symbol $\sigma_{\overline{k}}$;
\item for $k=1,\dots,np$, the point $x_{n}^-(\pm k)$ is associated with the symbol $\sigma_{\overline{1-k}}$;
\item $x_{n}^\pm (np+1)$ is the point associated with the symbol $\tau^\mp$.
\end{itemize}
We also extend the notations by setting
$x_{n}^\pm(k):=x_{n}^\pm(k\text{ mod } (2np+2))$ for any integer
$k \in \Z$, and similarly for $s_{n}^\pm$ and $\varphi_{n}^\pm$.

For all $k\in \{1,\dots,p\}$, we let  $\ell^\pm(k):=h(s^\pm(k),s^\pm(k+1))$ be the length of the line segment  between $x^\pm(k)$ and $x^\pm(k+1)$, and we denote by $\mathcal{K}^\pm(k)\in \R$ the curvature at $x^\pm(k)$. We have
\begin{equation}\label{def diffff}
D_{x^\pm(k)}\mathcal{F}=-\frac{1}{\cos(\varphi^\pm(k+1))}
\begin{pmatrix}
\alpha^\pm(k) & \ell^\pm(k)\\
\delta^\pm(k) & \gamma^\pm(k)
\end{pmatrix},
\end{equation}
where
\begin{align*}
\alpha^\pm(k)&:=\ell^\pm(k) \mathcal{K}^\pm(k)+\cos(\varphi^\pm(k)),\\
\gamma^\pm(k)&:=\ell^\pm(k) \mathcal{K}^\pm(k+1)+\cos(\varphi^\pm(k+1)),\\
\delta^\pm(k)&:=\ell^\pm(k) \mathcal{K}^\pm(k) \mathcal{K}^\pm(k+1)+ \mathcal{K}^\pm(k)\cos(\varphi^\pm(k+1))+\mathcal{K}^\pm(k+1)\cos(\varphi^\pm(k)).
\end{align*}
For any $k \in \{1,\dots,p\}$, we have
$$
D_{x^\pm(k)}\mathcal{F}^p=D_{x^\pm(k+p-1)}\mathcal{F} \circ D_{x^\pm(k+p-2)}\mathcal{F}\circ \cdots\circ D_{x^\pm(k)}\mathcal{F} \in \mathrm{SL}(2,\R),
$$
and we denote by $\lambda=\lambda(\sigma)<1$, resp. $\mu=\mu(\sigma) > 1$, the smallest, resp. largest eigenvalue of $D_{x^\pm(k)}\mathcal{F}^p$ (it is independent of $k$).

Let us also define the two orbits associated with the following infinite words:
\begin{align*}
h_{\infty}^+=h_{\infty}^+(\sigma,\tau)&:=(\overline{\sigma}^\infty \tau^+ \sigma^\infty)=(\dots\overline{\sigma\sigma}\tau^+ \sigma\sigma\dots),\\
h_{\infty}^-=h_{\infty}^-(\sigma,\tau)&:=(\sigma^\infty\tau^-\overline{\sigma }^\infty )=(\dots\sigma\sigma\tau^- \overline{\sigma\sigma}\dots).
\end{align*}
We denote by $(x_{\infty}^{\pm,\sigma,\tau}(k))_{k \in \Z}$ the points in the  orbit $h_{\infty}^\pm(\sigma,\tau)$, with $x_{\infty}^{\pm,\sigma,\tau}(k)=(s_{\infty}^{\pm,\sigma,\tau}(k),\varphi_{\infty}^{\pm,\sigma,\tau}(k))$.  Similarly, we abbreviate the coordinates as $x_{\infty}^\pm(k)=(s_{\infty}^\pm(k),\varphi_{\infty}^\pm(k))$, and we order the points in such a way that
\begin{itemize}
\item $x_{\infty}^\pm(0)$ is the point associated with the symbol $\tau^\pm$;
\item for $k\geq 1$, the point $x_{\infty}^+(\pm k)$ is associated with the symbol $\sigma_{\overline{k}}$;
\item for $k\geq 1$, the point $x_{\infty}^-(\pm k)$ is associated with the symbol $\sigma_{\overline{1-k}}$.
\end{itemize}

\begin{lemma}\label{pelfotk}
  For all $k \in \Z$, and for all $n \geq 1$,
  we have
  \begin{align*}
    x_{ n}^\pm(k) &= x_{n}^\pm(-k), & x_{\infty}^\pm(k)&=x_{\infty}^\pm(-k)
  \end{align*}
  and
  \begin{align*}
    \varphi_{n}^\pm(0)=\varphi_{n}^\pm(np+1)=\varphi_{\infty}^\pm(0)=0.
  \end{align*}
\end{lemma}
\begin{proof}
As in Lemma \ref{lemma palind}, this is a consequence of the palindromic symmetries of the orbits $h_{ n}'$ and $h_{\infty}^\pm$.
Indeed, the identities follow from the fact that the respective future and
  past of $h_{n}'=((h_{n}^+)^\infty)=((h_{n}^-)^\infty)$, $h_{\infty}^\pm$ have the same symbolic coding. Therefore,  by expansivity of the dynamics, we have $x_{ n}^\pm(k)=\mathcal{I}(x_{n}^\pm(k))$ and $x_{\infty}^\pm(k)=\mathcal{I}(x_{\infty}^\pm(k))$, which gives the result.
\end{proof}

In particular, this lemma tells us that we can focus either on the
forward or on the backward orbit of the above orbits.

\subsection{Estimates on the parameters}\label{subsection esit parammm}

\begin{lemma}\label{premier lemmmme deux}
There exists two constants $\Lambda_\pm>1$ such that for any integers $n,r \geq 0$, and for any integer $k\in \{0,\dots,np+1\}$, we have
\begin{equation*}
\|x_{n}^\pm(\epsilon k)-x_{n+r}^\pm(\epsilon  k)\|= O(\Lambda_\pm^{-(np+1)+k}),
\end{equation*}
for $\epsilon \in \{+,-\}$.
\end{lemma}

\begin{proof}
The proof is similar to that of Lemma \ref{premier lemmmme}, by looking at the symbolic dynamics, and by the exponential growth of dispersing wave fronts, since the backward and forward orbits of $x_{n}^\pm(0)$ and $x_{n+r}^\pm(0)$ have the same symbolic codings for $np$ steps. For instance, we  have
\begin{align*}
(h_n^+)^\infty&=(\dots\tau^- \underbrace{\overline{\sigma\sigma\dots \sigma}}_\text{n\ \text{times}}\tau^+\underbrace{\sigma\sigma\dots \sigma}_\text{n\ \text{times}}\tau^-\dots),\\
(h_{n+r}^+)^\infty&=(\dots\tau^-\underbrace{\overline{\sigma\sigma\dots \sigma}}_\text{n+r\ \text{times}}\tau^+\underbrace{\sigma\sigma\dots \sigma}_\text{n+r\ \text{times}}\tau^-\dots).
\end{align*}
\end{proof}

In particular, for any $k\in \{0,\dots,\lceil \frac{np}{2}\rceil\}$, we have
\begin{equation*}
\|x_{n}^\pm(\epsilon k)-x_{\infty}^\pm(\epsilon  k)\|= O(\Lambda_\pm^{-\frac{np}{2}}),
\end{equation*}
hence $\lim_{n \to +\infty} x_{ n}^\pm(k)=x_{ \infty}^\pm(k)$. In other terms, the first $\lceil \frac{np}{2}\rceil$ backward and forward iterates of the point $x_n^\pm(0)$, i.e., the points $x_{ n}^\pm(\epsilon k)$, $k\in \{0,\dots,\lceil \frac{np}{2}\rceil\}$, shadow the associated points  of the  orbit $h_{\infty}^\pm$. \\

The point $x^\pm(1)$ is a saddle fixed point of $\mathcal{F}^p$, with eigenvalues $\lambda=\lambda(\sigma)<1$ and $\mu=\mu(\sigma):=\lambda^{-1}>1$.  We denote by
\begin{align*}
\mathrm{E}_{\mathcal{F}^p}^s(x^\pm(1))&:=\{v\in T_{x^\pm(1)} \mathcal{M}: D_{x^\pm(1)} \mathcal{F}^p\cdot v=\lambda v\},\\
\mathrm{E}_{\mathcal{F}^p}^u(x^\pm(1))&:=\{v\in T_{x^\pm(1)} \mathcal{M}: D_{x^\pm(1)} \mathcal{F}^p\cdot v=\lambda^{-1} v\},
\end{align*}
the stable, resp. unstable space of $\mathcal{F}^p$ at $x^\pm(1)$.

Thus, as in Section \ref{section linerasition}, for any $\varepsilon>0$,
there exist a neighborhood $\mathcal{U}^\pm$ of $x^\pm(1)$,  a neighborhood $\mathcal{V}^\pm$ of $(0,0)$, and a $C^{1,\frac 12}$-diffeomorphism $\Psi^\pm\colon \mathcal{U}^\pm \to \mathcal{V}^\pm$, such that
\begin{equation*}\label{fseccc conjuugg deuxiee}
\Psi^\pm \circ \mathcal{F}^p \circ (\Psi^\pm)^{-1} = D_\lambda,\qquad \|\Psi^\pm-L^\pm\|_{C^1} \leq \varepsilon, \qquad \|(\Psi^\pm)^{-1}-(L^\pm)^{-1}\|_{C^1} \leq \varepsilon,
\end{equation*}
for some linear isomorphism $L^\pm=L^\pm(\mathcal{F}) \in \mathrm{SL}(2,\R)$, and
\begin{equation}\label{contole non deuxieeeeee}
(\Psi^\pm)^{-1}(z)-(\Psi^\pm)^{-1}(z')=(L^\pm)^{-1}(z-z')+O(\max(|z|^{\frac 12},|z'|^{\frac 12})|z-z'|),
\end{equation}
where $D_\lambda:=\mathrm{diag}(\lambda,\lambda^{-1})$.

\begin{lemma}\label{coro sympt hoc deuxieeeee}
The orbits $h_{\infty}^\pm$ are homoclinic  to the periodic orbit $\sigma$. Moreover, for $j \in \{1,\dots,p\}$, there exist  $p$ vectors $v^{\pm, s}_{1}=(C_{1,s}^{\pm,s},C_{1,\varphi}^{\pm,s}),\dots,v^{\pm, s}_{p}=(C_{p,s}^{\pm,s},C_{p,\varphi}^{\pm,s})\in \R^2$, with $v^{\pm, s}_{1}=(L^\pm)^{-1} (1,0) \in \mathrm{E}_{\mathcal{F}^{p}}^s(x^\pm(1))$, $\|v^{\pm, s}_{1}\|=1$,  $v^{\pm, s}_{2}:=D_{x^\pm(1)} \mathcal{F} \cdot v^{\pm, s}_{1} \in \mathrm{E}_{\mathcal{F}^{p}}^s(x^\pm(2)), \dots, v^{\pm, s}_{p}:=D_{x^\pm(p-1)} \mathcal{F} \cdot v^{\pm, s}_{p-1} \in \mathrm{E}_{\mathcal{F}^{p}}^s(x^\pm(p))$, such that  for each $ j\in \{1,\dots,p\}$, and $k \gg 1$, we have
\begin{equation*}
x_{\infty}^\pm( j+\epsilon kp)-x^\pm( j)=\lambda^k \xi_{\infty}^\pm\cdot v_{j}^{\pm,s}+O(\lambda^{\frac{3k}{2}}),
\end{equation*}
for $\epsilon \in \{+,-\}$, and for some nonzero real numbers $\xi_{\infty}^+ ,\xi_{\infty}^- \in \R^*$.
\end{lemma}

\begin{proof}
The proof of the above lemma is similar to that of Corollary \ref{coro sympt hoc}, and follows from the estimates associated to the conjugacy $\Psi^\pm$. Let us compare the respective symbolic codings of $h_{ \infty}^\pm$ and $\sigma,\overline{\sigma}$:
\begin{equation*}
    \begin{array}{rrcl} h_{\infty}^+= %
      &\dots\overline{\sigma\sigma\sigma\sigma}&\tau^+&\sigma\sigma\sigma\sigma\dots\\
      \sigma^\infty= %
      &\dots\sigma\sigma\sigma\sigma&\sigma_{p}&\sigma\sigma\sigma\sigma\dots\\
      h_{\infty}^-= %
      &\dots\ \sigma\sigma\sigma\sigma&\tau^-&\overline{\sigma\sigma\sigma\sigma}\dots\\
      \overline{\sigma}^\infty =
      &\dots\ \overline{\sigma\sigma\sigma\sigma}&\sigma_1 &\overline{\sigma\sigma\sigma\sigma}\dots
    \end{array}
  \end{equation*}
By the palindromic symmetries of $h_{\infty}^\pm$ (see Lemma \ref{pelfotk}), we can focus on the forward orbit of $h_{\infty}^{\pm}$. 
Let us start with the point $x_{\infty}^\pm(1)$. Since the  symbolic codings of the forward orbits of 
$x_{\infty}^\pm(1)$ and $x^\pm(1)$ coincide, and by the exponential growth of wave fronts, then as in Lemma \ref{homcolincc}, we can show that the point $x_{\infty}^\pm(1)$ is in the stable  manifold 
of $x^\pm(1)$ for the dynamics of $\mathcal{F}^p$. Therefore, the sequence of points $(x_\infty^\pm(1+kp))_{k \geq 0}$ converges to the saddle fixed point $x^\pm(1)$ in the future or in the past, and we can use the linearization $\Psi^\pm$, in particular \eqref{contole non deuxieeeeee}, to show the above estimates for $j=1$.

To get the estimates for other indices $j \in \{1,\dots,p\}$, we just have to apply the map $\mathcal{F}$. For instance, we have
\begin{align*}
x_{\infty}^\pm(2+kp)-x^\pm(2)&=\mathcal{F}(x_{\infty}^\pm(1+kp))-\mathcal{F}(x^\pm(1))\\
&=D_{x^\pm(1)}\mathcal{F}\cdot (x_{\infty}^\pm(1+kp)-x^\pm(1))+O(\lambda^{2k})\\
&=\lambda^{k}\xi_{\infty}^\pm\, D_{x^\pm(1)}\mathcal{F}\cdot v_1^{+,s}+O(\lambda^{\frac{3k}{2}}).
\end{align*}
\end{proof}

\begin{remark}
Note that now, we have two different constants $\xi_{\infty}^+,\xi_{\infty}^-$, and not just one as before, since we have two different homoclinic orbits $h_{\infty}^+,h_{\infty}^-$. Besides, the orbit $\sigma$ is no longer assumed to be palindromic, hence we cannot easily connect the respective futures of $x^+(1)$ and $x^-(1)$ (the future of $x^-(1)$ coincides with the past of $x^+(p)$) based on this symmetry as we did before, and $\xi_{\infty}^+,\xi_{\infty}^-$ are  \textit{a priori}  unrelated.
\end{remark}

\begin{lemma}\label{main crororor deuxieieieieie}
There exist  $p$ vectors $v^{\pm, u}_{1}=(C_{1,s}^{\pm,u},C_{1,\varphi}^{\pm,u}),\dots,v^{\pm, u}_{p}=(C_{p,s}^{\pm,u},C_{p,\varphi}^{\pm,u})\in \R^2$, with $v^{\pm, u}_{1}  \in \mathrm{E}_{\mathcal{F}^p}^u(x^\pm(1))$,  $v^{\pm, u}_{2}:=D_{x^\pm(1)} \mathcal{F} \cdot v^{\pm, u}_{1} \in \mathrm{E}_{\mathcal{F}^p}^u(x^\pm(2)),\dots, v^{\pm, u}_{p}:=D_{x^\pm(p-1)} \mathcal{F} \cdot v^{\pm, u}_{p-1} \in \mathrm{E}_{\mathcal{F}^p}^u(x^\pm(p))$, such that  for each $ j\in \{1,\dots,p\}$, and $k\in \{0,\dots,n\}$, we have
\begin{equation*}
x_{n}^\pm(j+\epsilon kp)-x_{\infty}^\pm(j+\epsilon kp)=\lambda^{n-k } \xi_{\infty}^\pm\cdot v_{j}^{\pm,s}+O(\lambda^{n-\frac{k}{2}}),
\end{equation*}
for $\epsilon \in \{+,-\}$.
\end{lemma}

\begin{proof}
Again, as in Remark \ref{remarque geomee}, this follows essentially from the exponential growth of wave fronts and from the symbolic dynamics,  by noting that the first $n$ backward and forward iterates  under $\mathcal{F}^p$ of the point of index $1$ in the orbits $h_{ n}^\pm$ and $h_{ \infty}^\pm$ have the same  symbolic codings:
\begin{equation*}
\begin{array}{rccccll}
h_{\infty}^+= &\dots \sigma & \overline{\sigma\sigma}\dots\overline{\sigma\sigma}&\tau^+&\sigma\sigma\dots \sigma\sigma & \sigma \dots\\
h_{n}^+= &\dots \tau^- & \overline{\sigma\sigma}\dots\overline{\sigma\sigma}&\tau^+&\sigma\sigma\dots \sigma\sigma &\tau^-\dots\\
h_{n}^-= &\dots \tau^+ & \sigma\sigma\dots \sigma\sigma&\tau^-&\overline{\sigma\sigma}\dots\overline{\sigma\sigma} &\tau^+\dots\\
h_{\infty}^-= &\dots  \sigma & \sigma\sigma\dots \sigma\sigma&\tau^-
&\overline{\sigma\sigma}\dots\overline{\sigma\sigma} & \sigma \dots
\end{array}
\end{equation*}
In particular, the norm of the difference between the points $x_{n}^\pm(1+\epsilon kp)$ and $x_{\infty}^\pm(1+\epsilon kp)$ is very small when $k$ is small, and increases with $k$. Moreover, for $k$ sufficiently large, the points are in a neighborhood of the periodic orbit $\sigma$, and then, the expansion is at a rate close to the largest eigenvalue $\mu=\lambda^{-1}>1$ of $D_{x^\pm(1)}\mathcal{F}^p$.

More formally, we argue as in the proof  of Corollary \ref{main crororor}.  Indeed, we have an analogue of Lemma \ref{lemma eaxxu}, noting that the orbits $h_{ n}^\pm$ and $h_{ \infty}^\pm$ are palindromic with respect to the point of index $0$ (see Lemma \ref{pelfotk}). The fact that we can take $\epsilon \in \{+,-\}$ in the above estimates also follows from this symmetry; in the following, we assume that $\epsilon=+$.  In particular, by a similar argument as in Lemma \ref{lemma eaxxu}, it is possible to obtain precise estimates on the parameters after the change of coordinates given by $\Psi^\pm$, due to the palindromic symmetry. Indeed, by Lemma \ref{premier lemmmme deux}, we know that for $n$ large, and for $k\in \{0,\dots,n\}$, the points $x_{n}^\pm(1+ kp)$ are in a neighborhood of the separatrix (they are close to $x_{\infty}^\pm(1+  kp)$), where we can linearize the dynamics thanks to the conjugacy map $\Psi^\pm$. Then, due to the palindromic symmetry, we have two ways to connect the points $\Psi^\pm(x_{ n}^\pm(1+  kp))$ and $\Psi^\pm(x_{n}^\pm(1- kp))$, by iterating the map $D_\lambda$ (the points $x_{ n}^\pm(1+  kp)$ and $x_{n}^\pm(1- kp)$  are in the same orbit of $\mathcal{F}^p$), or by the gluing map $\Psi \circ \mathcal{I} \circ \Psi^{-1}$, where $\mathcal{I}\colon (s,\varphi) \mapsto (s,-\varphi)$. Since these two ways have to match, we obtain an analogue of \eqref{rel bizar} for the coordinates of $\Psi(x_{n}^\pm(1+ kp))$. We also have $\Psi(x_{\infty}^\pm(1+  kp))=\xi_{\infty}^\pm (\lambda^k,0)+ O(\lambda^{\frac{3k}{2}})$, and then, by \eqref{contole non deuxieeeeee}, we get the  estimates of Lemma \ref{main crororor deuxieieieieie} for $j=1$.

To show the result for other indices $j \in \{1,\dots,p\}$, as in the proof of Lemma \ref{coro sympt hoc deuxieeeee}, we just have to apply the dynamics. Indeed,  for $j=1$, the inital error is close to a  small vector in the unstable space of $\mathcal{F}^p$ at $x^\pm(1)$, hence its  image by the differential of $\mathcal{F}$ is close to a vector in the unstable space at $x^\pm(2)$, etc.
\end{proof}

\subsection{Marked Lyapunov Spectrum}

This last part is dedicated to the proof of Theorem \ref{main prop per}: given a periodic orbit of period $p \geq 1$ encoded by a word $\sigma \in \{1,2,3\}^p$, and $\tau=(\tau^-,\tau^+) \in \{1,2,3\}^2$ as above, we apply the  results of the previous part to show that the lengths of the periodic orbits $(h_{n}'(\sigma,\tau))_{n \geq 0}$ can be combined in order to recover the Lyapunov exponent of the periodic orbit $\sigma$. The strategy is similar to the proof given in Subsection \ref{section estimmmm}, and we refer to this part for more details. As we have seen there, two different cases occur: as in Proposition \ref{main proprorop}, we have two (\textit{a priori} different) constants $C_{\mathrm{even}},C_{\mathrm{odd}}>0$ in the estimates, depending on the parity of the number $n\geq 0$ of repetitions of $\sigma$ in the words $h_{n}^\pm=h_n^\pm(\sigma,\tau)$.

Recall that for $n \geq 0$, the period of $h_{n}'=h_n'(\sigma,\tau)$ is equal to $2np+2$.
In the following, we assume that the period $p$ of $\sigma$ is even, i.e., $p=2q$, for some integer $q \geq 1$, and we restrict ourselves to the case of words $h_n^\pm$, for $n$ even, i.e., $n=2m$, for some integer $m \geq 0$. In particular, the period of $h_n'$ is equal to $4nq+2=4mp+2$. 
Similar computations can also be done when $p$ is odd and/or $n$ is odd. In the following, we split the words associated to $h_n'$  into two symmetric intervals of $2nq+1=2mp+1$ points centered respectively at $\tau^+$ and $\tau^-$:
\begin{equation*}
h_n' \quad \leftrightarrow \quad (\underbrace{\underbrace{\overline{\sigma\sigma\dots \sigma}}_\text{m\ \text{times}}\tau^+ \underbrace{\sigma\sigma\dots \sigma}_\text{m\ \text{times}}}_{2nq+1=2mp+1}\vert\underbrace{\underbrace{\sigma\sigma\dots \sigma}_\text{m\ \text{times}}\tau^- \underbrace{\overline{\sigma\sigma\dots \sigma}}_\text{m\ \text{times}}}_{2nq+1=2mp+1}).
\end{equation*}
The two above words are symmetric with respect to $\tau^\pm$:  we use the coordinates $x_{n}^+$ introduced above for the first one, and the coordinates $x_n^-$ for the second one:
\begin{align*}
&\mathcal{L}(h_{n}')-2n\mathcal{L}(\sigma)= \sum_{k=-nq}^{nq} [h(s_{n}^+(k),s_{n}^+(k+1))-h(s^+(k),s^+(k+1))]\\
&+\sum_{k=-nq}^{nq}  [h(s_{n}^-(k),s_{n}^-(k+1))-h(s^-(k),s^-(k+1))]+2h(s^+(p),s^+(1))\\
&=2\sum_{k=0}^{nq-1} \Delta_n(k)+\Delta_n(mp)+2h(s^+(p),s^+(1)),
\end{align*}
where
\begin{align*}
\Delta_n(k):=&h(s_{n}^+(k),s_{n}^+(k+1))-h(s^+(k),s^+(k+1))\\
+&h(s_{n}^-(k),s_{n}^-(k+1))-h(s^-(k),s^-(k+1)).
\end{align*}
Recall that $n=2m$, with $m \geq 0$.
We deduce that
\begin{equation*}
\sum_{k=0}^{nq-1} \Delta_n(k)=\sum_{k=0}^{mp-1}\Delta_n(k)= \sum_{k=0}^{m-1}\sum_{j=1}^p \Delta_n(j+kp)+(\Delta_n(0)-\Delta_n(mp)),
\end{equation*}
and
\begin{equation*}
\mathcal{L}(h_{n}')-2n\mathcal{L}(\sigma)=2\sum_{k=0}^{m-1}\sum_{j=1}^p \Delta_n(j+kp)+2\Delta_n(0)-\Delta_n(mp)+2h(s^+(p),s^+(1))
\end{equation*}
Moreover, by  Lemma \ref{coro sympt hoc deuxieeeee} and Lemma \ref{main crororor deuxieieieieie},  for each $k\in \{0,\dots, m-1\}$, we have
$$
\sum_{j=1}^p \Delta_n(j+kp)=\sum_{j=1}^p \Delta_\infty(j+kp)+ O(\lambda^{n-k})=O(\lambda^{k}),
$$
with
\begin{align*}
\Delta_\infty(k):=&h(s_{\infty}^+(k),s_{\infty}^+(k+1))-h(s^+(k),s^+(k+1))\\
+&h(s_{\infty}^-(k),s_{\infty}^-(k+1))-h(s^-(k),s^-(k+1)),
\end{align*}
so that
\begin{align*}
&\lim_{m \to +\infty}2\sum_{k=0}^{m-1}\sum_{j=1}^p \Delta_n(j+kp)+2\Delta_n(0)-\Delta_n(mp)\\
&= 2\lim_{m \to +\infty}\sum_{k=0}^{m-1}\sum_{j=1}^p \Delta_n(j+kp)+2 \Delta_n(0)=2\sum_{k=0}^{+\infty} \sum_{j=1}^p  \Delta_\infty(j+kp)+2 \Delta_\infty(0).
\end{align*}
We define $\mathcal{L}^\infty=\mathcal{L}^\infty(\sigma,\tau)\in \R$ as
$$
\mathcal{L}^\infty:=2\sum_{k=0}^{+\infty} \sum_{j=1}^p  \Delta_\infty(j+kp)+2\Delta_\infty(0)+2h(s^+(p),s^+(1)).
$$
We deduce from what precedes that
\begin{align*}
\mathcal{L}(h_{n}')-2n \mathcal{L}(\sigma)-\mathcal{L}^\infty
&=2\sum_{k=0}^{m-1}\sum_{j=1}^p \Delta_n'(j+kp)+2\Delta_n'(0)-\Delta_n'(mp)\\
&-2\sum_{k=m}^{+\infty}\sum_{j=1}^p \Delta_\infty(j+kp),
\end{align*}
with
\begin{align*}
\Delta_n'(k):=\Delta_n(k)-\Delta_\infty(k)=&h(s_{n}^+(k),s_{n}^+(k+1))-h(s_{\infty}^+(k),s_{\infty}^+(k+1))\\
+&h(s_{n}^-(k),s_{n}^-(k+1))-h(s_{\infty}^-(k),s_{\infty}^-(k+1)).
\end{align*}

Then, as in the proof of Proposition \ref{main proprorop}, we add up  the expansions of $\Delta_n'(k)$ and $\Delta_\infty(k)$ given by Lemma \ref{lemma utile}. Again, first order terms vanish, due to the periodicity of the orbits $h_{n}'$, $n \geq 0$. Thus, we only need to consider second order terms. The contribution of the term $\Delta_{n}'(0)$ can be neglected, since by Lemma \ref{main crororor deuxieieieieie}, it is of order at most  $O(\lambda^{2n})$, which is in the error term of the expansions that follow.
By Lemma \ref{lemma utile}, we thus get
\begin{equation}\label{dernier extime}
\mathcal{L}(h_{n}')-2n \mathcal{L}(\sigma)-\mathcal{L}^\infty=-(\Sigma_{n}^{1,+} +\Sigma_{n}^{1,-})-\Delta_n'(mp)-(\Sigma_{n}^{2,+} +\Sigma_{n}^{2,-})+O(\lambda^{2n}),
\end{equation}
where
\begin{align*}
\Sigma_{n}^{1,\pm}&:=\sum_{k=1}^{mp}
h(s_{n}^\pm(k),s_{n}^\pm(k+1))\left(\zeta^-(x_{n}^\pm(k))\left(1+\zeta^-(x_{n}^\pm(k))\right)\cdot(\varphi_{n}^\pm(k)-\varphi_{\infty}^\pm(k))^2\right.\\
&\left.+\zeta^+(x_{n}^\pm(k)) \left(1+\zeta^+(x_{n}^\pm(k))\right)\cdot(\varphi_{n}^\pm(k+1)-\varphi_{\infty}^\pm(k+1))^2+ 2 \zeta^-(x_{n}^\pm(k))\right.\\
&\left.\cdot\, \zeta^+(x_{n}^\pm(k))\cdot(\varphi_{n}^\pm(k)-\varphi_{\infty}^\pm(k))(\varphi_{n}^\pm(k+1)-\varphi_{\infty}^\pm(k+1))\right)+O(|s_{n}^\pm(k)-s_{\infty}^\pm(k)|^3),
\end{align*}
and
\begin{align*}
\Sigma_{n}^{2,\pm}&:=\sum_{k=mp+1}^{+\infty}
h(s^\pm(k),s^\pm(k+1))\left(\zeta^-(x^\pm(k))\left(1+\zeta^-(x^\pm(k))\right)\cdot(\varphi^\pm(k)-\varphi_{\infty}^\pm(k))^2\right.\\
&\left.+\zeta^+(x^\pm(k)) \left(1+\zeta^+(x^\pm(k))\right)\cdot(\varphi^\pm(k+1)-\varphi_{\infty}^\pm(k+1))^2+ 2 \zeta^-(x^\pm(k))\right.\\
&\left.\cdot\, \zeta^+(x^\pm(k))\cdot(\varphi^\pm(k)-\varphi_{\infty}^\pm(k))(\varphi^\pm(k+1)-\varphi_{\infty}^\pm(k+1))\right)+O(|s^\pm(k)-s_{\infty}^\pm(k)|^3).
\end{align*}

By Lemma  \ref{main crororor deuxieieieieie} (resp. Lemma \ref{coro sympt hoc deuxieeeee}), the terms $(\varphi_{n}^\pm(k)-\varphi_{\infty}^\pm(k))^2$ (resp. $(\varphi^\pm(k)-\varphi_{\infty}^\pm(k))^2$) in the above expression of $\Sigma_n^{1,\pm}$ (resp. $\Sigma_n^{2,\pm}$) are maximal close to (resp. far from) the periodic orbit encoded by $\sigma$, i.e., for large indices $1 \leq k \leq mp$ (resp. small indices $k \geq mp+1$). Indeed, for each $j \in \{1,\dots,p\}$, each $k \in \{0,\dots,m\}$ and each $k' \in \{m,\dots,+\infty\}$, we have
\begin{align*}
(\varphi_{n}^\pm(j+kp)-\varphi_{\infty}^\pm(j+kp))^2&=(\xi_{\infty}^\pm C_{j,\varphi}^{\pm,u})^2 \lambda^{2(n-k)}+O(\lambda^{3m}),\\
 (\varphi^\pm(j+k'p)-\varphi_{\infty}^\pm(j+k'p))^2&=(\xi_{\infty}^\pm C_{j,\varphi}^{\pm,s})^2 \lambda^{2k'}+O(\lambda^{3m}).
\end{align*}
Another discrepancy comes from the fact that we replace $h(s_{n}^\pm(k),s_{n}^\pm(k+1))$ with $h(s^\pm(k),s^\pm(k+1))$ while estimating $\Sigma_n^{1,\pm}$, but it is in the error term.
Then, by Lemma \ref{coro sympt hoc deuxieeeee} and  Lemma \ref{main crororor deuxieieieieie}, the first term of the sum $\Sigma_{n}^{1,\pm}$ becomes
\begin{align*}
&\sum_{k=1}^{mp}
h(s_{n}^\pm(k),s_{n}^\pm(k+1))\left(\zeta^-(x_{n}^\pm(k))\left(1+\zeta^-(x_{n}^\pm(k))\right)\cdot(\varphi_{n}^\pm(k)-\varphi_{\infty}^\pm(k))^2\right)\\
&=\sum_{j=1}^p\sum_{k=0}^{m-1}
h(s_{n}^\pm(j+kp),s_{n}^\pm(j+kp+1))\cdot \zeta^-(x_{n}^\pm(j+kp))\\
&\cdot\left(1+\zeta^-(x_{n}^\pm(j+kp))\right)\cdot(\varphi_{n}^\pm(j+kp)-\varphi_{\infty}^\pm(j+kp))^2\\
&=\sum_{j=1}^p
h(s^\pm(j),s^\pm(j+1))\cdot \zeta^-(x^\pm(j))\\
&\cdot \big(1+\zeta^-(x^\pm(j))\big)\sum_{k=0}^{m-1}(\varphi_{n}^\pm(j+kp)-\varphi_{\infty}^\pm(j+kp))^2 +O(\lambda^{3m})\\
&=\frac{(\xi_\infty^\pm)^2\lambda^2}{1-\lambda^2}\left( \sum_{j=1}^p
\ell^\pm(j)\cdot \zeta^-(x^\pm(j))\big(1+\zeta^-(x^\pm(j))\big)\cdot(C_{j,\varphi}^{\pm,u})^2\right)  \lambda^n+O(\lambda^{\frac{3n}{2}}).
\end{align*}
We argue similarly for the other two terms in the expressions of $\Sigma_n^{1,\pm}$,  $\Sigma_n^{2,\pm}$ given above, and for the additional term $\Delta_n'(mp)$ in \eqref{dernier extime}, based on the estimates given by Lemma \ref{coro sympt hoc deuxieeeee} and  Lemma \ref{main crororor deuxieieieieie}.

Therefore, as in Proposition  \ref{main proprorop} and in \eqref{expression difference e}, we conclude that for some constant $C_{\mathrm{even}}=C_{\mathrm{even}}(\sigma,\tau)>0$,  we have
\begin{equation*}
\mathcal{L}(h_{n}')-2n \mathcal{L}(\sigma)-\mathcal{L}^\infty=-C_{\mathrm{even}}\cdot \lambda^{n}+O(\lambda^{\frac{3n}{2}}),
\end{equation*}
for each even integer $n=2m$, $m \geq 0$.

In the same way, we can get an analogous estimate in the case of odd integers $n \geq 0$, for some  constant $C_{\mathrm{odd}}>0$ which is \textit{a priori} different from $C_{\mathrm{even}}$. Besides, the case where the length $p$ of the word $\sigma$ is odd is handled similarly.

In particular, from the previous estimate, we see that it is possible to deduce from the Marked Length Spectrum the value of the Lyapunov exponent of the orbit labelled by $\sigma$, according to the formula:
$$
\mathrm{LE}(\sigma)=-\frac{1}{p}\log(\lambda)=-\frac{1}{p}\lim_{n \to+\infty} \frac{1}{n} \log\left( -\mathcal{L}(h_{n}')+2n\mathcal{L}(\sigma)+\mathcal{L}^\infty\right).
$$

\begin{remark}
The reason why we obtain two different constants $C_{\mathrm{even}}$ and  $C_{\mathrm{odd}}$  can be explained as follows. As we have seen for instance for $\Sigma_n^{1,\pm}$, the above estimates are obtained by considering geometric sums whose summands are maximal close to the periodic orbit $\sigma$. By expansiveness of the dynamics, this is achieved for points associated with symbols marked with an arrow in the following symbolic expressions of  $h_n^+=h_n^+(\sigma,\tau)$.  In particular, those symbols are  either ``in the middle'' or ``at the beginning and the end'' of the word $\sigma$, depending on the parity of $n \geq 0$, in such a way that for odd integers $n$, we have an extra term to take care of in the estimates:
\begin{align*}
h_{2m}^+&=(\tau^+\overbrace{\sigma\dots \sigma}^\text{m\ \text{times}}\underset{\substack{\uparrow}}{}  \overbrace{\sigma\dots \sigma}^\text{m\ \text{times}}\tau^-\overbrace{\overline{\sigma\dots \sigma}}^\text{m\ \text{times}}\underset{\substack{\uparrow}}{}  \overbrace{\overline{\sigma\dots \sigma}}^\text{m\ \text{times}}),&n=2m\ \text{even},\\
h_{2m+1}^+&=(\tau^+\overbrace{\sigma\dots \sigma}^\text{m\ \text{times}}\underset{\substack{\uparrow}}{\sigma}  \overbrace{\sigma\dots \sigma}^\text{m\ \text{times}}\tau^-\overbrace{\overline{\sigma\dots \sigma}}^\text{m\ \text{times}}\underset{\substack{\uparrow}}{\sigma}  \overbrace{\overline{\sigma\dots \sigma}}^\text{m\ \text{times}}),&n=2m+1\ \text{odd}.
\end{align*}
\end{remark}


\section*{Acknowledgments}
The authors wish to thank the hospitality of the ETH Institute for Theoretical Studies Z\"urich 
and the support of Dr. Max Rössler, the Walter Haefner Foundation and the ETH Zurich Foundation,
as well as the Banff International Research Station -- where part of this work was carried over.  
The authors are also indebted to the anonymous referees, to L.  Stoyanov and M. Zworski 
for their most useful comments and suggestions. M.L. is grateful to L. Backes, A. Brown, 
S. Crovisier, F. Rodriguez-Hertz, D. Obata, A. Wilkinson and D. Xu for useful
conversations during visits at the Pennsylvania State University, the
University of Chicago, and the Universit\'e Paris-Sud.

\end{document}